\documentclass[a4paper,12pt]{article}
\usepackage{color}
\usepackage{amsmath,amsfonts,amssymb,amsthm,amscd}
\usepackage{hyperref}
\newcommand\blfootnote[1]{%
  \begingroup
  \renewcommand\thefootnote{}\footnote{#1}%
  \addtocounter{footnote}{-1}%
  \endgroup
}

\font\cmssl=cmss10 at 12 pt

\usepackage{slashed}
\usepackage[normalem]{ulem}
\newcommand{\di}{\slashed{d}}
\newcommand{\g}{\gamma}

\newtheorem{thm}{Theorem}
\newtheorem{lem}[thm]{Lemma}
\newtheorem{prop}[thm]{Proposition}
\newtheorem{defn}[thm]{Definition}
\newtheorem{cor}[thm]{Corollary}
\newtheorem{rem}[thm]{Remark}
\newtheorem{notation}[thm]{Notation}
\newtheorem{exa}[thm]{Example}

\title{Generalized connections, spinors,  and integrability of generalized structures on Courant algebroids}

\date{\today}

\author{Vicente Cort\'es and Liana David}

\begin{document}

\maketitle

\begin{abstract}
\noindent
We present a characterization, in terms of torsion-free generalized connections,
for the integrability of
various generalized structures  (generalized almost complex structures,  generalized almost
hypercomplex structures, generalized almost Hermitian structures and generalized almost hyper-Hermitian structures)
defined on  Courant algebroids. 
We develop a new, self-contained,  approach for the theory of Dirac generating operators
on  regular Courant algebroids  with scalar product of neutral signature.    As an application we provide a criterion for the integrability of  generalized almost Hermitian structures 
$(G, \mathcal J)$ and generalized almost hyper-Hermitian structures 
$(G, \mathcal J_{1}, \mathcal J_{2}, \mathcal J_{3})$  
defined on a regular Courant algebroid $E$ with scalar product of neutral signature  in terms of canonically
defined differential operators on spinor bundles associated to $E_{\pm}$ (the subbundles of $E$ determined
by the generalized metric $G$).  

\end{abstract}
\newpage

\tableofcontents
\section{Introduction}

Generalized complex geometry is a well established field in present mathematics. It unifies complex 
and symplectic geometry and represents an important direction of 
current research in differential geometry and theoretical physics. In generalized geometry, 
the role of the tangent bundle $TM$ of a manifold $M$ is played by the generalized tangent bundle $\mathbb{T}M := TM\oplus T^{*}M$,
or, more generally, by a Courant algebroid. Many classical objects from differential geometry (including almost complex, almost Hermitian,   almost hypercomplex and almost hyper-Hermitian structures) 
were defined and studied by 
several authors  in this more general setting, see e.g.\ \cite{gualtieri-annals} (also \cite{b-c-g,g-red,g-c}).  
\

While the  integrability of  such  generalized structures is defined and understood in terms of the Dorfmann  bracket of the Courant algebroid,
it seems that characterizations   in terms of torsion-free generalized connections (analogous to the standard 
characterizations of integrability from classical geometry) are missing.\

 In the first part of the present  paper  we  fill this gap by answering this natural question.
We shall do this by a careful analysis of the space of generalized connections which are `adapted', i.e.\ preserve, 
a given generalized structure $Q$ on a Courant algebroid $E$. In analogy with the classical case, we introduce the notion of intrinsic torsion $t_{Q}$ of $Q$,  see Definition \ref{intrinsic_torsionDef}. This  will play
an important role in our treatment, when $Q$ is a generalized almost complex structure or a generalized almost hypercomplex structure. 
We compute the intrinsic torsion for these structures and  we prove that their  integrability is equivalent to the existence of a torsion-free adapted generalized connection, see Theorem \ref{complex:thm} and Theorem \ref{hyper:thm}. In the same framework, we prove that a generalized almost Hermitian 
structure on a Courant algebroid  is integrable  if and only if it admits a torsion-free adapted generalized connection, see Theorem \ref{integr-kahler}.
A similar integrability criterion for generalized almost hyper-Hermitian structures is proved  in 
Theorem \ref{integr-hk}.
Our treatment shows  that  in generalized geometry  the torsion-free condition on a generalized connection $D$
adapted to a generalized structure $Q$  does in general no longer determine $D$ uniquely, even if the classical structure generalized by $Q$ has a unique torsion-free adapted connection. This was noticed  already in \cite{rubio,garcia}  for  generalized Riemannian metrics: 
for a given generalized  Riemannian metric $G$,  there is an entire family of generalized connections which are torsion-free and preserve $G$.
We will show that the same holds  for the structures considered in this paper, for instance,  for generalized hypercomplex structures:  the Obata connection   \cite{Obata,bonan} (see also \cite{gauduchon}) of a  hypercomplex structure 
is  replaced in generalized geometry by  an entire family  of generalized connections.\

In more general terms, the theory developed here allows to decide whether a given generalized  structure $Q$ on a 
Courant algebroid admits an adapted  generalized connection with prescribed torsion and to describe the space of all such generalized connections, see Proposition \ref{connProp}.
A torsion-free generalized connection adapted to $Q$, for instance, exists if and only if the 
intrinsic torsion of $Q$ vanishes and is unique if and only if the generalized first 
prolongation $\mathfrak{h}^{\langle 1\rangle}$ of the Lie algebra $\mathfrak{h}$ of the structure group of $Q$  is zero. 
As a further  application of the generalized first prolongation we present an alternative proof for the uniqueness of the canonical connection
of a Born structure  defined  in \cite{FRS}  (see Section \ref{section-born}).

In the second part of the paper 
we present a self-contained treatment of the canonical Dirac generating operator
on a regular Courant algebroid with scalar product of neutral signature.   A Courant algebroid is regular if its anchor
has constant rank. Regular Courant algebroids form an important 
class of Courant algebroids,  which was studied systematically in  \cite{chen}. 
A Dirac generating operator is a first order odd differential operator on a suitable irreducible  spinor bundle which encodes the anchor and the Dorfman bracket of the Courant algebroid.  We recover the crucial  result of  \cite{AX}  (see also 
\cite{GMX} and \cite{coimbra}, \cite{garcia})
namely that any regular Courant algebroid $E$  with scalar product of neutral signature admits a canonical Dirac generating
operator $\di$ and we express $\di$  in terms of a dissection of $E$. Such an expression
for the canonical Dirac generating operator  allows a better understanding of 
the relation between $\di$ and the structure of the regular Courant algebroid.\

The third part of the paper is devoted to applications.  We assume that $E$ is a regular 
Courant algebroid with scalar product of neutral signature. 
Owing to the failure of uniqueness of generalized connections adapted to various generalized structures $Q$
on  $E$ 
 it is natural to search for characterizations of the  integrability of  $Q$ using   the canonical Dirac generating operator
of $E$. This was done in \cite{AX} when $Q =\mathcal J$ is a generalized almost complex structure.  
Here we present   a criterion for a generalized almost Hermitian  structure $(G, \mathcal J )$  on  $E$  to be generalized K\"ahler 
in terms of two canonical   Dirac  operators
and the pure spinors associated to $\mathcal J\vert_{E_{\pm}}$
(where $E_{\pm}$ are the subbundles of $E$ determined by the generalized metric $G$), see Theorem \ref{spinors-gk}. 
Our arguments  use, besides the theory of Dirac generating operators, 
the results on generalized connections adapted to a generalized almost Hermitian structure, developed in the first part of the paper.
A similar spinorial characterization for generalized hyper-K\"ahler structures on regular Courant algebroids 
with scalar product of neutral signature 
 is obtained in  Corollary \ref{hyperKCor}.\

In the appendix, intended for completeness of our exposition, we briefly recall basic facts we need on the theory of 
$\mathbb{Z}_{2}$-graded algebras and  the integrability criterion in terms of pure spinors for generalized almost complex structures on regular Courant algebroids with scalar product of neutral signature, mentioned above.\   

In Parts II and III of the paper the assumption that the scalar product of the Courant algebroid  $E$  has neutral
signature plays a crucial role. Then  the spinor bundles $S$ on which the Dirac generating operators act 
are irreducible $\mathbb{Z}_{2}$-graded  $\mathrm{Cl}(E)$-bundles, with the essential property 
that any $\mathrm{Cl}(E)$-bundle morphism $f: S \rightarrow S$ is a multiple of the identity. (More generally, considering  
signatures $(t,s )$ with $t-s \equiv 0$ or $2 \pmod{8}$, would ensure the same property, where $t$ stands for the index and the Clifford relation is $v^2 = \langle v,v\rangle 1$.) It would be interesting to 
extend the theory of Dirac generating operators  to Courant algebroids of other  signatures.

\vspace{1cm}

\begin{center}
\part[First Part: Generalized connections and integrability]{}
\end{center}

\section{Preliminary material}\label{preliminary}

We start by reviewing the basic definitions we need on Courant algebroids, generalized connections,  their torsion
(see  \cite{garcia})   and the definition of 
the generalized structures we shall consider in this paper. 
We  prove  a  basic property of the  Nijenhuis tensor of a generalized
almost complex structure 
(see Lemma \ref{nj-3form}), 
we  introduce the notion of  intrinsic torsion of a generalized structure $Q$   on a Courant algebroid and we describe
the space of generalized connections adapted to $Q$,  with prescribed  torsion, as an affine space modeled on the space of sections of a 
vector bundle (see Proposition \ref{connProp}). The fibers of this  vector bundle are isomorphic to the generalized first 
prolongation $\mathfrak{h}^{\langle 1\rangle}$ of the structure group $H$ of $Q$, a notion which will  be  defined  for any   
Lie subgroup $H\subset \mathrm{O}(k,\ell )$.

\subsection{Courant algebroids}

\begin{defn}\label{def-courant}  A {\cmssl Courant algebroid} on a manifold $M$ is a vector bundle $E\rightarrow M$
equipped with a non-degenerate symmetric bilinear form  
$\langle \cdot , \cdot \rangle \in \Gamma(\mathrm{Sym}^2(E^*))$  (called the {\cmssl scalar product}), a bilinear  operation  
$[\cdot , \cdot ]$  (called the {\cmssl Dorfman bracket}) on the space of smooth  sections $\Gamma (E)$ of $E$ and a homomorphism of vector bundles $\pi : E \rightarrow TM$ (called the {\cmssl anchor}), such that the following conditions are satisfied:
for all sections $u, v, w\in \Gamma (E)$,\

C1) $ [u, [v, w]] = [ [u, v],w ]  + [v,[u, w]]$;\

C2) $\pi ( [u, v]) = [ \pi (u), \pi (v)]$;\

C3) $[u, f v ]= \pi (u) (f) v + f [u, v ]$;\

C4) $\pi ( u)\langle v, w \rangle = \langle [ u, v], w\rangle + \langle v, [u, w ]\rangle$;\

C5) $2[u, u ] = \pi^{*} d \langle u, u \rangle$.

A Courant algebroid is called {\cmssl regular} if the anchor $\pi$ has constant rank. 

\end{defn}

Here $\pi^* : T^*M \rightarrow E$ denotes the map obtained by dualizing $\pi : E\rightarrow TM$ 
and identifying $E^*$ with $E$ using the scalar product. Therefore
C5) can be written in the equivalent way 
\begin{equation}\label{explicit-pol}
2\langle [u, u], v\rangle = \pi (v) \langle u, u\rangle .
\end{equation}

\begin{rem}{\rm As already pointed out in \cite{u,b-h},  the axioms 
of a Courant algebroid can be reduced. One can show that axioms 
C2) and C3) from the definition of 
Courant algebroids follow from the other axioms. In fact, C2) can be checked by calculating  
$[\pi (u),  \pi (v) ]\langle w, e\rangle $, for any $w,e\in \Gamma (E)$,
 with the help of C4) and C1). Similarly, C3) can be checked
by taking the scalar product with a section $w\in \Gamma (E)$ and evaluating the result with 
help of C4).}
\end{rem}

\begin{exa}\label{exa-fundam} 
The fundamental example of a Courant algebroid is the generalized tangent bundle $\mathbb{T}M = TM \oplus T^*M$ of a smooth manifold $M$
with the canonical projection $\mathbb{T}M \rightarrow TM$ as anchor, the scalar product defined by 
$\langle X +\xi , Y+\eta \rangle = \frac12(\xi (Y) + \eta (X))$, and Dorfman  bracket by 
\begin{equation}\label{twist-0}
 [X +\xi , Y+\eta ] = [X,Y] + \mathcal{L}_X\eta - i_{Y}d\xi + H(X, Y, \cdot ),
\end{equation}
where $H\in \Omega^{3}(M)$ is a closed $3$-form and
$X, Y \in \Gamma (TM), \xi, \eta \in \Gamma (T^*M)$. On the right-hand side, $[X,Y]$ stands for the usual Lie bracket
of vector fields and ${\mathcal L}_{X}\eta$ for the Lie derivative of $\eta$ in the direction of $X$. 
It is known that every Courant algebroid for which the sequence $0\rightarrow T^*M\stackrel{\pi^*}{\rightarrow}E\stackrel{\pi}{\rightarrow}TM\rightarrow 0$ is exact is isomorphic to a Courant algebroid of this form. Such Courant algebroids are called {\cmssl exact}. 
\end{exa}

\subsection{$E$-connections and generalized connections}\label{gen-conn-torsion}

Unless otherwise stated, $E$ will denote a Courant algebroid, with anchor $\pi : E \rightarrow TM$,  scalar product 
$\langle\cdot, \cdot \rangle$ and Dorfman bracket $[\cdot , \cdot ] .$ Let  $V\rightarrow M$ be  a vector bundle.

\begin{defn} i)  An {\cmssl $E$-connection} on $V$ is an $\mathbb{R}$-linear map 
\[ D : \Gamma (V) \rightarrow \Gamma (E^*\otimes V),\quad v \mapsto Dv,\]
which satisfies the following Leibniz rule 
\[ D_e(fv) = \pi (e) (f)v+ fD_ev\]
for all $e\in \Gamma (E), v\in \Gamma (V), f\in C^{\infty}(M)$, where $D_ev = (Dv)(e)$.\

ii) An $E$-connection on $E$ is  called a {\cmssl generalized connection} if it is compatible with the 
scalar product: 
\begin{equation} \label{def-conn} \pi (u) \langle v,w\rangle = \langle D_uv,w\rangle + \langle v,D_uw\rangle \end{equation} 
for all $u,v,w\in \Gamma (E)$. 
\end{defn}

\begin{exa} \label{ext:exa} A connection  $\nabla$ on a vector bundle $V\rightarrow M$  
 induces an 
$E$-connection D on $V$ by the formula
\[ D_ev := \nabla_{\pi (e)}v,\]
for all $e\in \Gamma (E), v\in \Gamma (V)$. 
If $\nabla$ is a connection on $E$, then the induced $E$-connection is a generalized connection if and only if 
$\langle \cdot ,\cdot \rangle$ is $\nabla$-parallel along $\pi ( E)$, that is $\nabla_X \langle \cdot ,\cdot \rangle=0$ for all
$X\in \pi (E)$. 
Conversely,  if the Courant algebroid $E$ is regular, an $E$-connection $D$ on $V$ which satisfies $D_{e}=0$ for all $e\in \ker \pi$
is induced by a connection $\nabla$ on $V$ and a similar statement holds for generalized connections.  In fact, an $E$-connection $D$  on $V$ such that $D_e=0$  for all $e\in \ker \pi$ induces a partial connection\footnote{Recall that the notion of a partial connection is defined  by the same properties as that 
of a connection, namely tensoriality in the first argument and Leibniz rule in the second.}  
$\nabla^F : \Gamma(F) \times \Gamma (V) \rightarrow \Gamma (V)$ along the distribution 
$F=\pi (E)\subset TM$ such that $D_{e}=\nabla^F_{\pi (e)}$ for all $e\in E$. Choosing a 
partial connection $\nabla^{F'} : \Gamma(F') \times \Gamma (V) \rightarrow \Gamma (V)$ along a complementary distribution $F'\subset TM$, we can define 
a connection $\nabla$ such that  $D_{e}=\nabla_{\pi (e)}$ for all $e\in E$ by $\nabla_{X+Y}= \nabla_X^F + \nabla_Y^{F'}$ for all $X\in F$, $Y\in F'$. 
If $V =E$ is a  regular Courant algebroid and $D$  happens to be 
a generalized connection  on $E$  then the partial connection $\nabla^F$ is metric and we can choose 
$\nabla^{F'}$ (and hence $\nabla$) to be metric as well.
\end{exa}

\begin{defn} The {\cmssl torsion} $T^D\in \Gamma (\wedge^2E^* \otimes E)$ of a generalized connection $D$ is  defined by 
\begin{equation}\label{torsion-def}
T^{D} (u, v) = D_{u}v  - D_{v}u - [u, v] + (Du)^{*}v,\ \forall u, v\in \Gamma (E),
\end{equation}
where $(Du)^{*}$ is the metric adjoint  of $Du\in \Gamma (\mathrm{End}\, E)$ with respect to $\langle \cdot , \cdot \rangle$.  
\end{defn}
Note that the tensoriality of $T^{D}(u,v)$ in $v$ is obvious since the operators $D_u$ and $[u,\cdot ]$ satisfy the same Leibniz rule. 
The tensoriality in $u$ is a consequence of the skew-symmetry of $T^{D}$, which follows from 
\begin{equation}\label{c-d}
\langle [u,v] + [v,u],w\rangle = \pi (w) \langle u,v\rangle 
\end{equation}
(polarization of (\ref{explicit-pol})) and  the compatibility of $D$ with the scalar product.  
Moreover, one can check that $T^{D}(u,v,w) := \langle T^{D}(u,v),w \rangle$ is totally skew-symmetric. 
From (\ref{torsion-def}), 
\begin{equation}\label{courant-br}
\langle [u, v],w\rangle  = - T ^{D}(u, v, w) +  \langle D_{u}v - D_{v} u,w\rangle  + \langle D_{w}u, v\rangle .
\end{equation}

We often identify $E$ with $E^{*}$ using the scalar product $\langle \cdot , \cdot \rangle .$ 
By means of this identification, $\Lambda^{2} E \subset \mathrm{End}\, E$ is the subbundle of skew-symmetric 
endomorphisms of $E$, where
$$
(e_{1}\wedge e_{2})(e_{3}):= \langle e_{1}, e_{3}\rangle e_{2} - \langle e_{2}, e_{3} \rangle e_{1},\ e_{i}\in E.
$$
With this convention,  any two generalized connections are related by $\tilde{D} = D +\eta$, for $\eta \in \Gamma (E^{*}\otimes \Lambda^{2}E)$ 
and 
\begin{equation}\label{formula-torsion}
T^{\tilde{D}} (u, v, w) = T^{D}(u,v,w) + \eta (u, v, w) + \eta (w,u,v) +\eta (v,w,u).
\end{equation}
The space of torsion-free generalized connections is non-empty, see \cite{garcia}.   
If $D$ and $\tilde{D}= D +\eta$ are torsion-free  (or, more generally, have the same torsion), then $\eta$ is of the form
$$
\eta (u, v, w) = \sigma (u, v, w) -\sigma (u, w, v),
$$
for a section $\sigma\in \Gamma (S^{2} E\otimes E)$,  see \cite{garcia}. \\

\subsection{Generalized  metrics }\label{gen-riemann}

Let $E$ be a Courant algebroid with anchor $\pi : E \rightarrow TM$, scalar product $\langle \cdot , \cdot \rangle$ and
Dorfmann bracket $[\cdot , 
\cdot ].$

\begin{defn} A {\cmssl generalized  metric} on $E$ is a vector subbundle $E_{+}$ of $E$ on which
$\langle \cdot , \cdot \rangle$ is non-degenerate.
\end{defn}

Let
$E_{-}:=E_{+}^{\perp}$ be the orthogonal complement of $E_{+}$
with respect to $\langle \cdot , \cdot \rangle $. Then   $G:= \langle\cdot , \cdot\rangle\vert_{E_{+}\times E_{+} } -  
 \langle\cdot , \cdot\rangle\vert_{E_{-}\times E_{-}}$
is a non-degenerate metric on $E$. Alternatively, a generalized metric on  $E$ can be defined as a non-degenerate symmetric  bilinear form  
$G$ on the vector bundle $E$  such that the endomorphism $G^{\mathrm{end}}$, defined by
$G(u, v)  = \langle G^{\mathrm{end}}u, v\rangle$,  satisfies $(G^{\mathrm{end}} )^{2} =\mathrm{Id}_{E}.$
The bundles $E_{\pm}$ are the $\pm1$-eigenbundles of $G^{\mathrm{end}}.$  
Along the paper  we  denote by $e_{\pm}:=   \frac{1}{2} ( \mathrm{Id} \pm G^{\mathrm{end}}) e$  the $E_\pm$-components of a vector 
$e\in E$  in the decomposition 
$E = E_+ \oplus   E_-$ determined by a generalized metric $G$.

Given a generalized metric $G$ there is always a torsion-free generalized connection which preserves $G$  (see \cite{rubio,garcia}). 
Such a  connection  is not unique if $\mathrm{rank}\,  E > 1$. It  is called a {\cmssl Levi-Civita connection} of $G$.  The non-uniqueness is due to the fact that although
the first prolongation $\mathfrak{so}(k,\ell )^{(1)}$ is always trivial (which is responsible for the uniqueness of the Levi-Civita connection of a pseudo-Riemannian manifold), the generalized first prolongation $\mathfrak{so}(k,\ell )^{\langle 1\rangle}$ (see Definition \ref{prolonDef} below)  
is non-trivial if $k+\ell >1$.\

\subsection{Generalized complex and hyper-complex structures}

\begin{defn}
A {\cmssl generalized almost complex structure} on $E$ is an endomorphism $\mathcal J\in \Gamma (\mathrm{End}E)$ which satisfies $\mathcal J^{2}
=-\mathrm{Id}_{E}$ and is orthogonal with respect to the  scalar product  $\langle\cdot , \cdot\rangle$ of $E$.   
We say that $\mathcal J$ is  {\cmssl integrable} (or is a {\cmssl generalized complex structure}) 
if its  Nijenhuis tensor 
\begin{equation}\label{def-nj}
N_{\mathcal J }(u, v) : = [\mathcal J u , \mathcal Jv] - [ u, v] - \mathcal J ( [\mathcal J u, v] + [u, \mathcal J v])
\end{equation}
vanishes identically. 
\end{defn}

The following simple lemma will be useful for us.

\begin{lem}\label{nj-3form} Let $\mathcal J$ be a generalized almost complex structure on $E$. Then 
$N_{\mathcal J}(u, v, w) := \langle N_{\mathcal J}(u, v), w\rangle$ is a $3$-form on $E$.
\end{lem}

\begin{proof} 
The skew-symmetry of $N_{\mathcal J}(u, v, w)$ in the first two arguments follows from
(\ref{c-d}).  In order to prove the skew-symmetry of $N_{\mathcal J }(u, v, w)$ in the last two arguments 
we compute 
\begin{align}
 \nonumber&N_{\mathcal J} (u, v, w) +  N_{\mathcal J}(u, w, v)\\ 
 \nonumber&= \langle [\mathcal J u, \mathcal J v]- [u, v] ,w\rangle
+\langle [\mathcal J u, v]  + [u, \mathcal J v] , \mathcal J w\rangle\\
\nonumber& +  \langle [\mathcal J u, \mathcal J w] - [u, w]  ,v\rangle   +\langle [\mathcal J u, w] + [u, \mathcal J w]  , \mathcal J v\rangle\\
\nonumber&=  \langle [\mathcal J u, \mathcal J v], w\rangle 
+\langle \mathcal J v,  [\mathcal J u, w]\rangle - ( \langle  [u, v] ,w\rangle
+\langle  [u, w], v\rangle )\\
\nonumber & +  \langle [\mathcal J u,  v],\mathcal J  w\rangle  +
\langle [\mathcal J u,  \mathcal J w], v\rangle 
+\langle [u, \mathcal J  v],\mathcal J  w\rangle  +
\langle [ u,  \mathcal J w],\mathcal J  v\rangle .
\end{align}
Using the axiom C4) from the definition of Courant algebroids we obtain 
\begin{align}
\nonumber & N_{\mathcal J} (u, v, w) +  N_{\mathcal J}(u, w, v) \\
\nonumber & = \pi (\mathcal J u) \langle \mathcal J v, w\rangle -  \pi (u) \langle  v, w\rangle
+ \pi (\mathcal J u) \langle v, \mathcal J w\rangle + \pi (u) \langle\mathcal J   v,\mathcal J w\rangle =0.
\end{align}
We have proved that $N_{\mathcal J}= N_{\mathcal J}(u, v,w)$ is completely skew-symmetric. 
As  it is  $C^{\infty}(M)$-linear in $w$, it is $C^{\infty}(M)$-linear
also in $u$,  $v$. 
\end{proof}

\begin{defn}
A {\cmssl generalized almost hypercomplex structure}  on $E$ is a triple
$(\mathcal J_{1}, \mathcal J_{2}, \mathcal J_{3})$ 
of anti-commuting generalized almost complex structures such that $\mathcal J_{3} =\mathcal J_{1}\mathcal J_{2}.$ 
We say that $(\mathcal J_{1}, \mathcal J_{2}, \mathcal J_{3})$ is {\cmssl integrable} (is a {\cmssl generalized hypercomplex structure})
if all $\mathcal J_{i}$ are generalized  (integrable) complex structures.
\end{defn}

\subsection{Generalized K\"{a}hler and hyper-K\"{a}hler structures}

\begin{defn}  A {\cmssl generalized almost Hermitian} structure on $E$ is a 
pair $(G, \mathcal J)$ where $G$ is a generalized  metric and $\mathcal J$ a generalized
almost complex structure on $E$, such that
$G(\mathcal J u, \mathcal J v) =  G(u, v)$ for all $u, v\in E.$ 
\end{defn}

The endomorphism
$G^{\mathrm{end}}$ of $E$ determined by $G$  (see Section \ref{gen-riemann}) commutes with
$\mathcal{J}_1 :=\mathcal{J}$ and $\mathcal J_{2} := G^{\mathrm{end}} \mathcal J$ is a generalized almost complex structure. The generalized almost complex structures $\mathcal J_{1}$ and $\mathcal J_{2}$ preserve $E_{\pm}$
(the $\pm1$-eigenbundles of $G^{\mathrm{end}}$).  
Moreover, $\mathcal J_{2} =\pm \mathcal J_{1}$ on $E_{\pm}.$ 

\begin{defn} A generalized almost Hermitian structure $(G, \mathcal J)$ is called a {\cmssl generalized K\"{a}hler structure}
if $\mathcal J_{1} = \mathcal J$ and $\mathcal J_{2} = G^{\mathrm{end}} \mathcal J_{1}$ are integrable.
\end{defn}

The argument from Proposition 2.17 of \cite{gualtieri-kahler}
works also in the setting of non-exact Courant algebroids and shows that
$(G, \mathcal J)$ is a generalized K\"{a}hler structure if and only if 
$L\cap (E_{\pm})_{\mathbb{C}}$  are closed under the
Dorfmann bracket $[\cdot ,\cdot ]$ of $E$, where $L$ denotes the $(1,0)$-bundle of $\mathcal J .$

\begin{defn}
i) A {\cmssl generalized  almost hyper-Hermitian structure}   is a generalized almost  
hypercomplex structure  $(\mathcal J_{1}, \mathcal J_{2}, \mathcal J_{3})$ 
together with a generalized metric $G$,  
such that $(G, \mathcal J_{i})$ is a generalized almost Hermitian structure, for all $i=1,2,3.$\

ii)   A generalized almost hyper-Hermitian structure  $(G, \mathcal J_{1},\mathcal J_{2}, \mathcal J_{3})$ 
is a {\cmssl generalized hyper-K\"{a}hler structure}  if   
$(G, \mathcal J_{i})$ is a generalized K\"{a}hler structure, for all $i=1,2,3.$\\
\end{defn}

\subsection{Intrinsic torsion of a generalized structure}

Given a linear Lie group $H\subset \mathrm{O}(k,\ell )$, a {\cmssl generalized $H$-structure} on a Courant algebroid $E$ 
with scalar product  $\langle \cdot , \cdot \rangle$ of signature $(k, \ell )$ is a reduction of the structure group of $E$ from $\mathrm{O}(k,\ell )$ to $H$. 

Let $(T_i^0)$ be a system of tensors on $\mathbb{R}^{k+\ell }$ and 
$Q= (T_{i})$ a system of tensor fields on $E$, such that $T_i$ is pointwise linearly equivalent to $T_i^0$. 
Then we associate to $Q$ the generalized $H$-structure defined as the  principal bundle of standard orthonormal frames 
of $E$ with respect to which each of the tensor fields $T_{i}$ takes the constant form $T_i^0$ and the scalar product is represented by 
the matrix $\mathrm{diag}(\mathbf{1}_k,-\mathbf{1}_\ell)$.
Examples  considered in this paper 
are: 
\begin{enumerate}
\item generalized almost complex structures $Q=\mathcal{J}$,  $H=\mathrm{U}(\frac{k}{2}, \frac{\ell}{2})$,

\item generalized almost hypercomplex structures $Q=(\mathcal{J}_1,\mathcal{J}_2,\mathcal{J}_3)$,
$H=\mathrm{Sp}(\frac{k}{4},\frac{\ell}{4} )$,
\item generalized Riemannian metrics $Q=G$,  $H= \mathrm{O}(k_{+}, l_{+}) \times \mathrm{O}(k_{-}, l_{-})$,
where $\langle \cdot , \cdot\rangle\vert_{E_{\pm}}$ has signature $(k_{\pm}, \ell_{\pm})$
(in particular, $k_{+} + k_{-} =k$ and $\ell_{+} + \ell_{-} = \ell$), 

\item generalized almost Hermitian structures $Q=(G,\mathcal{J})$,  $H= \mathrm{U}(\frac{k_{+}}{2},
\frac{\ell_{+}}{2}) \times \mathrm{U}(\frac{k_{-}}{2},
\frac{\ell_{-}}{2})$, where $\langle \cdot , \cdot\rangle\vert_{E_{\pm}}$ has signature $(k_{\pm}, \ell_{\pm})$,  and 
\item generalized almost hyper-Hermitian structures $Q=(G, \mathcal{J}_1,\mathcal{J}_2,\mathcal{J}_3)$, 
$H= 
\mathrm{Sp}(\frac{k_{+}}{4},\frac{\ell_{+}}{4}) \times \mathrm{Sp}(\frac{k_{-}}{4},\frac{\ell_{-}}{4}) $ .
\end{enumerate}
We will refer to these simply as {\cmssl generalized $H$-structures} $Q$.

In analogy with the classical case (see e.g.\ \cite{sternberg}), given a 
generalized $H$-structure  $Q$ on  $E$  we may consider the space of  generalized connections $D$ 
which preserve, or are adapted, to  $Q$ (i.e.  $DQ=0$). 
 A generalized connection adapted to a generalized  almost complex (respectively, hypercomplex) structure 
$Q=\mathcal J$ (respectively, $Q = (\mathcal J_{1}, \mathcal J_{2}, \mathcal J_{3})$)  will be called {\cmssl complex} 
(respectively, {\cmssl hypercomplex}).\

\begin{lem} Any generalized $H$-structure $Q$ on a Courant algebroid $E$ admits an adapted generalized connection.
\end{lem}

\begin{proof} Any connection on the principal $H$-bundle of frames which defines the generalized $H$-structure induces a 
connection on $E$ for which $Q$ and the scalar product $\langle \cdot ,\cdot \rangle$ are parallel. This connection extends as in Example \ref{ext:exa} to a generalized connection adapted to $Q$.
\end{proof}

The space of generalized connections adapted to a generalized $H$-structure $Q$  is an affine space, modelled on the vector space  of sections of the bundle $E^{*} \otimes \mathrm{ad}_{Q}( E)$,
where $\mathrm{ad}_{Q} (E)$ denotes the bundle of $2$-forms $\alpha\in \Lambda^{2}E^{*}$, adapted to $Q$.
By the latter condition we mean that $\mathrm{ad}_{\alpha}\in \mathrm{End} ({\mathcal T}_{E})$, defined as the natural extension of the action of $\alpha\in \Lambda^{2}E^{*}\cong \mathfrak{so}(E)\subset \mathrm{End}(E)$ to the tensor bundle ${\mathcal T}_{E}
= \oplus_{p, q} E^{p}\otimes (E^{*})^{q}$ of $E$,  
annihilates the tensor fields from $Q$. 
When $Q$ belongs to the above list,    
we require that $[\alpha ,\mathcal J ]=0$ and $[\alpha , G^{\mathrm{end}} ]=0$ when
$\mathcal J , G\in Q.$    These conditions are equivalent to $\mathcal{J}^*\alpha = \alpha$ and $(G^{\mathrm{end}})^*\alpha = \alpha$.  
Note that the fiber $\mathrm{ad}_{Q} (E)_x$, $x\in M$, of the bundle $\mathrm{ad}_{Q}(E)\rightarrow M$ is a Lie 
algebra isomorphic to the Lie algebra $\mathfrak{h}\subset \mathfrak{so}(k,\ell )$ of the structure group $H$.
The map 
\begin{equation}\label{linear-torsion}
\partial_{Q} : E^{*} \otimes \mathrm{ad}_{Q} (E) \rightarrow \Lambda^{3} E^{*},\ 
(\partial_{Q} \eta ) (u, v, w) := \eta (u, v,w) + \eta (w, u, v) +\eta (v, w, u)
\end{equation}
is called the {\cmssl algebraic torsion map}. From (\ref{formula-torsion}), the image of $T^{D}\in \Gamma (\Lambda^{3}E^{*})$ in the quotient bundle 
$(\Lambda^{3} E^{*})/\mathrm{im}\,  \partial_{Q}$  
is independent of the choice of adapted generalized  connection $D$. 

\begin{defn} \label{intrinsic_torsionDef} Let $Q$ be a  generalized $H$-structure on $E$ and $D$ 
an adapted generalized connection.  The class of $T^{D}$ in  
$(\Lambda^{3}E^{*})/ \mathrm{im}\,  \partial_{Q}$  is called the {\cmssl  intrinsic torsion} of $Q$.
\end{defn}

We shall often consider the intrinsic torsion as a $3$-form on $E$, by choosing a suitable 
complement $C(E)$ of $\mathrm{im}\,  \partial_{Q}$ in $\Lambda^{3}E^{*}$ and identifying 
the quotient   $(\Lambda^{3}E^{*})/ \mathrm{im}\,  \partial_{Q}$  with $C(E).$

\begin{defn}\label{prolonDef}
The {\cmssl generalized first prolongation} 
of a Lie algebra $\mathfrak{h}\subset \mathfrak{so}(k,\ell )$ $\cong \Lambda^2 V^*$, $V :=\mathbb{R}^{k+\ell}$,  is the subspace 
\[ \mathfrak{h}^{\langle 1\rangle} := \{ \eta \in V^*\otimes \mathfrak{h} \mid \partial \eta =0\} \subset 
V^*\otimes  \Lambda^{2} V^{*},\]
where $\partial : V^*\otimes  \Lambda^2 V^* \rightarrow \Lambda^3V^*$ is defined
by: 
\[ (\partial \eta)(u,v,w) := \eta (u, v,w) + \eta (w, u, v) +\eta (v, w, u),\quad u,v,w\in V.\]
\end{defn}

\begin{prop} \label{connProp} Let $Q$ be a generalized $H$-structure on a Courant algebroid $E$ and $D_0$ 
an adapted generalized connection with torsion $T^{D_{0}}$. Given a section $T\in \Gamma (\Lambda^3E^*)$
there exists a generalized connection $D$ adapted to $Q$ with torsion $T^D=T$ if and only if 
$T-T^{D_0}\in \Gamma (\mathrm{im}\, \partial_Q)$. The generalized connection $D$ is unique 
up to addition of a section of $\mathrm{ker}\, \partial_Q \subset E^*\otimes \mathrm{ad}_{Q}(E)$. 
It is unique if and only if the generalized first prolongation $\mathfrak{h}^{\langle 1\rangle}$ of the 
Lie algebra $\mathfrak{h} \subset \mathfrak{so}(k,\ell )$ of the structure group $H$ vanishes. 
\end{prop}
\begin{proof} This is obtained by writing an arbitrary adapted generalized connection as 
$D=D_0+\eta$, where $\eta \in \Gamma (E^*\otimes \mathrm{ad}_{Q}(E))$ and observing 
that $T^D = T^{D_0} + \partial_Q\eta$, see (\ref{formula-torsion}). The last statement follows from the fact that 
the fibres of the bundle $\mathrm{ker}\, \partial_Q$ are isomorphic to $\mathfrak{h}^{\langle 1\rangle}$. 
\end{proof}

The next corollary  follows easily from Proposition \ref{connProp}.  

\begin{cor}\label{cor-torsion}
In the setting of Proposition \ref{connProp},  $Q$ admits an adapted torsion-free  generalized connection $D$  if and only if the intrinsic torsion of $Q$ vanishes.  
The generalized connection $D$ is unique if and only if $\mathfrak{h}^{\langle 1\rangle}=0.$
\end{cor}

\subsection{Application to Born geometry}\label{section-born}

The vanishing of the generalized first prolongation for certain Lie subalgebras of $\mathfrak{so}(k,\ell )$ can be used to prove the 
uniqueness of certain connections (rather than generalized connections). As an example we mention the canonical connection in Born geometry defined in \cite{FRS}, 
which is related to string compactifications and more specifically to double field theory. 
Its uniqueness, proven in  \cite{FRS}, can be alternatively deduced from the vanishing of the generalized first prolongation of the diagonal $\mathfrak{so}(n)$-subalgebra 
\[ \Delta_{\mathfrak{so}(n)} = \{ (A,A) \in \mathfrak{so}(n) \oplus \mathfrak{so}(n)\mid A\in \mathfrak{so}(n) \}  \subset \mathfrak{so}(n,n).\] 
The corresponding diagonally embedded $\mathrm{O}(n)$-subgroup $\Delta_{\mathrm{O}(n)} \subset \mathrm{O}(n,n)$ is precisely 
the automorphism group of the following data on $V=\mathbb{R}^{2n}$:
\begin{enumerate}
\item a scalar product $\eta = \langle \cdot , \cdot \rangle$ of neutral signature,
\item a positive definite scalar product $g$ and 
\item a linear involution $K$,
\end{enumerate}
which satisfy the following compatibility conditions: 
\begin{enumerate}
\item  $K$ is skew-symmetric with respect to $\eta$, 
\item $J=\eta^{-1}g$ is an involution and  
\item $J$ anti-commutes with $K$. 
\end{enumerate}
These properties imply that the triple $(I:=KJ,J,K)$ is a {\cmssl para-hypercomplex structure} on $V$, i.e.\ 
$I, J, K$ pairwise anti-commute, $K=IJ$ and $-I^2= J^2=K^2 = \mathrm{Id}$. However, the structure is 
not para-hyper-Hermitian with respect to $\eta$ as only $K$ is skew-symmetric, whereas $I$ and $J$ are 
symmetric with respect to $\eta$. A quadruple $(\eta , I, J,K)$ with these properties is called a {\cmssl Born structure} on $V$ 
if the symmetric bilinear form $g= \eta J = \eta (J \cdot , \cdot )$ is positive definite. 
So the data $(\eta , g, K)$ with the above compatibility relations 1.-3.\ is equivalent to a Born structure $(\eta , I, J,K)$ on $V$. 

Given smooth tensor fields $(\eta, I, J,K)$ on a manifold $M$ such that  $(\eta_p, I_p, J_p$, $K_p)$ is a Born structure on $T_pM$ for all $p\in M$, the data
$(\eta, I, J,K)$ is called a {\cmssl Born structure} on $M$.  
It is proven in \cite{FRS} that every Born structure $(\eta, I, J,K)$ on a manifold $M$ admits a canonical compatible connection $\nabla$,  
called the {\cmssl Born connection},   
with vanishing {\cmssl generalized torsion} $\mathcal{T}_\nabla \in \Omega^3(M)$, where 
\[ \mathcal{T}_\nabla (X,Y,Z) := \eta (\nabla_XY-\nabla_YX-[X,Y]^c+(\nabla X)^*Y,Z)\]
is defined in terms of the so-called {\cmssl canonical D-bracket}
\[ [X,Y]^c = \nabla^c_XY-\nabla^c_YX +(\nabla^cX)^*Y.\]
Here $B^*$ denotes the $\eta$-adjoint of an endomorphism field $B$ and $\nabla^c := \nabla^\eta +\frac12K\nabla^\eta K$ is the  
canonical connection compatible with the almost para-Hermitian structure $(\eta ,K)$, where $\nabla^\eta$ denotes the Levi-Civita connection of $\eta .$
The Born connection is denoted $\nabla^B$ and is defined by 
\[ \nabla^B_XY := [X_-,Y_+]_+^c + [X_+,Y_-]^c_- + (K[X_+,KY_+]^c)_+ + (K[X_-,KY_-]^c)_-,\]
where $\pm$ stands for the projections onto the eigendistributions of $J$. 
It is clear that any two connections $\nabla'$ and $\nabla$ compatible with the same Born structure and 
having the same generalized torsion differ by a section $A\in \Gamma (T^*M \otimes \mathfrak{so}(TM))$ in the kernel of the 
total skew-symmetrization map $\partial$, 
such that $A_p$ belongs to the generalized first prolongation of the Lie algebra $\mathrm{aut}(T_pM,\eta_p,I_p.J_p,K_p) \cong 
\Delta_{\mathfrak{so}(n)}$. This shows, in particular, that the uniqueness of $\nabla^B$ follows from the next lemma. 
\begin{lem} $\Delta_{\mathfrak{so}(n)}^{\langle 1\rangle }=0$. 
\end{lem}

\begin{proof} Let $(e_1,\ldots ,e_n,e_1',\ldots ,e_n')$ be a basis of $V$,  orthonormal with respect to both $\langle\cdot , \cdot \rangle$ and $g$,
where 
$$
\langle e_{i} , e_{i}\rangle = g(e_{i}, e_{i}) =1,\ \langle e_{i}^{\prime}, e_{i}^{\prime} \rangle =  - g (e_{i}^{\prime}, e_{i}^{\prime} ) = -1, 
$$
such that $K(e_{i}) = e_{i}^{\prime}$, $K(e_{i}^\prime ) = e_{i}.$ Then $J(e_{i}) =e_{i}$ and $J(e_{i}^{\prime})= - e_{i}^{\prime}$, 
from where  we deduce that $A_{v}$ preserves $\mathrm{span} \{ e_1,\ldots ,e_n\}$ and  $\mathrm{span} \{ e^{\prime}_1,\ldots ,e^{\prime}_n\}$,
for any $A\in V^{*}\otimes \Delta_{\mathfrak{so}(n)}$ and $v\in V$
(since $A_{v}$  commutes with $J$). We deduce that  $A$  is completely determined by two $(1,2)$-tensors 
$(A_{ij}^k)$ and $({A'}_{ij}^k)$ on $\mathbb{R}^n$, where $A_{e_i}e_j=\sum_{k} A_{ij}^ke_k$ and $A_{e_i'}e_j=\sum_{k} {A'}_{ij}^ke_k$, since 
(from $A_{v} K = K A_{v}$ for any $v$) 
$A_{e_{i}} e_{j}^{\prime} = \sum_{k}A_{ij}^{k} e_{k}^{\prime}$ and $A_{e_{i}^{\prime}} e_{j}^{\prime} = \sum_{k}{A'}_{ij}^{k} e_{k}^{\prime}$.
From $\partial A=0$ we obtain 
\[ 0=\langle A_{e_i}e_j',e_k'\rangle + \langle A_{e_j'}e_k',e_i\rangle + \langle A_{e_k'}e_i,e_j'\rangle = -A_{ij}^k\]
and similarly ${A'}_{ij}^k=0$. This proves that $A=0$. 
\end{proof}

\section{Generalized almost complex structures}\label{sect-complex}

\subsection{Intrinsic torsion of a generalized almost complex  structure}

In this section we compute the intrinsic torsion $t_\mathcal{J}$ of a generalized almost complex structure $\mathcal{J}$ on a Courant algebroid $E$ and relate it to the Nijenhuis tensor $N_\mathcal{J}$.   
Before we need to recall basic facts on projectors. 
Recall that an endomorphism $P\in \Gamma (\mathrm{End}V)$ of a vector bundle $V$  is a projector onto a subbundle
$V_{0}\subset V$ if
$P^{2} = P$ and $\mathrm{im}\, P  = V_{0}.$ Then $V$ decomposes as $V =  V_{0}\oplus \ker  P $ and
$P$ (respectively $\mathrm{Id}- P$)  are the projections onto $V_ {0}$ (respectively $\ker P$) along this decomposition. 
In particular, there is a canonical choice of a complement of $\mathrm{im}\,  P$ in $V$, namely, $\ker P.$

Consider now  the algebraic torsion map
$\partial_{\mathcal J} : E^{*} \otimes 
 \Lambda^{1,1}_{\mathcal J} E^{*} \rightarrow \Lambda^{3} E^{*}$   of $\mathcal J$, defined by (\ref{linear-torsion}),   where $\mathrm{ad}_{\mathcal J} (E)= \Lambda^{1,1}_{\mathcal J}E^{*}$ is the bundle  of $\mathcal J$-invariant $2$-forms on $E$.

\begin{lem}\label{alg-complex} i) The  map $\Pi_{\mathcal J} \in\Gamma ( \mathrm{End}( \Lambda^{3} E^{*}))$ defined by
\begin{eqnarray}\label{PiJ:eq}
&(\Pi_{\mathcal J} \alpha)(u, v, w) :=\nonumber\\
&\frac{1}{4} \left( \alpha (u, v, w) - \alpha (u,\mathcal J v,
\mathcal J w) 
- \alpha (\mathcal J u, v, \mathcal J w)  - \alpha (\mathcal J u, \mathcal J v, w) \right)
\end{eqnarray}
is a projector onto the subbundle 
$$
\Lambda^{3}_{\mathcal J} E^{*} := \{ \alpha \in \Lambda^{3}E^{*},\ 
\alpha (\mathcal J u, v,  w) = \alpha (u, \mathcal J v,  w)= \alpha (u, v, \mathcal J w)\} .
$$
ii) The equality $\mathrm{im}\,   \partial_{\mathcal J} = \ker  \Pi_{\mathcal J} $ holds. 
\end{lem}

\begin{proof}
It is straightforward to check  that $\mathrm{im}\, \Pi_{\mathcal J} \subset  \Lambda^{3}_{\mathcal J}E^{*}$.
Also, for all $\alpha \in \Lambda^{3}_{\mathcal J}E^{*}$, $\Pi_{\mathcal J}(\alpha ) =\alpha $. We obtain that 
$\mathrm{im}\, \Pi_{\mathcal J} =  \Lambda^{3}_{\mathcal J}E^{*}$ and 
$\Pi_{\mathcal J}^{2} = \Pi_{\mathcal J}$. Claim i) is proved. To prove claim ii), 
we notice that $\Pi_{\mathcal J}\circ \partial_{\mathcal J}=0$, i.e.   $\mathrm{im}\,  \partial_{\mathcal J} \subset  \ker   \Pi_{\mathcal J}$. Let
$$
\tilde{\pi}_{\mathcal J} : \Lambda^{3}E^{*}\rightarrow E^{*}\otimes \Lambda^{1,1}_{\mathcal J}E^{*},\  (\tilde{\pi}_{\mathcal J} \alpha )(u, v, w):=
\alpha (u, v, w) +\alpha (u, \mathcal J v, \mathcal J w).
$$
It is straightforward to check that $\Pi_{\mathcal J}=\mathrm{Id}_{\Lambda^{3}E^{*}}-\frac{1}{4} \partial_{\mathcal J}\circ 
\tilde{\pi}_{\mathcal J}$, which implies $\ker  \Pi_{\mathcal J} \subset \mathrm{im}\, \partial_{\mathcal J}.$  We proved that 
$\mathrm{im}\, \partial_{\mathcal J}  =  \ker  \Pi_{\mathcal J}$. 
\end{proof}

The next corollary follows from Lemma \ref{alg-complex} and our comments before this lemma.

\begin{cor}\label{complement}
With the notation from Lemma \ref{alg-complex},  
$\mathrm{im}\, \Pi_{\mathcal J}$ is a complement of $\mathrm{im}\, \partial_{\mathcal J}$ in $\Lambda^{3}E^{*}.$  
\end{cor}

From Corollary \ref{complement}, we can (and will)
identify the quotient $(\Lambda^{3} E^{*})/( \mathrm{im}\, \partial_{\mathcal J} )$ with $\mathrm{im}\, \Pi_{\mathcal J} = \Lambda^{3}_{\mathcal J} E^{*}$ 
and  consider the intrinsic torsion $t_{\mathcal J}$ of $\mathcal J$ as a section of $ \Lambda^{3}_{\mathcal J }E^{*}$.
On the other hand, 
$N_{\mathcal J}(\mathcal J u, v) = - \mathcal J N_{\mathcal J}(u, v)$ (easy check),  which implies that 
$N_{\mathcal J}$,   
considered as a $3$-form 
(see  Lemma \ref{nj-3form}),  is also a section of $\Lambda^{3}_{\mathcal J}E^{*}$. Up to a constant, 
$N_{\mathcal J}$ and $t_{\mathcal J}$ coincide. More 
precisely,  we have the following result. 

\begin{cor}\label{pi-d} 
The torsion $T^{D}$ of a  generalized connection 
$D$ with $D\mathcal J =0$ satisfies
\begin{equation}\label{torsion-complex}
T^{D} (u, v, w) - T^{D} (\mathcal J u, v,\mathcal J w) - T^{D}( u, \mathcal J v, \mathcal J w) - T^{D}( \mathcal J u,
\mathcal J v, w) =N_{\mathcal J}(u, v, w),
\end{equation}
for all $u, v, w\in \Gamma (E).$  In particular,   $t_{\mathcal J }= \frac{1}{4} N_{\mathcal J}$ (viewed as $3$-forms on $E$).
\end{cor}

\begin{proof} 
Relation (\ref{torsion-complex}) follows from  (\ref{def-nj}), 
together  with
\begin{equation}\label{courant-torsion}
[u, v] = D_{u}v - D_{v} u +  (Du)^{*} v  - T^{D}(u, v)
\end{equation}
and $D\mathcal J =0$. 
Relation (\ref{torsion-complex})  can be written as 
\begin{equation}\label{courant-torsion-new}
\Pi_{\mathcal J} (T^{D}) =  \frac{1}{4} N_{\mathcal J}
\end{equation}
 which implies the second statement.
\end{proof}

\subsection{Integrability using torsion-free generalized connections}

In this section we prove the following theorem.

\begin{thm}\label{complex:thm} A generalized almost complex structure $\mathcal J$ on $E$  is integrable if and only if there is a torsion-free generalized connection $D$ such that $D\mathcal J =0.$  
\end{thm}

Part of the statement of Theorem \ref{complex:thm} follows from 
Corollary  \ref{pi-d}: if there is a torsion-free generalized connection $D$ such that $D\mathcal J =0$ then 
from relation
(\ref{torsion-complex})  $\mathcal J$ is
integrable.  For the converse statement, let  $\mathcal J\in \Gamma ( \mathrm{End} E)$ be a generalized almost complex structure. 
We will  construct  a generalized complex connection  whose torsion equals the intrinsic torsion of
$\mathcal J$.  
Such a generalized connection will be torsion-free (and complex) when  $\mathcal J$ is integrable. 
This will conclude the proof of Theorem \ref{complex:thm}. The next remark represents
our motivation for the choice of
generalized connection $\tilde{D}$ in  Proposition \ref{essential}.

\begin{rem}\label{usual-case}{\rm  Given an almost  complex structure $J$ and a torsion-free connection 
$\nabla$ on a manifold $M$, the connection
\begin{equation}\label{conn-KN}
\tilde{\nabla}_{X}Y= \nabla_{X}Y -\frac{1}{4}\{ (\nabla J)X, J\} Y  -\frac{1}{2} J (\nabla_{X} J)Y,
\end{equation} 
where  $\{ A, B\} := AB + BA$  denotes the anti-commutator
of $A$ and $B$ and $(\nabla  J)X\in \Gamma (\mathrm{End}\, TM)$ is 
defined by $Y\rightarrow (\nabla_{Y}J)X$,   is complex ($\tilde{\nabla} J =0$) 
and 
its torsion satisfies $T^{\tilde{\nabla}}(X, Y) =\frac{1}{8} N_{J}(X, Y)$ 
(see Theorem 3.4 of \cite{KN};
remark the difference by a multiplicative factor between our definition for  the 
Nijenhuis tensor and that of \cite{KN}).   
In particular, if $J$ is integrable, then $\tilde{\nabla}$ is torsion-free (and complex). 
Now, for  a generalized almost complex structure $\mathcal J$ and a generalized
torsion-free connection $D$ on the Courant algebroid $E$, we may  define the analogous expression
\begin{equation}\label{conn-KN-gen}
\tilde{D}^{\prime}_{u}v= D_{u}v -\frac{1}{4}\{ (D \mathcal J)u, \mathcal J\} v  -\frac{1}{2} \mathcal J (D_{u} \mathcal J)v.
\end{equation}
However, $\tilde{D}^{\prime}$ defined by (\ref{conn-KN-gen})   is not a generalized connection
(while $\mathcal J D_{u}\mathcal J$ is  skew-symmetric with respect to the scalar product
$\langle\cdot , \cdot\rangle$ of $E$,  the anticommutator   
$\{ (D \mathcal J)u, \mathcal J\} $ is not
and $\tilde{D}^{\prime}$ does not preserve  $\langle\cdot , \cdot \rangle$,  in general; we shall give more details on this argument in 
Lemma \ref{easy}).   In the next lemma we modify $\tilde{D}^{\prime}$ in order to obtain a generalized
connection. It will turn out that it has the required properties.}
\end{rem}

\begin{prop}\label{essential}Let $\mathcal J$ be a generalized almost complex structure and $D$ a torsion-free generalized connection 
on $E$. Define
\begin{equation}\label{tilde-d}
\tilde{D}_{u}v = D_{u}v -\frac{1}{4} \{  A_{u}^{\mathrm{sym}}, \mathcal J \}v -\frac{1}{2} \mathcal J (D_{u} \mathcal J )v,
\end{equation}
where $A_{u}:= (D\mathcal J )u$ and $A_{u}^{\mathrm{sym}}$ is its $\langle\cdot , \cdot \rangle $-symmetric part. 
Then $\tilde{D}$ is a generalized connection, which preserves $\mathcal J$. Its torsion is given by 
\begin{equation}\label{torsion-tilded}
T^{\tilde{D}}(u, v, w) =\frac{1}{4}  N_{\mathcal J}(u, v, w). 
\end{equation}
In particular, if $\mathcal J$ is integrable, then $\tilde{D}$ is torsion-free (and complex). 

\end{prop}

We divide the proof of the above proposition into several lemmas.

\begin{lem}\label{easy} Equation (\ref{tilde-d}) defines a generalized complex connection. 
\end{lem}

\begin{proof}
Note that  $\mathcal J D_{u}\mathcal J$ is skew-symmetric with respect to 
the  scalar product $\langle\cdot , \cdot\rangle $ of $E$  ($D_{u}\mathcal J$ is skew-symmetric and  also
$\mathcal J D_{u}\mathcal J$ is  skew-symmetric, being the composition of two anti-commuting skew-symmetric endomorphisms). 
Similarly,  $\{  A_{u}^{\mathrm{sym}}, \mathcal J \}$ is skew-symmetric, 
because $A_{u}^{\mathrm{sym}}$ is symmetric and $\mathcal J$ is skew-symmetric.  
We obtain that $\tilde{D}_{u}$ and $D_{u}$ differ by a skew-symmetric endomorphism, i.e.\ $\tilde{D}$ is a generalized connection.  
The generalized connection 
\begin{equation} \label{D1:eq} D^{(1)}_{u}:=D_{u} -\frac{1}{2} \mathcal J D_{u}\mathcal J\end{equation} 
preserves $\mathcal J .$ As $\{ A_{u}^{\mathrm{sym}},
\mathcal J \}$ commutes with $\mathcal J$,  we obtain that 
\begin{equation}\label{d-t-1}
\tilde{D}_{u} = D^{(1)}_{u} -\frac{1}{4} \{ A_{u}^{\mathrm{sym}} , \mathcal J \}
\end{equation} 
preserves $\mathcal J$ as well. 
\end{proof}

In the next lemmas we prove  relation (\ref{torsion-tilded}). 
From (\ref{d-t-1}), 
\begin{align}
\nonumber \langle \tilde{D}_{u}v, w\rangle = \langle D^{(1)}_{u}v , w\rangle & -\frac{1}{8} \left( \eta (\mathcal J v, u, w)
+ \eta ( w, u, \mathcal J v) \right)\\
\label{tilde-d-d1} & +\frac{1}{8} \left( \eta (v, u, \mathcal J w) +\eta (\mathcal J w, u, v)\right) ,
\end{align}
where $\eta = \eta^{D, \mathcal J}$ is defined by  
\begin{equation}\label{eta}
\eta  (u, v, w) : = \langle (D_{u} \mathcal J )v, w\rangle .
\end{equation}
Remark that $\eta$  has the symmetries
\begin{equation}\label{symmetries}
\eta (u, v, w) = - \eta (u, w, v),\ \eta (u, \mathcal J v , w) = \eta (u, v, \mathcal J w).
\end{equation}

\begin{lem} The torsion of $D^{(1)}$ is given by 
\begin{equation}\label{t-d1}
T^{D^{(1)}} (u, v, w) = \frac{1}{2} \sum_{(u,v,w)\; \mathrm{cyclic}} \eta (u, v, \mathcal J w)
\end{equation}
where the sum is over  cyclic permutations on $(u,v, w).$ 
\end{lem}

\begin{proof}   The claim follows from the torsion-free property of $D$ together with relations (\ref{formula-torsion}) and 
(\ref{D1:eq}).
\end{proof}

\begin{lem}\label{nij} The following relation holds:
\begin{equation}\label{cyclic-n}
N_{\mathcal J} (u, v , w) = \sum_{(u,v,w)\; \mathrm{cyclic}} \left( \eta (u, v, \mathcal J w) 
+\eta (\mathcal J u , v, w)\right) .
\end{equation}
\end{lem}

\begin{proof} 
Using relations (\ref{def-nj}), (\ref{courant-torsion}) and $T^{D}=0$, we obtain
\begin{align*}
N_{\mathcal J}(u, v) & = (D_{\mathcal J u}\mathcal J )v - (D_{\mathcal J v}\mathcal J )u
 -\mathcal J (D_{u} \mathcal J )v\\
& +\mathcal J (D_{v} \mathcal J )u + (D(\mathcal J u))^{*} \mathcal J v-  (Du)^{*}v\\
& -  \mathcal J (D (\mathcal J u))^{*}v -  \mathcal J (Du)^{*} \mathcal J v.
\end{align*}
Taking the inner product of the above equality with $w$ and using the symmetries (\ref{symmetries}) of $\eta$ we obtain
\begin{align}
\nonumber N_{\mathcal J}(u, v,w)= &
\langle N_{\mathcal J }(u, v), w\rangle  =\eta ( \mathcal J u, v, w) -\eta ( \mathcal J v, u, w) +\eta (w,u, \mathcal J v)\\
\label{cyclic-rel}&- \eta (v, u, \mathcal J w) +\eta ( \mathcal J w, u, v) +\eta (u, v, \mathcal J w) .
\end{align} 
Taking in (\ref{cyclic-rel})  cyclic permutations over $u$, $v$, $w$, using again the symmetries (\ref{symmetries}) of $\eta$ 
and that $N_{\mathcal J} (u, v, w)$ is completely skew 
we obtain (\ref{cyclic-n}).
 \end{proof}

The next lemma concludes the proof of Proposition \ref{essential} and
Theorem \ref{complex:thm}.

\begin{lem} The torsion of $\tilde{D}$ satisfies relation (\ref{torsion-tilded}).\end{lem}

\begin{proof} From relations  (\ref{tilde-d-d1}) and (\ref{formula-torsion}), we obtain 
\begin{align*}
&  T^{\tilde{D}}(u, v, w) - T^{D^{(1)}} (u, v, w)   = \\
& - \frac{1}{8}\sum_{(u,v,w)\, \mathrm{cyclic}} \left( \eta (\mathcal J  v, u,  w)   + \eta ( w ,u,\mathcal J v) -\eta (v,u,\mathcal Jw) -\eta (\mathcal Jw,u,v)\right) \\
 & = - \frac{1}{4}\sum_{(u,v,w)\, \mathrm{cyclic}} \left( \eta (\mathcal J  u, w,  v)   + \eta ( u ,v,\mathcal J w) \right)  , 
\end{align*}
where in the last equality  we have used
the symmetries (\ref{symmetries}) of $\eta .$ 
From  (\ref{t-d1}) we then obtain  
\begin{equation}
T^{\tilde{D}}(u, v, w)  =  \frac{1}{4}\sum_{(u,v,w)\; \mathrm{cyclic}} \left( 
\eta (u, v, \mathcal J w)+\eta (\mathcal  J u,  v,w ) \right) ,
\end{equation}
which implies (\ref{torsion-tilded}), from Lemma \ref{nij}.
\end{proof}

\section{Generalized almost hypercomplex structures}\label{sect-hypercomplex}

\subsection{Intrinsic torsion of a generalized almost hypercomplex structure}

Let $(\mathcal J_{1}, \mathcal J_{2}, \mathcal J_{3})$ be a generalized almost hypercomplex structure on a Courant algebroid
$E$ and  $\partial_{\mathbb{H}}: E^{*} \otimes \Lambda^{1,1}_{\mathbb{H}} E^{*} \rightarrow \Lambda^{3} E^{*}$
the algebraic torsion map defined by (\ref{linear-torsion}), where 
$$
\mathrm{ad}_{(\mathcal J_{i})} (E) = \Lambda^{1,1}_{\mathbb{H}} E^{*}:= \{\alpha \in \Lambda^{2} E^{*},\ \alpha (\mathcal J_{i}u, \mathcal J_{i}v)
=\alpha (u, v) ,\ u, v\in E,\ i=1,2,3\} .
$$

\begin{lem}  The endomorphism $P\in \Gamma (\mathrm{End}( \Lambda^{3}E^{*}))$ defined by
\begin{equation}\label{proj-P}
P:= \frac{2}{3} \sum_{i=1}^{3} \Pi_{\mathcal J_{i}}
\end{equation}  
is a projector with  $ \ker P= \mathrm{im}\, \partial_{\mathbb{H}}.$ 
In particular, $\mathrm{im}\, P$ is a complement of $\mathrm{im}\, \partial_{\mathbb{H}}$ in $\Lambda^{3}E^{*}.$ 
\end{lem}

\begin{proof}  For any generalized almost complex structure $\mathcal J$, 
define the endomorphism $\Pi^{0,2}_{\mathcal J}$ of $\Lambda^{2}E^{*}\otimes E$ by 
$ \langle
(\Pi^{0,2}_{\mathcal{J}}\alpha )(u, v), w\rangle  = (\Pi_{\mathcal J }\alpha )(u, v, w)$,  where 
$\alpha\in \Lambda^{2}E^{*}\otimes E^*$,
$u, v, w\in E$
and 
$(\Pi_{\mathcal J }\alpha )(u, v, w)$ 
is given by  (\ref{PiJ:eq}).
Then  $\Pi^{0,2}_{\mathcal J}$ coincides with the operator  defined  by relation (5) of \cite{gauduchon}.  We obtain
that $(P\alpha  )(u, v, w) = \langle p(\alpha) (u, v), w\rangle$ where $p =\frac{2}{3} \sum_{i=1}^{3} \Pi^{0,2}_{\mathcal J_{i}}
\in \Gamma (\mathrm{End} (\Lambda^{2} E^{*}\otimes E))$ 
is the map from Lemma 1 of \cite{gauduchon}. As $p$ is a projector, $P$ is also a projector. 
From Lemma \ref{alg-complex} ii), $\mathrm{im}\,  \partial_{\mathbb{H}}  \subset \ker P. $ Let
\begin{equation}\label{p-i}
\tilde{\pi} : \Lambda^{3} E^{*} \rightarrow E^{*}\otimes \Lambda^{1,1}_{\mathbb{H}} E^{*},\ 
(\tilde{\pi}\alpha ) (u, v, w):= \alpha ( u, v, w) +\sum_{i=1}^{3} \alpha (u,\mathcal J_{i}v, \mathcal J_{i}w).
\end{equation}
 It is straightforward to check that 
\begin{equation}\label{pi-h}
\partial_{\mathbb{H}}\circ\tilde{ \pi } = 6 ( \mathrm{Id}_{\Lambda^{3} E^{*}} - P),
\end{equation}
which implies  $\ker P \subset \mathrm{im}\,  \partial_{\mathbb{H}}$ and thus 
$\mathrm{im}\,  \partial_{\mathbb{H}} =  \ker P. $ As $P$ is a projector, 
$\mathrm{im}\,  \partial_{\mathbb{H}} =  \ker P$ is a complement of $\mathrm{im}\, P$ in $\Lambda^{3}E^{*}.$ 
\end{proof}

As in the previous section, we will identify $(\Lambda^{3} E^{*})/  (\mathrm{im}\, \partial_{\mathbb{H}} )$ with $\mathrm{im}\, P$ and consider
the intrinsic torsion $t_{(\mathcal J_{i})}$ of $(\mathcal J_{1}, \mathcal J_{2}, \mathcal J_{3} )$ as a section of
$\mathrm{im}\, P. $

\begin{cor}
 The intrinsic torsion  $t_{(\mathcal J_{i})}$ of $(\mathcal J_{1}, \mathcal J_{2}, \mathcal J_{3})$ is given by
$\frac{1}{6} \sum_{i=1}^{3} N_{\mathcal J_{i}}.$  
\end{cor}

\begin{proof}
Let $D$ be a hypercomplex connection. 
By Corollary  \ref{pi-d} we have that $\Pi_{\mathcal J_{i}} (T^{D}) 
= \frac{1}{4} N_{\mathcal J_{i}}$. So $t_{(\mathcal J_{i})}=P(T^D)=\frac16\sum_{i=1}^{3} N_{\mathcal J_{i}}$.
\end{proof}

\subsection{Integrability using torsion-free generalized connections}

Our aim in this section is to prove the next theorem.

\begin{thm}\label{hyper:thm}  A generalized almost hypercomplex structure  $(\mathcal J_{1},\mathcal J_{2}, \mathcal J_{3})$ 
on a Courant algebroid $E$ 
is
integrable if and only if there is a torsion-free generalized  connection $D$ such that $D\mathcal J_{i}=0$
for all $i=1,2,3.$
\end{thm}

Part of the statement of Theorem \ref{hyper:thm} is obvious: if there is a torsion-free 
generalized connection which preserves 
all $\mathcal J_{i}$, then $\mathcal J_{i}$ are integrable from Theorem \ref{complex:thm}. 
The converse statement is proved in the next proposition.

\begin{prop}\label{gen-hypercomplex} Let $(\mathcal J_{1}, \mathcal J_{2}, \mathcal J_{3})$ be a generalized almost hypercomplex structure 
and $D$ a generalized connection on $E$, such that $D\mathcal J_{1} =0.$ 
Define 
\begin{equation}
D^{(1)}:= D -\frac{1}{2} \mathcal J_{2} D  \mathcal J_{2} 
\end{equation}
and 
\begin{equation}
\tilde{D}:= D^{(1)} - \frac{1}{6}\tilde{ \pi } ( T^{D^{(1)}}),
\end{equation}
where  $\tilde{\pi}$ is the map (\ref{p-i}). 
Then ${D}^{(1)}$ and $\tilde{D}$ are  generalized hypercomplex connections. Moreover, 
$$
T^{\tilde{D}} (u, v, w)= \frac{1}{6} \sum_{i=1}^{3} N_{\mathcal J_{i}} (u, v, w).
$$ 
In particular, if $(\mathcal J_{1}, \mathcal J_{2}, \mathcal J_{3})$  is a generalized hypercomplex structure, then $\tilde{D}$ is torsion-free
(and hypercomplex). 
\end{prop}

\begin{proof} 
The existence of a generalized connection $D$ with $D\mathcal J_{1} =0$ follows from the proof of
Lemma \ref{easy} (take any generalized connection, say $D^{(0)}$,  and define
$D:= D^{(0)}-\frac{1}{2}\mathcal J_{1} D^{(0)}\mathcal J_{1}$). 
The same argument shows that 
$D^{(1)}$ is a generalized connection which preserves  $\mathcal J_{2}$.  As
$D\mathcal J_{1} =0$
and $\mathcal J_{2}$ anti-commutes with $\mathcal J_{1}$, we obtain  that 
$D_{u} \mathcal J_{2}$  anti-commutes with $\mathcal J_{1}$ as well, for any $u\in E.$   
Then 
$$
D^{(1)}_{u} \mathcal J_{1} = D_{u} \mathcal J_{1}-\frac{1}{2} [ \mathcal J_{2}D_{u}\mathcal J_{2}, \mathcal J_{1}]=
 -  \frac{1}{2} [ \mathcal J_{2}D_{u}\mathcal J_{2}, \mathcal J_{1}]=0. 
$$
As $D^{(1)}\mathcal J_{1} = D^{(1)}\mathcal J_{2} =0$ also 
$D^{(1)}\mathcal J_{3} =0$  and $D^{(1)}$ is hypercomplex. This proves the statement on $D^{(1)}.$\

It remains to prove the statements on $\tilde{D}.$  Let $\eta := -\frac{1}{6}
\tilde{ \pi}  (  T^{D^{(1)}})$. With this notation, 
$\tilde{D} = D^{(1)} + \eta $. 
Since $\mathrm{im}\, \tilde{ \pi } \subset E^{*} \otimes \Lambda^{1,1}_{\mathbb{H}} E^{*}$
and $D^{(1)}\mathcal J_{i}=0$, also   $\tilde{D}\mathcal J_{i}=0$ for all $i=1,2,3$, i.e. $\tilde{D}$ is hypercomplex.  
The torsion of $\tilde{D}$  is given by 
\begin{equation}\label{torsion-nj}
T^{\tilde{D}} = T^{D^{(1)}} - \frac{1}{6} (\partial_{\mathbb{H}}\circ  \tilde{\pi})  (T^{D^{(1)}})
= P(T^{D^{(1)}}) = \frac{1}{6} \sum_{i=1}^{3} N_{\mathcal J_{i}},
\end{equation}
where in the second equality we used (\ref{pi-h}) and in the third 
equality we used 
(\ref{courant-torsion-new}) (which holds since  $D^{(1)}$ is hypercomplex). 
\end{proof}

\section{Generalized almost Hermitian structures: integrability and  torsion-free generalized connections}\label{sect-hermitian}

In this section we characterize the integrability of generalized almost Hermitian structures  using Levi-Civita connections
(see Theorem \ref{integr-kahler} below).  
We begin with the following simple lemma.

\begin{lem}\label{add} Let $(G, \mathcal J )$ be a generalized K\"{a}hler structure on $E$.  Then 
\begin{equation}\label{ajutatoare-0}
[ u, \mathcal J v]_{+} -\mathcal J [u, v]_{+} =0,\ \forall u\in \Gamma( E_-),\ v\in\Gamma ( E_+).
\end{equation}
Above we denoted by $e_{\pm}$  the $E_\pm$-components of a vector 
$e\in E$  in the decomposition 
$E = E_+ \oplus  E_-$ determined by   $G$. Similarly, 
\begin{equation}\label{ajutatoare-1}
[ u, \mathcal J v]_{-} -\mathcal J [u, v]_{-} =0,\ \forall u\in \Gamma( E_+),\ v\in\Gamma ( E_-).
\end{equation}

\end{lem}

\begin{proof} 
Relation (\ref{ajutatoare-0}) is equivalent to
\begin{equation}\label{ajutatoare-p}
[u, v]_{+} \in \Gamma ( (E_+)_{\mathbb{C}}\cap L_{1}),\ \forall u\in\Gamma (( E_-)_{\mathbb{C}}),\ v\in\Gamma ((E_+)_{\mathbb{C}}
 \cap L_{1}), 
\end{equation}
where we have denoted by $L_{1}$ the $(1,0)$-bundle 
of $\mathcal J_{1} =\mathcal J.$  
Remark that $(E_+ )_{\mathbb{C}}\cap L_{1} = (E_+)_{\mathbb{C}} \cap L_{2}$, 
where $L_{2}$ is the $(1,0)$-bundle of $\mathcal J_{2}=  G^{\mathrm{end}}\mathcal J$  
(since  $\mathcal J_{1} = \mathcal J_{2}$ on $E_+$).   
In (\ref{ajutatoare-p}) we distinguish two cases: a)  $u\in \Gamma (L_{1}\cap (E_-)_{\mathbb{C}})$; b)  $u\in \Gamma (\bar{L}_{1}\cap (E_-)_{\mathbb{C}})$.  
In case a), relation (\ref{ajutatoare-p}) follows from the integrability of $\mathcal J$. In case b), relation (\ref{ajutatoare-p}) follows from the  integrability of $\mathcal J_{2}$.
\end{proof}

\begin{rem}{\rm When the Courant algebroid is exact the above lemma can  be proved  using 
the Bismut connection for generalized K\"{a}hler
 structures, constructed in \cite{gualtieri-bismut}. More precisely, for  a generalized K\"{a}hler
structure $(G, \mathcal J)$ on an exact Courant algebroid $E$, there is a unique generalized connection $D$ 
(called in \cite{gualtieri-bismut}  the  {\cmssl Bismut connection}),  such
that $DG=0$, $D\mathcal J =0$,  and whose torsion is of type $(2,1)+(1,2)$ with respect to 
$\mathcal J .$ The expression of $D$ is given in Theorem 3.1 of \cite{gualtieri-bismut}. Its mixed components
$D_{u}v$ and $D_{v}u$, for $u\in \Gamma (E_+)$ and $v\in \Gamma (E_-)$,  are 
$D_{u}v:= [u, v]_{-}$ and $D_{v}u = [v,u]_{+}.$ 
Relation (\ref{ajutatoare}) follows from $D\mathcal J =0$. }
\end{rem}

 \begin{thm}\label{integr-kahler}  A generalized almost Hermitian structure $(G,\mathcal J )$ on a Courant algebroid  $E$ is generalized
K\"{a}hler if and only if there is a Levi-Civita connection of $G$ which preserves $\mathcal J .$\end{thm}

\begin{proof}
In one direction the statement is obvious: if there is a Levi-Civita connection $D$ of $G$ which preserves $\mathcal J$, 
then it preserves also $\mathcal J_{2} = G^{\mathrm{end}} \mathcal J$ and we deduce that  both $\mathcal J$ and $\mathcal J_{2}$ are integrable 
(from Theorem \ref{complex:thm}, because $D$ is torsion free). We obtain that $(G, \mathcal J)$ is generalized K\"{a}hler.
The converse statement follows from the next theorem. 
\end{proof}
\begin{thm} Let $(G,\mathcal J )$ be a generalized
K\"{a}hler structure on a Courant algebroid $E$ and $D$ a 
Levi-Civita connection. Then the generalized connection  
\[\tilde{D} =D -\frac12 \mathcal{J}D\mathcal{J}-\frac14 \{ A^{\mathrm{sym}},\mathcal{J}\},\] 
is a torsion-free generalized connection compatible 
with $(G,\mathcal J )$. 
\end{thm}
\begin{proof}From Proposition 
\ref{essential} we know that $\tilde{D}$ is torsion-free and complex ($\tilde{D}\mathcal J =0$). 
So it suffices to show that $\tilde{D}G=0$. Since $DG=0$ and $\mathcal{J}$ is $G$-skew-symmetric,  
the anticommuting endomorphisms $\mathcal{J}$ and $D_u\mathcal{J}$ ($u\in E$) are both $G$-skew and, hence, 
$\mathcal{J}D_u\mathcal{J}$ as well. 
We conclude that the generalized connection $D -\frac12 \mathcal{J}D\mathcal{J}$ preserves $G$. 

It remains to check that 
$\{ A^{\mathrm{sym}}_u,\mathcal{J}\}$ is $G$-skew-symmetric for all $u\in E$. 
We know that it is skew-symmetric with respect to $\langle \cdot ,\cdot \rangle$, since both 
$D -\frac12 \mathcal{J}D\mathcal{J}$ and $\tilde{D}$ are generalized connections. 
Therefore it suffices to check that $\{ A^{\mathrm{sym}}_u,\mathcal{J}\}$ preserves the 
decomposition $E=E_+\oplus E_-$. Since $\mathcal{J}$ does, we only need to 
check that $\langle A^{\mathrm{sym}}_uE_+,E_-\rangle =0$.  We check that $\langle A_uE_\pm,E_\mp\rangle =0$, 
which implies the latter. Let $v\in \Gamma (E_\pm)$, $w\in \Gamma (E_\mp)$. If $u\in E_\pm$, then 
$\langle A_uv,w\rangle = \langle (D_v\mathcal{J})u,w\rangle=0$,
because  $D_v\mathcal{J}$ preserves the decomposition $E=E_+\oplus E_-$ (for all $v\in E$).
For $u\in E_\mp$, we use  a) that  $D$ has zero torsion to express derivatives by brackets with the help of  the previous equation $\langle (D_{E}\mathcal{J})E_\pm , E_\mp\rangle=0$ and the property that $D$  
preserves the subbundles $E_\pm$    
and b) Lemma \ref{add}:  
\[ \langle A_uv,w\rangle = \langle (D_v\mathcal{J})u,w\rangle= \langle D_v(\mathcal{J}u)-\mathcal{J}D_vu,w\rangle \stackrel{a)}{=}\langle [v,\mathcal{J}u]-\mathcal{J}[v,u],w\rangle\stackrel{b)}{=}0.\qedhere\]
\end{proof}
Note that the equations $\langle (D_{E_\pm}\mathcal{J})E ,E_\mp\rangle=\langle (D_{E_\pm}\mathcal{J})E_\mp,E\rangle= 0$,  
established in the proof, can be also written as: 
\begin{equation}\label{help:eq} (D_{E_\pm}\mathcal{J})E_\mp =0, \quad (D_{E_\pm}\mathcal{J})E_\pm \subset E_\pm .
\end{equation}

\section{Generalized almost hyper-Hermitian structures: integrability and torsion-free generalized connections}\label{sect-hyperhermitian}

\begin{thm}\label{integr-hk} A generalized almost hyper-Hermitian structure $(G, \mathcal J_{1}, \mathcal J_{2}, \mathcal J_{3})$ on a Courant algebroid
$E$ is generalized hyper-K\"{a}hler if and only if there is a Levi-Civita connection $D$ of $G$ which satisfies $D\mathcal J_{i}=0$, for all 
$i=1,2,3.$   
\end{thm}

\begin{proof}
If there is a Levi-Civita  connection $D$ of $G$ which is hypercomplex,  then 
$(G, \mathcal J_{i})$ is generalized K\"{a}hler (see Theorem   
\ref{integr-kahler}).   

For the converse statement, let $(G,\mathcal J_{1}, \mathcal J_{2}, \mathcal J_{3})$ be a generalized hyper-K\"{a}hler structure and $D$ a generalized Levi-Civita connection of $G$ with $D\mathcal J_{1}=0$ (which exists, by Theorem \ref{integr-kahler}). We will show that the generalized connection $\tilde{D}$ constructed in Proposition \ref{gen-hypercomplex},  starting from $D$,   is a Levi-Civita connection of $G$ which preserves  the 
$\mathcal{J}_{i}$. 
Define the generalized connections 
\begin{equation}
D_{u}^{(1)}:= D_{u} -\frac{1}{2}  \mathcal J_{2} D_{u}\mathcal J_{2}
\end{equation}
and 
\begin{equation}
\tilde{D}_{u} := D^{(1)}_{u} + \eta_{u} = D_{u}  -\frac{1}{2} \mathcal J_{2} 
D_{u}\mathcal J_{2} + \eta_{u} ,
\end{equation}
where $\eta :=-  \frac{1}{6}\tilde{ \pi} ( T^{D^{(1)}})$. 
From Proposition \ref{gen-hypercomplex},  $D^{(1)}$ and $\tilde{D}$ are  hypercomplex 
and  $\tilde{D}$ is  torsion-free.   We claim that $\tilde{D}G=0$, which proves the theorem.

Note first that $D^{(1)}G=0$, since $DG=0$ and  $\mathcal{J}_2D_u\mathcal{J}_2$ is skew-symmetric  with respect to $G$. So it suffices to show that $\eta(u,v,w)=0$ for all $u\in E$, $v\in E_\pm$, $w\in E_\mp$.
Note  that
$\eta  = -\frac{1}{6} \tilde{\pi} ( T^{D^{(1)}}) $  has the following expression:
\begin{align*}
\nonumber &\eta (u, v, w) =\\
\nonumber& \frac{1}{12}  \langle (D_{\mathcal J_{1}v}\mathcal J_{2})(\mathcal J_{3} w) 
+ ( D_{\mathcal J_{2}v}\mathcal J_{2})(w) -  (D_{\mathcal J_{3}v}\mathcal J_{2})(\mathcal J_{1}w) 
+ \mathcal J_{2}  (D_{v}\mathcal J_{2})(w),u\rangle  \\
\nonumber& - \frac{1}{12} \langle  \mathcal J_{3} (D_{\mathcal J_{1}w}\mathcal J_{2})(u) 
-   (D_{\mathcal J_{2}w}\mathcal J_{2})(u) -  \mathcal J_{1}( D_{\mathcal J_{3}w}\mathcal J_{2})(u) 
-  \mathcal J_{2}  (D_{w}\mathcal J_{2})(u),v\rangle ,
\end{align*}
for any $u, v, w\in \Gamma (E)$. For all $u\in E$, $v\in E_\pm$, $w\in E_\mp$, each summand belongs either 
to the set $\langle (D_{E_\pm}\mathcal{J}_\alpha ) E_\mp ,E\rangle$ or to $\langle (D_{E_\mp}\mathcal{J}_\alpha )E,E_\pm\rangle$, which  both reduce to zero by (\ref{help:eq}).
\end{proof}

\vspace{1cm}

\begin{center}
\part[Second Part: Dirac generating operators]{}
\end{center}

\section{The space of  local Dirac generating operators}\label{sect-dirac-add}

Let $E$ be an oriented Courant algebroid with anchor  $\pi : E \rightarrow TM$,
 scalar product  $\langle\cdot , \cdot \rangle$ and Dorfman bracket  $[\cdot , \cdot ]$. 
We assume that $\langle\cdot , \cdot \rangle$ is of neutral signature $(n, n)$. 
We denote by $\mathrm{Cl}(E)$ the bundle of  Clifford algebras over $(E,\langle \cdot , \cdot \rangle)$ with the Clifford relation 
$e^2 = \langle e ,e\rangle$, $e\in E$. Let $S\rightarrow M$ be a real vector bundle of irreducible $\mathrm{Cl}(E)$-modules. 
We will call $S$ a {\cmssl spinor bundle} over $\mathrm{Cl}(E)$.
The representation of $\mathrm{Cl}(E)$ on $S$, denoted by 
\begin{equation} \label{iso:eq} \gamma : \mathrm{Cl}(E) \rightarrow \mathrm{End} (S),\quad a\mapsto \gamma (a) :=\gamma_a,\end{equation}
is an isomorphism of algebra  bundles. To simplify notation, we shall sometimes write $as$ for the Clifford action $\gamma_{a}s$ 
of $a\in \mathrm{Cl}(E)$ on $s\in S$.

Recall that the Clifford algebra bundle $\mathrm{Cl}(E)$ is $\mathbb{Z}_2$-graded.  We denote the subbundle of $\mathrm{Cl}(E)$  of degree $i\in \mathbb{Z}_2$ by
$\mathrm{Cl}^i(E)$. 
Since $\langle\cdot, \cdot\rangle$ has neutral signature, 
the bundle $S$ has a compatible $\mathbb{Z}_2$-grading denoted by $S=S^{0} \oplus S^{1}$,  where 
$S^{0} = \frac{1}{2} \gamma_{(1+\omega )} S$ and $S^{1} =\frac{1}{2} \gamma_{(1-\omega )} S$, with $\omega
=e_{1}\wedge \cdots \wedge e_{2n}$ the volume  form of $E$,  
determined by a positive oriented orthonormal basis of $E$, and considered as an element of $\mathrm{Cl}(E)$
(see e.g.\  Proposition 3.6 of \cite{LM}).   An argument analogous to the proof of Proposition 5.10 of \cite{LM} shows that 
the  $\mathrm{Cl}^0(E)$-submodules $S^{0}$ and $S^{1}$ are pointwise inequivalent and irreducible. There is an induced $\mathbb{Z}_2$-grading on $\mathrm{End} (\Gamma (S))$ and, in particular, on the algebra of differential operators on $S$, which includes
$\Gamma (\mathrm{End}\, S)\subset \mathrm{End} (\Gamma (S))$ as the subalgebra of operators of $0$-th order. We will denote by 
\[ [ A,B] = AB -(-1)^{\deg A \deg B} BA\] 
the {\cmssl super commutator} of two homogeneous elements $A, B \in \mathrm{End} (\Gamma (S))$, where $\deg A$ stands for the degree of $A$.

\begin{defn} \label{Dgo:def} A first order odd differential operator 
$\slashed{d}$ on  a spinor bundle $S$ over $\mathrm{Cl}(E)$  is called a 
{\cmssl Dirac generating operator} for $E$ if for all $f\in C^\infty (M)$ and $e, e_1,e_2 \in \Gamma (E)$, 
\begin{enumerate}
\item[i)] $[[\di , f],\g_e]] = \pi (e)(f)$,
\item[ii)] $[[\di ,\g_{e_1}],\g_{e_2}] = \g_{[e_1,e_2]}$ and 
\item[iii)] $\di\,^2\in C^\infty(M)$.
\end{enumerate}
\end{defn}
Note that given $(E,\langle \cdot , \cdot \rangle )$ and $\di$ one can reconstruct 
the full Courant algebroid structure from i) and ii). This is why the operator $\slashed{d}$ is called
{\cmssl generating}.

\begin{prop}\label{prop-space} Suppose that there is a Dirac generating operator $\di$ for $E$ on $S$.
Then the set of Dirac generating operators for $E$ on $S$ has the structure of an affine space modelled on the vector space
\begin{equation} \label{Vd:eq}  V_\di := \{ e\in \Gamma (E) \mid [\di , \g_e] \in C^\infty (M)\}.\end{equation}
In particular,  $V_\di$ is independent of the choice of Dirac generating operator $\di$.
\end{prop}

\begin{proof} We first check that $\di\,' := \di + \g_{e_0}$ is a Dirac generating operator for all
$e_0\in V_\di$. Since $[\g_{e_0},f]=0$ and $[\g_{e_0},\g_{e_1}] = 2 \langle e_0,e_1\rangle \in C^\infty(M)$ the properties
i) and ii) in Definition \ref{Dgo:def}  for $\di$ immediately imply the same properties for $\di\,'$. Finally, the equation $(\di\,')^2=\di\,^2 + [\di , \g_{e_{0}}] + \langle e_0,e_0\rangle$
shows that property iii) holds for $\di\,'$ if it holds for $\di$ and $e_0\in V_\di$.

Conversely, we show that given Dirac generating operators $\di$ and $\di\,'$, there exists $e_0\in V_\di$ such that 
$L:= \di\,' - \di = \g_{e_0}$. We first observe that $[L,f]$ is a $0$-th order operator of odd degree for all $f\in C^\infty (M)$. 
By property i) it satisfies $[[L,f],\g_e]=0$ for all $e\in \Gamma (E)$. This implies that $[L,f]$ commutes with 
$\mathrm{Cl}^0(E)$. Being of odd degree, it interchanges $S^{0}$ and $S^{1}$  and we deduce that    $[L,f]=0$ since 
the irreducible $\mathrm{Cl^0(E)}$-modules $S^0$ and $S^1$ are inequivalent. 
This shows that the odd differential operator $L$ is of $0$-th order, that is $L=\g_a$ for some $a\in \Gamma(\mathrm{Cl}^1(E))$. 

Next we consider the even $0$-the order operator 
\[L':= [L,\g_{e}]=\gamma (ae+e a),\] 
$e\in \Gamma (E)$. It commutes with  $\mathrm{Cl}^1(E)$, 
in virtue of property ii), and is hence a scalar operator
(since $\langle \cdot , \cdot \rangle$ has  neutral signature). 
We conclude that 
$ae+ea$ is a scalar in $\mathrm{Cl}(E)$,  for any $e\in E$. This easily implies 
that $a =: e_0\in \Gamma (E)$, by a straightforward computation in which $e$ runs through the elements of an orthonormal frame. 
Now property iii) implies that $e_0\in V_\di$.

The last claim is now obvious: if $\di$ and  $\di\,' $ are two Dirac generating operators then $\di\,'  =  \di + \gamma_{e}$ for $e\in V_{\di}\subset \Gamma (E)$
which implies 
$V_{\di\,'  } = V_{\di }.$ 
\end{proof}

The next theorem is our main result in this section.

\begin{thm}[Alekseev-Xu] \label{ex:thm}
Let $E$ be a regular Courant algebroid with scalar product of neutral signature.  
Every spinor bundle $S$ over $\mathrm{Cl}(E)$ 
admits  locally  a Dirac generating operator. 
\end{thm}

We divide the proof of Theorem \ref{ex:thm} into several steps.  Let $D$ be a generalized connection on $E$. 
The existence of $D$  is ensured 
by Example \ref{ext:exa}. 
The generalized connection $D$ induces an $E$-connection in $\mathrm{Cl}(E)$, which we denote again by $D$. 
Next we choose an $E$-connection $D^S$ on $S$  compatible with $D$ in the sense that
\[ D^S_e(as) = (D_ea)s + a D^S_es,\]
for all $e\in \Gamma (E)$, $a\in \Gamma (\mathrm{Cl}(E))$, $s\in \Gamma (S)$. 

The existence of such a connection $D^S$ can be shown as follows. The bundle $(E,\langle \cdot ,\cdot \rangle)$ 
admits locally a spin structure. To this structure we can associate a spinor bundle $\Sigma$ over some domain $U\subset M$.   
The $E$-connection $D$ on $E$ induces an $E$-connection $D^\Sigma$ on $\Sigma$. The connection form of 
$D^\Sigma$ with respect to a local trivialization of the spin structure is one half of the connection form
of $D$ with respect to the corresponding local orthonormal frame of $E$. (Both forms can be considered as local sections of
$E^*\otimes \mathfrak{so}(E)$, after identifying $\mathfrak{spin}(E)\cong \mathfrak{so}(E)$
via the adjoint representation $\mathrm{ad} : \mathfrak{spin} (E) \rightarrow \mathfrak{so}(E)$,
$\mathrm{ad}_{u}v := uv - vu$.)
In more concrete terms,  let $(e_{i})$ be an  orthonormal frame of $E\vert_{U}$ and 
$(\sigma_{\alpha })$ a frame of $\Sigma$ such that 
$$
e_{i}  \sigma_{\alpha } = \sum_{\beta} C_{i\alpha}^{\beta} \sigma_{\beta}, 
$$
where $C_{i\alpha}^{\beta}$ are constants. Let  $(\omega_{ij})$  be the  (skew-symmetric)  matrix of $1$-forms defined by 
\begin{align*}
\nonumber D_{v} (e_{k}) & = -\epsilon_{k} \sum_{j} \omega_{jk} (v) e_{j}= 2\sum_{j<p} \omega_{pj} (v) (\epsilon_{p} e_{p}^{*} \otimes e_{j} 
-\epsilon_{j} e_{j}^{*}\otimes e_{p} ) (e_{k})\\
\nonumber& = 2\sum_{j<p} \omega_{pj} (v) (e_{p}\wedge e_{j}) (e_{k}), 
\end{align*}
where $\epsilon_{k} = \langle e_{k}, e_{k}\rangle \in \{ \pm 1\}$.    Then
$D^{\Sigma}_{e} (\sigma_{\alpha} ) :=\frac{1}{2} \sum_{i<j} \omega_{ji} (v) e_{i} e_{j}  \sigma_{\alpha}$ is compatible with $D\vert_{U}$.
(Note that the element $\frac12 e_{i} e_{j}\in \mathfrak{spin}(E) \subset \mathrm{Cl(E)}$ ($i\neq j$) acts under the adjoint representation 
as $\epsilon_j e_j^* \otimes e_i -\epsilon_ie_i^*\otimes e_j\in  \mathfrak{so}(E)$ and the latter corresponds to the bivector $e_j\wedge e_i \in \Lambda^2E$.)  
Since $(E,\langle \cdot ,\cdot \rangle)$ has neutral signature, $S|_U$ and $\Sigma$ differ only by 
a real line bundle $L$ over $U$: $S|_U \cong \Sigma \otimes L$. Choosing an $E$-connection in $L$,
we obtain an $E$-connection $D^{S,U}$ in $S_U$ by taking the tensor product with the connection $D^\Sigma$. 
By considering an open covering $(U_i)$ of $M$ and a corresponding partition of unity, we can glue
the $E$-connections $D^{S,U_i}$ to an $E$-connection $D^S$.

The $E$-connection $D^S$ gives rise to a first order differential operator on $S$, 
which we call the {\cmssl Dirac operator}: 
\begin{equation}\label{dirac-def} 
\slashed{D}^S = \frac12 \sum_{i=1}^{2n} \tilde{e}_iD_{e_i}^S,
\end{equation}
where $(e_i)$ is any local frame of $E$ and $(\tilde{e}_i)$ is the metrically dual frame, that is 
$\langle e_i,\tilde{e}_j\rangle = \delta_{ij}$.

\begin{lem}\label{c1-2} For any generalized connection $D$ and compatible $E$-connection $D^{S}$, the operator 
\begin{equation} \label{ansatz:eq}
\slashed{d} = \slashed{D}^S + \frac14 \g_T
\end{equation} 
satisfies conditions i) and ii) from Definition  \ref{Dgo:def}.
Above  $T \in \Gamma(\wedge^3E^*)\cong \Gamma(\wedge^3E)\subset \Gamma (\mathrm{Cl}(E))$ denotes the torsion of $D$. 
\end{lem}

\begin{proof} We compute  $[[\di , f],\g_v]] = [[\slashed{D}^S,f],\g_v]$ for $f\in C^\infty (M)$ and $v\in \Gamma (E)$. We find 
\[ [\slashed{D}^S,f] = \frac12 \sum_{i} \pi (e_i)(f)\g_{\tilde{e}_i},\quad  [[\slashed{D}^S,f],\g_v] = \sum_{i} \pi (e_i)(f)\langle\tilde{e}_i, v\rangle = 
 \pi (v) (f).\]
 This shows that i) in Definition \ref{Dgo:def} is satisfied.

Next we compute 
\[ [\slashed{D}^S,\g_v]= \frac12 \sum_{i} \g_{\tilde{e}_i}\g_{D_{e_i}v} + D_v^S,\quad 
[[\slashed{D}^S,\g_v],\g_w] = \frac12 \sum_{i} [\g_{{\tilde{e}_i}}\g_{D_{e_i}v},\g_w] + \g_{D_vw},\] 
where $w\in \Gamma (E)$. A simple calculation in the Clifford algebra shows that 
for all $u, v, w\in E$:
\[ uvw -wuv = -2\langle u,w\rangle v + 2 \langle v,w\rangle u.\]
So 
\[ \frac12 \sum_{i} [\g_{\tilde{e}_i}\g_{D_{e_i}v},\g_w] = -\sum_{i} \langle \tilde{e}_i,w\rangle \g_{D_{e_i}v} +
\sum_{i} \langle D_{e_i}v,w\rangle \g_{\tilde{e}_i} = -\g_{D_wv}+
\sum_{i}\langle D_{e_i}v,w\rangle \g_{\tilde{e}_i}\]
and using that $D$ has torsion $T$ we obtain 
\begin{eqnarray*} [[\slashed{D}^S,\g_v],\g_w] &=& \g_{D_vw-D_wv} +
\sum_{i}\langle D_{e_i}v,w\rangle \g_{\tilde{e}_i}\\ 
&=&  \g_{T(v,w)}+\g_{[v,w]} -\gamma_{(Dv)^{*}w} +
\sum_{i} \langle D_{e_i}v,w\rangle \g_{\tilde{e}_i}\\ 
&=& \g_{T(v,w)}+\g_{[v,w]}.\end{eqnarray*} 
A simple calculation in the Clifford algebra shows that (for any $3$-form $T$)
\[ [[\g_T,\g_v],\g_w] = -4\g_{T(v,w)}.\] 
Thus we can conclude that 
\[ [[\di , \g_v],\g_w] = [[ \slashed{D}^S +\frac{1}{4}\gamma_{T}, \g_v],\g_w] = \g_{[v,w]}.\]
So ii) in Definition \ref{Dgo:def} is also satisfied. 
\end{proof}

In order to conclude the proof of Theorem \ref{ex:thm} 
we therefore need to find  a locally defined generalized connection $D$ on $E$ and a compatible $E$-connection 
$D^{S}$ on $S$  such that condition iii) from Definition  \ref{Dgo:def} holds as well.  
To analyze  condition iii) in Definition \ref{Dgo:def} we will use the following lemma, see  \cite{chen}. (When $\langle \cdot ,\cdot \rangle$ has arbitrary signature, the next lemma still holds with the only difference
that the restriction of $\langle \cdot ,\cdot \rangle$ to $\mathcal G$ does
not have neutral signature).

\begin{lem}\label{lem-chen}(\cite{chen}) Let $E$ be a regular Courant algebroid with scalar product $\langle\cdot , \cdot \rangle$ 
of neutral signature  and anchor $\pi  : E \rightarrow TM.$\

i)  The bundle  $\ker \pi\subset E$ is a coisotropic subbundle of $E$, that is $(\ker  \pi )^\perp \subset \ker \pi$.\ 

ii) The bundle $E$ decomposes  as $E =\ker \pi  \oplus \mathcal F$ where $\mathcal F$ is isotropic.\

iii) The bundle $\ker \pi$ decomposes  as $\ker \pi = (\ker \pi )^{\perp}
\oplus   \mathcal G$ where $\mathcal G$ is orthogonal to $\mathcal F$.\

iv) The decomposition $E = \left( (\ker \pi )^{\perp} \oplus \mathcal F\right) \oplus \mathcal G$ is orthogonal with respect to
$\langle\cdot , \cdot \rangle$. The restrictions of   $\langle\cdot , \cdot \rangle$  to the two factors
 $ (\ker \pi )^{\perp} \oplus \mathcal F$ and $\mathcal G$ have neutral signature.
\end{lem}

The above  lemma implies that for any $U\subset M$ sufficiently small,  the bundle $E\vert_{U}$ admits a  frame $(p_{a}, q_{a})$, $a=1,\ldots , n$, 
such that $(p_{a})$, $a=1,\ldots , n$,  span a maximally  isotropic subbundle $\mathcal{P}$ of $\ker \pi$, 
$(p_{a})$, $a=1,\ldots , s$,  span $ (\ker \pi )^{\perp}$,   $(q_{a})$, $a=1,\ldots , n$,  span  a maximally isotropic subbundle  $\mathcal{Q}$ of $E$, 
$q_{a}\in \ker \pi$ for any $a\geq s+1$,
$\langle p_{a}, q_{b} \rangle =\delta_{ab}$  and  $[\pi (q_{a}), \pi (q_{b}) ]=0$ for any $a$, $b$.
For the latter condition we are using that the image of 
$\pi$ is an integrable  distribution on $M$  (by the axiom C2) in Definition \ref{def-courant}).
More precisely,
using Lemma \ref{lem-chen} iv),   this basis can be constructed in the following way: start with any basis $q_{a}$, 
$a=1,\cdots , s$, of $\mathcal F$ such that $[\pi (q_{a}), \pi (q_{b}) ]=0$ for any $a$, $b$. Consider the basis
$p_{a}$, $a=1,\cdots , s$, 
of $(\ker \pi )^{\perp}$   such that $\langle p_{a}, q_{b} \rangle =\delta_{ab}$ for any $a$, $b$. Finally, choose a basis
$p_{a}, q_{a}$, $a=s+1,\cdots , n$,  of $\mathcal G$ such that $\langle p_{a}, p_{b}\rangle =  \langle q_{a}, q_{b}\rangle =0$
and $\langle p_{a}, q_{b}\rangle = \delta_{ab}$ for any $a, b=s+1, \cdots , n.$\
The following inclusions summarize the properties of the two complementary 
maximally isotropic subbundles  $\mathcal{P}$ and $\mathcal{Q}$ of $E$: 
\[  (\ker \pi)^\perp \subset \mathcal{P}\subset \ker \pi,\quad  \mathcal{F} \subset \mathcal{Q} \subset \mathcal{F}^\perp = \mathcal{F} \oplus \mathcal{G}.\]

The next corollary will be useful in the proof of Lemma \ref{concl} below.

\begin{cor}\label{sum} For any $\sigma\in \Gamma ( \ker \pi )$,  $\sum_{a=1}^{n} \pi (q_{a}) \langle \sigma, p_{a}\rangle =0.$
\end{cor}

\begin{proof} Each  term in the above sum vanishes: if $a\le s$ then $p_{a} \in \Gamma ( (\ker \pi )^{\perp})$
and $\langle \sigma, p_{a}\rangle = 0.$ If $a\geq s+1$, then  $q_{a} \in \Gamma ( \ker \pi )$ and $\pi (q_{a})=0. $ 
\end{proof}

Let $\nabla$ be the  connection on $E\vert_{U}$ such that the frame 
$(p_a,q_a)$ is parallel. Then $\nabla$ is flat, preserves the scalar product $\langle\cdot , \cdot\rangle$ of  $E$ and $S\vert_{U}$ admits a flat connection $\nabla^S$ compatible
with $\nabla$. Then $\nabla$ induces a generalized connection $D$ on $E\vert_{U}$ and $\nabla^S$ induces an $E$-connection $D^S$ on $S\vert_{U}$ which is compatible with $D$.

The next lemma concludes the proof of Theorem \ref{ex:thm}.

\begin{lem}\label{concl} The operator (\ref{ansatz:eq}) constructed using $D$ and $D^{S}$ 
satisfies $\di^{2}\in C^{\infty}(U) $ and is a Dirac generating operator.\end{lem}

\begin{proof} 
The Dirac operator  $\slashed{D}^{S}$  has the  expression 
\[ \slashed{D}^S = \frac12 \sum_{a} \left( p_aD_{q_a}^S + q_aD_{p_a}^S\right) = \frac12 \sum_{a} p_aD_{q_a}^S ,\]
since $\pi (p_a) =0$. To see this it is sufficient to remark that the frame $(\tilde{q}_a,\tilde{p}_b)$ dual to the frame
$(q_a,p_b)$ is precisely $(p_a,q_b)$. Its square is given by
\begin{equation}\label{square-1}
 (\slashed{D}^S)^2 = \frac14 \sum_{a} p_ap_b D^S_{q_a}D^S_{q_b} =  \frac{1}{4} \sum_{a} \langle p_a,p_b\rangle \nabla^S_{\pi (q_a)}\nabla^S_{\pi (q_b)}=0,
 \end{equation}
 where we  used
 $\nabla^{S}p_{a}=0$,   the flatness of $\nabla^S$ and $[ \pi (q_{a}), \pi (q_{b})]=0$ for any $a$, $b$.

Next,  we compute 
$\slashed{D}^S\g_T + \g_T\slashed{D}^S=[\slashed{D}^S,\g_T]$. We  write 
the torsion $T$ of $D$ as 
$T=\frac16 \sum T^{ijk}e_{ijk}\in \mathrm{Cl}(E)$, where 
$(e_i)$ is a $D$-parallel orthonormal frame and $e_{ijk} := e_ie_je_k$.
The coefficients $T^{ijk}$ are given by
\begin{equation}\label{t-ijk}
T^{ijk}=T(\tilde{e}_i,\tilde{e}_j,\tilde{e}_k)=\epsilon_i\epsilon_j\epsilon_kT(e_i,e_j,e_k) = 
-\epsilon_i\epsilon_j\epsilon_k\langle [e_i,e_j],e_k\rangle,
\end{equation}
where 
$(\tilde{e}_{i})$ is the frame of $E\vert_{U}$ metrically dual to $(e_{i})$, i.e. 
$\tilde{e}_{i} = \epsilon_{i} e_{i}$ with 
$\epsilon_i = \langle e_i , e_i \rangle$. 
Using the abbreviation $\g_{e_ie_je_k} = \g_{ijk}$, we write  
\begin{eqnarray*} &&12[\slashed{D}^S,\g_T] = \sum_{i,j,k,\ell} [\g_{\tilde{e}_\ell}D^S_{e_\ell},T^{ijk}\g_{ijk}]\\ 
&&=\sum _{i,j,k,\ell} \pi(e_\ell)(T^{ijk})\g_{\tilde{e}_{\ell}}\g_{ijk} + 
\sum_{i,j,k,\ell } T^{ijk}\g_{\tilde{e}_\ell}\g (\underbrace{D_{e_\ell}e_{ijk}}_{=0}) + 
\sum_{i,j,k,\ell } T^{ijk}[\g_{\tilde{e}_\ell},\g_{ijk}] D^S_{e_\ell}.
\end{eqnarray*}
Note that, for any fixed $\ell$, 
$$
\sum_{i,j,k} T^{ijk}[\g_{\tilde{e}_\ell},\g_{ijk}] = 6 \sum_{j,k} T^{\ell jk}\g_{jk}  . 
$$
Hence 
\[ \sum_{i,j,k,\ell} T^{ijk}[\g_{\tilde{e}_\ell},\g_{ijk}] D^S_{e_\ell} = -6 \sum_{j,k} \epsilon_j\epsilon_k\g_{jk} D^S_{[e_j,e_k]}.\]
To compute the last term we choose the orthonormal frame  $(e_{i})$ to be 
\[ (e_i)_{i=1,\ldots ,2n}=\left(\frac{1}{\sqrt{2}} (p_a+q_a), \frac{1}{\sqrt{2}} (p_a-q_a)\right)_{a=1,\ldots ,n},\] 
where $(p_{a}, q_{a})$ is the frame constructed above. 
Then $\pi [e_j,e_k] = [ \pi e_j, \pi e_k] = 0$, because 
$[ \pi (q_{a}), \pi (q_{b} ) ]=0$.   This implies that $D^S_{[e_j,e_k]}=0$,
showing that 
\begin{equation}\label{square-2}
 [\slashed{D}^S,\g_T] = \frac{1}{12} \sum_{i,j,k,\ell } \pi(e_\ell)(T^{ijk})\g_{\tilde{e}_{\ell}}\g_{ijk}.
 \end{equation}
 From (\ref{square-1}) and (\ref{square-2}),  we obtain 
\begin{equation}
{\di}\,^2=  \frac{1}{4} [ \slashed{D}^{S}, \gamma_{T} ] +\frac{1}{16} \gamma_{T}^{2}=
\frac{1}{16} \left( \frac{1}{3}  \sum_{i,j,k,\ell } \pi(e_\ell)(T^{ijk})\g_{\tilde{e}_{\ell}}\g_{ijk} +\gamma_{T}^{2} \right) .
\end{equation}
We compute
$$
 \g_T^2  = \frac{1}{4} {\sum}_{i,j,\ell  ,m,n} \epsilon_\ell T^{\ell ij}T^{\ell mn}\g_{ijmn} = \frac{1}{4} {\sum}_{i,j, \ell ,m,n}^{\prime} 
 \epsilon_\ell T^{\ell ij}T^{\ell mn}\g_{ijmn} -\sum_{i,j,r} \epsilon_{i}\epsilon_{j}\epsilon_{r} (T^{ijr})^{2},
$$ 
where the primed sum is only over pairwise distinct indices.  Similarly,
$$
 \sum_{i,j,k,\ell } \pi(e_\ell)(T^{ijk})\g_{\tilde{e}_{\ell}}\g_{ijk} =  {\sum}_{i,j,k,\ell }^{\prime} \pi(e_\ell)(T^{ijk})\g_{\tilde{e}_{\ell}}\g_{ijk} 
+ 3\sum_{j,k,\ell } \pi (e_{\ell}) (T^{\ell jk})\gamma_{jk}. 
$$
On the other hand, for any $j$ and $k$ fixed,
\begin{eqnarray*}
\sum_{\ell}\pi (e_{\ell}) (T^{\ell jk}) &=&
-\epsilon_{j}\epsilon_{k} \sum_{\ell} \pi (e_{\ell}) \langle [e_{j}, e_{k}], e_{\ell}\rangle \epsilon_{\ell}\\ 
&=& 
 -\epsilon_{j}\epsilon_{k} \sum_{a}\pi (q_{a})\langle [e_{j}, e_{k} ], p_{a}\rangle =0,
\end{eqnarray*}
where we used (\ref{t-ijk})   
and Corollary \ref{sum} (with ${\sigma} =[e_{j}, e_{k}]$, which is a section of $\ker \pi $).  
Combining the above relations we obtain
\begin{align*}
\nonumber{\di}\,^2 &=  \frac{1}{16} \left(  \frac{1}{3} { \sum}_{i,j,k,\ell }^{\prime} \pi(e_\ell)(T^{ijk})\g_{\tilde{e}_{\ell}}\g_{ijk}
+ \frac{1}{4} {\sum}^{\prime}_{i,j,\ell , m,n}  \epsilon_\ell T^{\ell ij}T^{\ell mn}\g_{ijmn}\right)\\
\nonumber&  -\frac{1}{16}
 \sum_{i,j,r} \epsilon_{i}\epsilon_{j}\epsilon_{r} (T^{ijr})^{2} .
\end{align*}
We aim to prove that 
\begin{equation}\label{square}
\di ^{2} 
= -\frac{1}{16} \sum_{i,j,k}\epsilon_{i}\epsilon_{j}\epsilon_{k} (T^{ijk})^{2}.
\end{equation}
For this, we need to show that 
\begin{equation} \label{final:eq} 
\frac{1}{3}  {\sum}_{i,j,k,\ell }^{\prime} \epsilon_\ell \pi (e_\ell) (T^{ijk}) \g_{\ell ijk}+ 
\frac{1}{4}{\sum}_{i,j,\ell  , m,n}^{\prime} \epsilon_\ell T^{\ell ij}T^{\ell mn}\g_{ijmn} =0.
\end{equation} 
To prove (\ref{final:eq})  
we use  axiom C1)
of Definition \ref{def-courant},  where indices of tensor components are metrically raised and lowered 
according to standard conventions: for any $i$, $j$, $k$ fixed, 
\begin{eqnarray*} 0&=&[e_i,[e_j,e_k]] - [[e_i,e_j],e_k]-[e_j,[e_i,e_k]]\\ 
&=& \sum_\ell \left(-[e_i,T_{jk}^{\;\;\; \ell }e_\ell] + [T_{ij}^{\;\;\; \ell }e_\ell,e_k] + [e_j,T_{ik}^{\;\;\; \ell }e_\ell]\right)\\
&=& \sum_{\ell } \left(-\pi (e_i) (T_{jk}^{\;\;\; \ell }) +\pi (e_j) (T_{ik}^{\;\;\; \ell }) \right)e_\ell
+\sum_{\ell , m} \left( T_{jk}^{\;\;\; \ell } T_{i\ell}^{\;\;\; m} - T_{ik}^{\;\;\; \ell } T_{j\ell}^{\;\;\; m}\right)e_m \\
&&+\underbrace{\sum_\ell \left(- [e_k,T_{ij}^{\;\;\; \ell }e_\ell]) + \pi^*d\langle T_{ij}^{\;\;\; \ell }e_\ell,e_k\rangle \right)}_{\pi^*d T_{ijk}-\sum_\ell\pi (e_k) (T_{ij}^{\;\;\; \ell })e_\ell +\sum_{\ell ,m} T_{ij}^{\;\;\; \ell }T_{k\ell}^{\;\;\; m}e_m}\\
&=& \sum_\ell \pi (e_\ell)(T_{ijk})\tilde{e}_\ell-\sum_{(i,j,k)\; \mathrm{cyclic}}\sum_\ell \left( \pi (e_i) (T_{jk}^{\;\;\; \ell })e_\ell -\sum_m T_{ij}^{\;\;\; \ell }T_{k\ell}^{\;\;\; m}e_m\right),
\end{eqnarray*} 
where we have used that $\pi^*df=\sum_\ell \pi (e_\ell)(f)\tilde{e}_\ell$ for all $f\in C^\infty(M)$. 
Therefore, for any $i$, $j$, $k$,  $\ell$ fixed, 
$$
\pi (e_\ell)(T_{ijk})\tilde{e}_\ell-\sum_{(i,j,k)\; \mathrm{cyclic}} \left( \pi (e_i) (T_{jk}^{\;\;\; \ell })e_\ell -\sum_s 
T_{ij}^{\;\;\; s }T_{k s}^{\;\;\; \ell}e_\ell\right) =0.
$$
Taking now $i$, $j$, $k$,  $\ell$ pairwise distinct, multiplying the above equality 
 with $\g^{ijk}$  and summing over (pairwise distinct) $i$, $j$, $k$, $\ell$, we obtain
\[
4{\sum}^{\prime} _{i,j,k,\ell }\pi (e_\ell)(T_{ijk}) \g^{\ell ijk}+3{\sum}' _{i,j,k,\ell,m}T_{ij}^{\;\;\; \ell }T_{k\ell m}\g^{m ijk}=0,
\]
which is precisely (\ref{final:eq}) after re-organising the indices. We proved relation 
(\ref{square}) which implies in particular that $\di ^{2}\in C^{\infty}(U).$ 
From Lemma \ref{c1-2}, $\di$ is a Dirac generating operator for $E$. 
\end{proof}

Combining Proposition \ref{prop-space} with Theorem \ref{ex:thm} we obtain:

\begin{cor} Let $E$ be a regular Courant algebroid on a manifold $M$  and $S$ a spinor bundle over $\mathrm{Cl}(E)$. 
For any  sufficiently small open subset  $U\subset M$,  the set of Dirac generating operators  for $E\vert_{U}$  on 
$S\vert_{U}$ has the structure of an affine space modelled on the vector space 
$$
{V_\di}\vert_{U} := \{ e\in \Gamma (E\vert_{U}),\ [\di , \gamma_{e}] \in C^{\infty} (U)\}  ,
$$
where $\di$ is an arbitrarily chosen   Dirac generating operator on $S\vert_{U}.$ 
\end{cor}

\section{The canonical  Dirac generating operator}\label{sect-d}

Let $E$ be a regular Courant algebroid with scalar product of neutral signature
$(n, n)$.  In this section we construct a canonical Dirac generating operator $\slashed{d}_{\mathbb{S}}$ on a suitable spinor bundle $\mathbb{S}$ of $E$, of the form (\ref{ansatz:eq}).  Its definition will involve an arbitrary  generalized connection $D$ on $E$.  By canonical we mean that $\slashed{d}_{\mathbb{S}}$  is independent of the choice
of $D$.

We start with an arbitrary  spinor bundle  $S$   over $\mathrm{Cl}(E)$. Let  
$D$ a generalized connection on $E$ and $D^{S}$ 
a compatible $E$-connection on $S$.  We begin by analyzing the dependence of
$\slashed{D}^{S}+\frac{1}{4}\gamma_{T^{D}}$
on the data $(D,D^S)$, 
where  $\slashed{D}^S$ is the Dirac operator defined by (\ref{dirac-def}) 
Let $D'=D+A$ be another generalized connection on $E$, where $A\in \Gamma (E^*\otimes \mathfrak{so}(E))$.

\begin{prop} \label{trafo:prop} The following holds. 
\begin{enumerate}
\item[(i)] The torsions $T'$ and $T$ of $D'$ and $D$ are related by:
\begin{equation}\label{torsion-c} 
T'= T + \alpha,
\end{equation}
where $\alpha \in \Gamma (\wedge^3E^*)$ is given by 
\begin{equation}\label{alpha-torsion} 
\alpha (u,v,w) = \sum_{(u,v,w)\; \mathrm{cyclic}} \langle A_uv,w\rangle.
\end{equation}
\item[(ii)] The $E$-connection 
\[ (D^S)' := D^S -\frac12 A \]
is compatible with the generalized connection $D'$. Here $A$ is considered as a map 
$E\rightarrow \wedge^2E^*\cong  \wedge ^2E\subset \mathrm{Cl}(E)$, so that $A_e$ acts by Clifford multiplication 
on $S$ for all $e\in E$. 
\item[(iii)] The Dirac operators $\slashed{D}^{S}$ and $(\slashed{D}^{S})'$ associated with 
$(D,D^S)$ and $(D'$, $(D^S)')$ are related by 
\begin{equation}\label{change-dirac} 
(\slashed{D}^{S})^{'} = \slashed{D}^{S} - \frac14 \g_\alpha -\frac14 \g_{v_A},
\end{equation}
where $v_A = \mathrm{tr}_{\langle \cdot ,\cdot \rangle}A= \sum A_{e_i}\tilde{e}_i \in \Gamma (E)$. 
\item[(iv)] The operators  $\slashed{d}_{S}= \slashed{D}^{S}+\frac{1}{4} \gamma_{T}$ and
 $\slashed{d}^{\prime }_{S}= (\slashed{D}^{S})^{\prime} +\frac{1}{4} \gamma_{T^{\prime }}$ are related by 
\begin{equation}\label{slashed_d:eq}
\slashed{d}^{\prime}_{S}= \slashed{d}_{S} -\frac{1}{4} \gamma_{v_{A}}.
\end{equation}
\end{enumerate}
\end{prop}

\begin{proof} (i) is relation (\ref{formula-torsion}).\\
(ii)  To check the compatibility let $e, v\in \Gamma (E)$ and $s\in \Gamma (S)$:
\begin{eqnarray*}
( D^{S})'_e(vs) &=& D^S_e(vs) -\frac12 A_e (vs) = 
D_e(v)s + vD_e^Ss - \frac{1}{2} (A_e v)s\\
& =& D_{e}(v) s + v D_{e}^{S}s -\frac{1}{2} [A_{e}, v] s -\frac{1}{2} v A_{e}s\\
&=&  D'_e(v)s +v(D^S)'_es.
\end{eqnarray*}
In the fourth equality we used that the commutator 
$[A_e ,v]\in \Gamma (E)$ in the Clifford algebra is related to the evaluation $A_e(v)$ of $A_e\in\Gamma( \mathfrak{so}(E))$ 
on $v$  by the 
formula 
\begin{equation} \label{so_spin:Eq} A_e(v) = -\frac12 [A_e,v].\end{equation}

\noindent 
(iii) With respect to an orthonormal frame $(e_{i})$ 
we write
$$
A =\frac{1}{2} \sum_{i,j,k} A^{ijk} e_{i}\otimes (e_{j}\wedge e_{k}),\ \alpha = \frac{1}{6} \sum_{i,j,k} \alpha^{ijk}e_{i}\wedge e_{j}\wedge e_{k},
$$
where $A^{ijk}:= A(\tilde{e}_{i}, \tilde{e}_{j}, \tilde{e}_{k})$  and $\alpha^{ijk} := \alpha (\tilde{e}_{i}, \tilde{e}_{j},
\tilde{e}_{k})$, where $E$ and $E^*$ are identified with the help of the scalar product. In particular, $A_{e_{i}} = \frac{1}{2} \epsilon_{i} \sum_{j,k} A^{ijk}e_{j}\wedge e_{k}$
is identified with $\frac{1}{2}\epsilon_{i} \sum A^{ijk} e_{j}e_{k}$ in $\mathrm{Cl}(E).$ Similarly, 
$\alpha$ is identified with $\frac{1}{6} \sum \alpha^{ijk} e_{i}e_{j}e_{k}$ in $\mathrm{Cl}(E).$ 
With this notation, 
\begin{eqnarray*} (\slashed{D}^{S})' -\slashed{D}^{S} &=&\frac12 \sum_{i} \tilde{e}_i \left(-\frac12 A_{e_i}\right) =-\frac14 \sum_i \tilde{e}_iA_{e_i} = -\frac18 \sum_{j\neq k} A^{ijk}\g_{ijk}\\
& = & -\frac{1}{8} {\sum}^{\prime}_{i,j,k} A^{ijk} \gamma_{ijk} -\frac{1}{4} \sum_{j,k} A^{jjk} \epsilon_{j} \gamma_{k} . 
\end{eqnarray*}
Remark that
$$
{\sum}^{\prime}_{i,j,k} A^{ijk} \gamma_{ijk}  =\frac{1}{3} {\sum}^{\prime}_{i,j,k} (\sum_{(i,j,k)\; \mathrm{cyclic}} A^{ijk} \gamma_{ijk})
 =\frac{1}{3} {\sum}^{\prime}_{i,j,k}
\alpha^{ijk} \gamma_{ijk} = 2\gamma_{\alpha}
$$
(where in the second  equality we used (\ref{alpha-torsion}))  and similarly 
$$
\sum_{j,k} A^{jjk} \epsilon_{j} \gamma_{k} = \sum_{j,k} \langle A (e_{j}, \tilde{e}_{j}), \tilde{e}_{k} \rangle \gamma_{k}
=\gamma_{v_{A}}.
$$
 Combining the above relations we obtain  (\ref{change-dirac}). (A formula equivalent to (\ref{change-dirac}) is stated as equation (2.22)  in \cite{garcia}.)

\noindent
(iv) follows from (\ref{torsion-c}) and (\ref{change-dirac}). 
\end{proof}

The operator $\slashed{d}_{S} = \slashed{D}^{S}+\frac{1}{4} \gamma_{T^{D}}$, defined using a generalized
connection $D$,  depends on the choice of $E$-connection $D^{S}$ compatible with $D$.
In order to remove this freedom  we define  
\begin{equation} \label{canspinor:eq} \mathcal{S} := S \otimes  | \mathrm{det}\,   S^*|^{\frac{1}{r_S}},\end{equation}
where $r_S:= \mathrm{rank}\, S$ 
and  $| \mathrm{det}\,  S^*|^{\frac{1}{r_S}}$ denotes the bundle of $\frac{1}{r_{S}}$-densities of
$S$.  We call the bundle $\mathcal S$ the {\cmssl canonical spinor bundle} of $S$.

\begin{rem}\label{densities-general}{\rm i) For a vector bundle $V$  of  rank $n$ 
and $s\in \mathbb{R}$, we denote by $| \mathrm{det}\,  V^{*} |^{s}$ 
the line bundle  of 
$s$-densities on  $V$, 
whose fiber over $p\in M$  consists of all  maps
$\omega_{p}: \Lambda^{n}V_{p}\setminus \{ 0\} \rightarrow \mathbb{R}$  (called $s$-densities) which satisfy   
$\omega_{p} ( \lambda \vec{v}) = | \lambda |^{s} \omega_{p} (\vec{v})$, for any
$\vec{v}\in \Lambda^{n}V_{p}\setminus \{ 0\}$ and $\lambda \in \mathbb{R}\setminus \{ 0 \}.$ 
Note that, when $s$ is an integer, $| \mathrm{det}\, V^{*} |^{s}$ is canonically isomorphic to  
$|\mathrm{det}\, V^{*} |^{\otimes s}$ and $| \mathrm{det}\,  V^{*} |^{2s}$ to $( \mathrm{det}\, V^{*})^{2s}$.
Also, 
\begin{equation}\label{det-densities}
|\mathrm{det}\, \Lambda V^{*} | \cong | \mathrm{det}\, V^{*} |^{N /2},\ 
| \mathrm{det}\, (V_{1}\otimes V_{2})^{*}) |^{s} \cong  | \mathrm{det}\, V_{1}^{*} |^{s n_{2}} \otimes
| \mathrm{det}\,  V_{2}^{*} |^{s n_{1}} ,
\end{equation}
where $N:= \mathrm{rank}\, \Lambda V$,   $V_{i}$ are vector bundles of rank $n_{i}$ and $s\in \mathbb{R}.$

ii) An $E$-connection $D$ on $V$ induces an $E$-connection on 
$\mathrm{det}\, V^{*}$ and on 
$| \mathrm{det}\,  V^{*} |^{s}$, for any $s\in \mathbb{R}.$  
The latter is  defined as follows: if $D_{u}\omega = \alpha (u) \omega$, where $u\in\Gamma ( E)$ and $\omega \in \Gamma (\mathrm{det}\, V^{*})$,
then $D_{u}(| \omega |^{s} )=s \alpha (u) |\omega |^{s}$, where 
$| \omega |^{s}$ is the $s$-density defined by $| \omega |^{s} (v_{1}, \cdots , v_{n}) :=
| \omega (v_{1}, \cdots , v_{n} ) |^{s}$. Remark that if $D^{\prime} = D + A$ are two $E$-connections on $V$
then the induced $E$-connections on  $| \mathrm{det}\,  V^{*} |^{s}$  are related by
$D^{\prime}_{u} = D_{u} - s\, \mathrm{trace}\, A_{u}.$\

iii)  Let 
$S$ and $S'$ be two  irreducible  spinor bundles of $E$. Since  
$\langle \cdot , \cdot \rangle$
has neutral signature,  we can write $S' = S \otimes L$, where $L$ is a  line bundle.
From the second relation  (\ref{det-densities}) we obtain  that 
$\mathcal{S}^{\prime}=\mathcal S\otimes L \otimes  | \mathrm{det}\, L^{*}|$. 
If $S$ and $S^{\prime}$ are isomorphic, then $L$ is trivializable, i.e.\ orientable. 
If $L$ is moreover oriented, then $L \otimes  | \mathrm{det}\, L^{*}|\cong L\otimes L^*$ is canonically trivial. 
We deduce that  
$\mathcal S$ and $\mathcal S^{\prime}$ are then canonically isomorphic. This is another reason why
$\mathcal S$ is called the canonical spinor bundle.}
\end{rem}

\begin{lem} \label{indep0:Prop} 
Let $D$ be a generalized connection on $E$ with torsion $T$ and $D^{S}$ an $E$-connection on $S$ compatible with $D$.
Then $D^{S}$ induces a connection $D^{\mathcal S}$  on $\mathcal S$, which is compatible with $D$ and depends only on $D$.
In particular, $\slashed{d}_{\mathcal S}:\Gamma (\mathcal S )\rightarrow \Gamma (\mathcal S )$  depends only on $D$.
\end{lem}

\begin{proof} It is clear that the $E$-connection $D^\mathcal{S}$ on $\mathcal{S}$ induced by $D^S$ is again compatible with $D$. 
Any other $E$-connection on $S$  compatible with $D$ is of the 
form $(D^{S})^{\prime} = D^S +\varphi \otimes \mathrm{Id}_{S}$ for some section $\varphi \in \Gamma (E^*)$. 
From Remark \ref{densities-general} ii) 
$D^{S}$ and $(D^{S})^{\prime}$ induce the same connection  on $\mathcal S .$
\end{proof}

In order to define a Dirac generating  operator independent of $D$ we consider as in \cite{garcia} the following 
{\cmssl canonical weighted spinor bundle} 
\begin{equation}\label{def-ss} 
\mathbb{S} := \mathcal{S}\otimes L,\quad L:=|\mathrm{det}\,  T^*M|^\frac12.
\end{equation}
The line bundle $L$ carries an induced $E$-connection 
$D^L$ defined by 
\begin{equation}\label{d-l} 
D^L_v \mu := \mathcal{L}_{\pi (v)}\mu  - \frac12 \mathrm{div}_D(v) \mu , 
\end{equation}
where $v\in \Gamma (E)$, 
$ \mathcal{L}_{\pi (v)}\mu$ denotes the Lie derivative of the densitity 
$\mu \in \Gamma (L)$ with respect to the vector field $\pi (v)$ 
and $\mathrm{div}_D(v) := \mathrm{tr}\, Dv$. 

\begin{rem}{\rm 
The Lie derivative $\mathcal L_{X}\mu $ of a density $\mu \in \Gamma ( | \mathrm{det}\, T^{*}M |^{1/r})$ 
in the direction of $X\in {\mathfrak X}(M)$
is defined by $\mathcal L_{X}\mu =\frac{1}{2r} \alpha (X)\mu$,  when ${\mathcal L}_{X} \mu^{2r}=
\alpha (X)\mu^{2r}$. In the latter expression  ${\mathcal L}_{X}$ denotes the Lie derivative 
on 
$(\mathrm{det}\, T^{*}M )^{2}$ (defined by ${\mathcal L}_{X} (\omega^{2} ) = 2 ({\mathcal L}_{X}\omega )
\omega$ where ${\mathcal L}_{X}\omega $ is the Lie derivative of the volume form 
$\omega $ of $M$).  
If $\nabla$ is a torsion-free connection on $M$, then it induces a connection (also denoted by $\nabla$)
on $ | \mathrm{det}\, T^{*}M |^{1/r}$ and one can show that
\begin{equation}\label{torsion-free-formula}
\nabla_{X} \mu = {\mathcal L}_{X}\mu -\frac{1}{r} \mathrm{trace}\, (\nabla X) \mu,\ 
\forall \mu \in \Gamma  ( | \mathrm{det}\, T^{*}M |^{1/r}),\ X\in {\mathfrak X}(M).
\end{equation}}
\end{rem}

\begin{lem}\label{lem-added}
If $D$ is the generalized connection from Lemma \ref{concl}, then the $E$-connection $D^{L}$  defined by (\ref{d-l})  is induced by  a usual  
connection on $L$. 
\end{lem} 

\begin{proof} We need to show that $\mathrm{div}_{D} (v) =0$ for any $v\in \Gamma (\ker\  \pi )$ (see Example \ref{ext:exa}). 
Consider the frame $(p_{a}, q_{a})$ 
constructed after Lemma 
\ref{lem-chen} and recall that it is  parallel with respect to the connection $\nabla_{\pi (e)} = D_{e}$ on $E.$ 
Its dual frame is $(\tilde{p}_{a}, \tilde{q}_{b}) = (q_{a}, p_{b})$ and
$$
\mathrm{div}_{D } (v) = \sum_{a}( \langle D_{p_{a}} (v), q_{a}\rangle + \langle D_{q_{a}} (v), p_{a} \rangle )= 
\sum_{a} \langle \nabla_{\pi (q_{a})} (v), p_{a} \rangle =\sum_{a} \pi (q_{a} )  \langle v, p_{a} \rangle
$$
where we used $\pi (p_{a} ) =0$ and $\nabla p_{a} =0$.  From Corollary \ref{sum}  we obtain 
$ \mathrm{div}_{D } (v) =0$ as needed.
\end{proof}

\begin{thm}\cite{AX} \label{indep:prop}Let $D$ be a generalized connection on $E$ with torsion $T$, 
$D^{\mathbb{S}}=D^\mathcal{S}\otimes D^L$ the induced compatible $E$-connection on the canonical 
weighted spinor bundle $\mathbb{S}$, and $\slashed{D}^{\mathbb{S}}$ the corresponding Dirac operator on $\mathbb{S}$. Then 
\begin{equation} \label{inv:eq}
\slashed{d}_{\mathbb{S}} = \slashed{D}^{\mathbb{S}} +\frac14 \g_T : \Gamma (\mathbb{S}) \rightarrow \Gamma (\mathbb{S}) 
\end{equation}
is a Dirac generating operator, independent of $D$.
\end{thm}

\begin{proof} 
 Replacing $D$ by another generalized connection $D'=D+A$ 
changes $D^{S}$ to  $(D^{\prime})^{S} = D^{S} -\frac{1}{2} A$  (see Proposition \ref{trafo:prop}). 
This implies that  $(D^{\prime})^{\mathcal S} = D^{\mathcal S} -\frac{1}{2} A$.
Here we use that the Clifford action of $A_{u}$ on $S$ is trace-free 
(see the next remark) and Remark \ref{densities-general} ii). 
From Proposition \ref{trafo:prop} again
(applied to the spinor bundle $\mathcal S$)  we obtain that 
${\slashed{D}}^{\mathcal S}+\frac{1}{4} \gamma_{T}$ changes to 
$({\slashed{D}}^{\mathcal S})'+\frac{1}{4} \gamma_{T^{\prime}} = 
(\slashed{D}^{\mathcal S} +\frac{1}{4}\gamma_{T})-\frac{1}{4} \gamma_{v_{A}}$.
(See Proposition \ref{trafo:prop} for the definition of $v_A$.)

On  the other hand, on $\mathbb{S}=\mathcal S \otimes L$,
$$
\slashed{D} (s\otimes l) = ( \slashed{D}^{\mathcal{S}} s) \otimes l + \frac{1}{2}\tilde{e}_{i} s \otimes D^{L}_{e_{i}} l,\ 
s\in \Gamma (\mathcal S ),\ l\in \Gamma (L)
$$
and a similar expression holds for the Dirac operator $\slashed{D}'$ on $\mathbb{S}$ computed with the generalized
connection $D' $. We deduce that 
\begin{align*}
& (\slashed{D} +\frac{1}{4} \gamma_{T})(s\otimes l)
 =  (\slashed{D}^{\mathcal S} s+\frac{1}{4}  Ts) \otimes l + \frac{1}{2} \tilde{e}_{i}s \otimes 
D^{L}_{e_{i}}l,\\
& (\slashed{D}' +\frac{1}{4} \gamma_{T^{\prime}})(s\otimes l)
 = (\slashed{D}^{\mathcal S} s+\frac{1}{4}  T s -\frac{1}{4}v_{A} s) \otimes l + \frac{1}{2} \tilde{e}_{i}s \otimes 
(D^{\prime})^{L} _{e_{i}}l.
\end{align*}  
But  $\mathrm{div}_{D'}=\mathrm{div}_D-\langle v_A,\cdot \rangle$  implies  that
$(D^{\prime})_{v}^{L}l  = D^{L}_{v} l +\frac{1}{2} \langle v_{A}, v \rangle l$. We deduce that 
(\ref{inv:eq}) is independent on $D$.

It remains to show that $\slashed{d}_{\mathbb{S}}$ is a Dirac generating operator.
As this is a local property, we need to show that for any  sufficiently small open subset $U\subset M$, the restriction 
$\slashed{d}_{\mathbb{S}} \vert_{U}: \Gamma (\mathbb{S}\vert_{U}) \rightarrow \Gamma (\mathbb{S}\vert_{U})$  is a Dirac generating operator.  Choose $U$  like in Lemma \ref{concl}, and let $D$ be the generalized connection on $E\vert_{U}$ defined by
the flat metric connection $\nabla$ used in that lemma. 
From Lemma \ref{lem-added}, $D^{\mathcal S}\otimes D^{L}$  on $\mathbb{S} \vert_{U}= (\mathcal S \otimes L)\vert_{U}$ is induced by a usual connection 
compatible with $\nabla$. Since $\slashed{d}_{\mathbb{S}}$ is independent on $D$, we deduce that on $U$ it
coincides with the  operator constructed in Lemma \ref{concl}.
In particular, it is a Dirac generating operator.  
\end{proof}

\begin{rem}\label{traceRemark} 
In the proof we have used the fact that any representation of $\mathfrak{spin}(k,\ell )\subset \mathrm{Cl}(\mathbb{R}^{k,\ell })$ induced by a representation of the Clifford algebra  $\mathrm{Cl}(\mathbb{R}^{k,\ell })$ is trace-free. This follows from the fact
that the Lie algebra $\mathfrak{spin}(k,\ell )$ is spanned by commutators $uv-vu$, where $u,v\in \mathbb{R}^{k,\ell }$.
\end{rem}

\begin{defn} The operator $\di_{\mathbb{S}} :\Gamma (\mathbb{S}) \rightarrow \Gamma (\mathbb{S})$ is called the
{\cmssl canonical Dirac generating operator} associated to $E$.
\end{defn}

\begin{rem}{\rm Our canonical Dirac generating operator coincides with the Dirac generating operator constructed
in Theorem 4.1 of \cite{AX}. This follows from formula (53) of \cite{AX}, by noticing a difference of a minus sign between our definition for  the torsion
of a generalized connection and that from \cite{AX} (see Section 3.2 of this reference). Our Dirac generating operator also coincides
(up to a multiplicative constant factor)  with the Dirac generating operator from Proposition 5.12 of \cite{GMX} (in formula (5.14) of \cite{GMX} 
the term $\frac{1}{2} \gamma (C_{\nabla})$ should have a minus sign).}
\end{rem}

\section{Standard form of the canonical Dirac generating operator}

In this section  we provide an  expression for the canonical Dirac generating operator  $\slashed{d}_{\mathbb{S}}$ 
which uses the structure of regular Courant algebroids (see \cite{chen}).
We consider the setting from the previous section. For simplicity, from  now on  $\slashed{d}_{\mathbb{S}}$ will be denoted by $\slashed{d}$. 
From  Lemma \ref{lem-chen}, there is a vector bundle  isomorphism 
\begin{equation}\label{iso-i}
I:  E \rightarrow F^{*}\oplus \mathcal G \oplus F, 
\end{equation}
where we recall that $F= \pi (E)\subset TM$ is an integrable distribution and $\mathcal{G}\subset \ker \pi \subset E$  
is a subbundle. 
The isomorphism $I$ maps the 
anchor  $\pi : E \rightarrow TM$ of $E$ to the map $\rho ( \xi + r + X)=X$ 
and the  scalar product of $E$ to  a  scalar product 
 \begin{equation}\label{metric-diss}
\langle \xi + r + X, \eta + s + y\rangle = \frac{1}{2} ( \xi (Y) + \eta (X)) + (r, s)^{\mathcal G},
\end{equation}
where $\xi + r+ X$, $\eta + s + Y\in  F^{*} \oplus \mathcal G \oplus F$,
and  the scalar product $(\cdot , \cdot )^{\mathcal G}$  on $\mathcal G$ is of neutral signature. 
The bundle $\mathcal G$, together with $(\cdot , \cdot )^{\mathcal G}$, is  canonically
associated to $E$.  More precisely, $(\mathcal G , (\cdot , \cdot )^{\mathcal G})$ is 
isomorphic to   $(\ker \pi )/  (\ker  \pi )^{\perp}$ with scalar product  induced by the scalar product  of $E$.
Moreover, $\mathcal G$ is a bundle of Lie algebras, with Lie bracket  $[\cdot , \cdot ]^{\mathcal G}$  
induced from the Dorfman bracket of $E$. The scalar product $(\cdot , \cdot )^{\mathcal G}$ is invariant
with respect to $[\cdot , \cdot ]^{\mathcal G}$, that is the adjoint representation of the Lie algebra is by skew-symmetric endomorphisms.
In fact, these properties follow immediately from $\mathcal{G}\subset \ker \pi$.  An isomorphism (\ref{iso-i})  as above is called a {\cmssl dissection}
of $E$  \cite{chen}.

The Dorfman bracket $[\cdot , \cdot ]$
of $F^{*} \oplus \mathcal G \oplus F$ induced from $E$ via a dissection satisfies
\begin{equation}\label{ind-dorf}
 \mathrm{Pr}_{\mathcal G} [r_{1}, r_{2} ] = [r_{1}, r_{2} ]^{\mathcal G},\ r_{i}\in \Gamma (\mathcal G ),
\end{equation}
where $\mathrm{Pr}_{\mathcal G}$ is the natural projection from $F^{*}\oplus \mathcal G \oplus F$ on
$\mathcal G$ (we shall use a similar notation $\mathrm{Pr}_{F^{*}}$ for the natural projection on $F^{*}$).

Therefore, we may (and will)  assume that the given regular Courant algebroid 
is of the form $E= F^{*} \oplus \mathcal G \oplus F$, with anchor 
$\rho (\xi + r + X)= X$,  metric given by (\ref{metric-diss}) and Dorfman bracket $[\cdot , \cdot ]$
satisfying (\ref{ind-dorf}).
As proved in  Lemma 2.1 of \cite{chen}, the Dorfman bracket of $E$  is  determined by its components 
\begin{align*}
&\nabla : \Gamma (F) \times \Gamma ( \mathcal G ) \rightarrow \Gamma (\mathcal G),\ \nabla_{X} (r):= \mathrm{Pr}_{\mathcal G} [X, r],\\
& R:\Gamma (F)\times \Gamma (F) \rightarrow \Gamma (\mathcal G),\ R(X, Y) := \mathrm{Pr}_{\mathcal G} [X, Y]\\
& {\mathcal H} : \Gamma (F)\times \Gamma (F) \times \Gamma (F) \rightarrow C^{\infty}(M),\ {\mathcal H}(X, Y, Z):= (\mathrm{Pr}_{F^{*}} [X, Y] )(Z).
\end{align*}
Note that here $[X,Y]$ stands for the Dorfman bracket 
\[ [X,Y]={\mathcal H}(X, Y, \cdot ) + R(X,Y) + \mathcal{L}_XY\]
of $X,Y$ as sections of $F\subset E$, whereas, for the rest of this section,  
the Lie bracket  of vector fields will be always denoted by $\mathcal{L}_XY$.
The map $\nabla$ is a partial connection  on $\mathcal G$, 
the  map $R$ is a  $2$-form on $F$ with values in $\mathcal G$ and $\mathcal H$ is a $3$-form on $F$. The properties of the triple $(\nabla , R, \mathcal H )$ are described
in Theorem 2.3 of \cite{chen}.

The next lemma was proved in \cite{chen}  and can be checked directly (we remark a difference of sign between our definition  for the torsion of a generalized
connection  and that  from \cite{chen}).  By a torsion-free connection on $F$ we mean a 
partial connection  $\nabla^{F} : \Gamma (F)\times \Gamma (F)
\rightarrow \Gamma (F)$ which satisfies  $\nabla_{X}Y - \nabla_{Y}X = \mathcal{L}_XY$ for any $X, Y\in \Gamma (F)$. 

\begin{lem}(\cite{chen}) \label{lem-reg} Let $\nabla^{F}$ be a  torsion-free connection on $F$. Then
\begin{equation}\label{nabla-e}
\nabla^{E}_{\xi + r + X} (\eta + s +Y) := ( \nabla^{F}_{X}\eta -\frac{1}{3} {\mathcal H}(X, Y,\cdot ), \nabla_{X}s +\frac{2}{3} [r, s]^{\mathcal G}, \nabla_{X}^{F}Y )
\end{equation}
is a generalized connection on $E$ with torsion given by
\begin{align}
\nonumber T^{\nabla^{E}}(u, v, w)   = &  - {\mathcal H }(X, Y, Z) - (R(X, Y), t)^{\mathcal G} - (R(Y ,Z), r)^{\mathcal G} - (R(Z, X), s)^{\mathcal G} \\
\label{torsion-ne} & + ( [r,s]^{\mathcal G}, t)^{\mathcal G},
\end{align}
for any $u = \xi + r + X$, $v= \eta + s +Y$, $w= \zeta + t + Z$ (where $\xi , \eta , \zeta \in F^{*}$, $r,s,t\in \mathcal G$, 
$X, Y, Z\in F$). 
\end{lem}

\begin{rem}{\rm   For any regular Courant algebroid $E$ with anchor  $\pi : E \rightarrow TM$,  
the quotient ${\mathcal A}:= E/ (\ker  \pi  )^{\perp}$  inherits a Lie algebroid structure from the Dorfman bracket of $E$.
This Lie algebroid  is called    in \cite{chen}  the {\cmssl  ample Lie algebroid} associated to $E$. A dissection induces a bundle  isomorphism 
$\mathcal A\cong   \mathcal G \oplus F$ and a Lie algebroid structure on $\mathcal G \oplus  F$
(inherited from the Lie algebroid structure of $\mathcal A$).
The restriction of the $3$-form $\Omega := -T^{\nabla^{E}}$  to $\mathcal G \oplus  F$ is closed 
with respect to the Lie algebroid differential of $\mathcal G \oplus  F$ and its cohomology class is independent on the chosen dissection.  
It is called  the {\cmssl Severa class } of $E$, as it coincides with the  Severa class when $E$ is exact \cite{chen}. } 
\end{rem}

Let $\Lambda F^{*}$ be the bundle of forms over $F$. It is a spinor  bundle over  the bundle of Clifford algebras 
$\mathrm{Cl}(F^{*} \oplus  F)$, where
$F^{*}\oplus  F$ has scalar product 
$\langle \xi + X, \eta + Y\rangle =\frac{1}{2} (\eta (X) +\xi (Y))$. The  Clifford  representation is  
\begin{equation}\label{spinor-part}
(X + \xi ) \cdot \alpha := \iota_{X}\alpha + \xi \wedge \alpha, 
\end{equation}
where $\iota_{X} \alpha := \alpha (X, \cdot )$ denotes the interior product. 
Let $S_{\mathcal G}$ be an irreducible  spinor  bundle over $\mathrm{Cl} (\mathcal G)$, where $\mathcal G$ is considered with the scalar product 
$ (\cdot  , \cdot )^{\mathcal G}$. 
We assume that $\mathcal G$ is oriented, so that $S_{\mathcal G}$ is $\mathbb{Z}_{2}$-graded.
The  $\mathbb{Z}_2$-graded tensor product $S:= \Lambda  F^{*} \hat{\otimes} S_{\mathcal G}$
is an irreducible spinor bundle over  $\mathrm{Cl}(E )= \mathrm{Cl} (F^{*} \oplus  F)\hat{\otimes} \mathrm{Cl}(\mathcal G)$. 
(Basic facts concerning the $\mathbb{Z}_2$-graded tensor product
are reviewed in more detail in Appendix \ref{complex-spinors}; in particular,
see relation  (\ref{gen-graded-spinor}) for the Clifford action of $E$ on $S$.) 

\begin{lem} The canonical weighted spinor bundle of $S$   is given by
\begin{equation}\label{graded-can-sp}
\mathbb{S}  =\left( \Lambda  F^{*} \otimes   | \mathrm{det}\, ( \mathrm{Ann}\, F) |^{1/2}\right) \hat{\otimes} \mathcal{S}_{\mathcal G}
\end{equation}
where  $\mathrm{Ann}\, F\subset T^{*}M$ is the annihilator of $F$  and 
$\mathcal{S}_{\mathcal G}= S_{\mathcal G} \otimes  |\mathrm{det}\, S^{*}_{\mathcal G}|^{1/g}$ is the 
canonical  spinor bundle of $S_{\mathcal G}$ (with  $g:= \mathrm{rank}\, S_{\mathcal G}$).
\end{lem}  

\begin{proof}
The  isomorphism (given by contraction) between  $(\mathrm{det}\, TM) \otimes \mathrm{det}\, F^{*}$  and
 $\mathrm{det}\, \left(  (\mathrm{Ann}\, F)^{*}\right) $, induces a canonical isomorphism 
\begin{equation}\label{det-F}
| \mathrm{det}\, F|^{1/2} \otimes | \mathrm{det}\ T^{*}M |^{1/2} \cong | \mathrm{det}\,  (\mathrm{Ann}\, F ) |^{1/2}.
\end{equation}
The claim follows from the definition of the canonical weighted spinor bundle $\mathbb{S}$, recall (\ref{def-ss}) and (\ref{canspinor:eq}), 
combined with relations (\ref{det-densities}) 
and (\ref{det-F}). 
\end{proof}

Recall the definition of the $E$-connection $D^{\mathbb{S}}$ from Theorem  \ref{indep:prop} computed from a generalized connection $D$.

\begin{lem}\label{connection-as} The $E$-connection $\nabla^{\mathbb{S}}$ on $\mathbb{S}$ computed 
from the generalized connection $\nabla^{E}$ defined in Lemma \ref{lem-reg}
has the following expression: for any 
$\omega \in\Gamma ( \Lambda F^{*})$, $\tau \in\Gamma (  | \mathrm{det}\,  ( \mathrm{Ann}\,  F) |^{1/2})$ 
 and $s\in\Gamma  (\mathcal{S}_{\mathcal G})$ we have  
\begin{align}
\nonumber& \nabla^{\mathbb{S}}_{u} ( \omega \otimes \tau  \otimes s) =( \nabla^{F}_{X} \omega ) \otimes \tau  \otimes s +\omega
\otimes ( {\mathcal L}_{X}  \tau  ) \otimes s \\
\label{expr-c}& 
+\frac{1}{3} (\iota_{X}{\mathcal H})\wedge \omega \otimes \tau  \otimes s + \omega \otimes \tau  \otimes ( \nabla^{0, \mathcal S_{\mathcal G}}_{X} s -\frac{1}{3} 
(\mathrm{ad}_{r}) (s)) .
\end{align}
Above $u= \xi + r + X$,  $\mathcal L_{X}\tau $ denotes the Lie derivative of $\tau$ in the direction 
of $X\in \Gamma (F)$,   $\nabla^{0, \mathcal S_{\mathcal G}}$ 
is an $F$-connection 
on $\mathcal{S}_{\mathcal G}$, induced by an $F$-connection $\nabla^{0,S_{\mathcal G}}$ on $S_{\mathcal G}$ compatible with 
the $F$-connection  $\nabla$ of $\mathcal G$ and $\mathrm{ad}_{r} \in \mathfrak{so}(\mathcal{G})$ is  considered as a $2$-form on 
$\mathcal G $, which acts by Clifford multiplication on $\mathcal S_{\mathcal G}.$   
\end{lem}

\begin{proof} We remark that $\nabla^{E } = \nabla^{F^{*} + F} + \nabla^{\mathcal G}$, where $\nabla^{F^{*}+ F}$ and $\nabla^{\mathcal G}$ 
are  $E$-connections  on $F^{*} \oplus  F$ and $\mathcal G$ respectively, 
defined by
\begin{align*}
\nonumber& \nabla^{F^{*} + F}_{u} (\eta + Y) = \nabla^{F}_{X}\eta -\frac{1}{3}{\mathcal H}(X, Y,\cdot ) +\nabla^{F}_{X}Y,\\
\nonumber& \nabla^{\mathcal G}_{u} (s):=\nabla_{X}s +\frac{2}{3} \mathrm{ad}_{r}(s)
\end{align*}
where $\mathrm{ad}_{r} (s) = [r, s]^{\mathcal G}.$
The $F$-connection $\nabla^{F}$ induces  an $F$-connection (also denoted by $\nabla^{F}$) 
on  $\Lambda F^{*}$.  A straightforward
computation shows that   the $E$-connection 
\begin{equation}\label{rel-spin-c}
\nabla^{F, \,\mathrm{spin}}_{u} := \nabla^{F}_{X} +\frac{1}{3} (\iota_{X}{\mathcal H})\wedge
\end{equation}
on $\Lambda F^{*}$ is compatible with $\nabla^{F^{*} + F}.$  
On the other hand,  like in the proof of 
Proposition \ref{trafo:prop}  ii),   
the $E$-connection 
\begin{equation}\label{sase}
\nabla^{\mathcal G ,\,\mathrm{spin}}_{u} := \nabla^{0, S_{\mathcal G}}_{X} -\frac{1}{3} \mathrm{ad}_{r} 
\end{equation}
on ${S}_{\mathcal G}$ is compatible with $\nabla^{\mathcal G}$. 
We obtain that $\nabla^{F,\,\mathrm{spin}}\otimes \nabla^{\mathcal G  ,\,\mathrm{spin}}$ is an $E$-connection on 
$S = \Lambda F^{*} \hat{\otimes} S_{\mathcal G}$ compatible with 
$\nabla^{E}$. 
The definition of $\nabla^{E}$ implies that
$\mathrm{div}_{\nabla^{E}} (u) = \mathrm{trace}\, (\nabla^{F} X)$ and thus  
$$ 
(\nabla^{E})^{L}_{u}  \mu   = \mathcal L_{X}\mu   -\frac{1}{2}\mathrm{trace}( \nabla^{F}X)\mu,
$$
for any section $\mu$ of  $L =  | \mathrm{det}\, T^{*}M|^{1/2}$. We obtain that the $E$-connection
$\nabla^{\mathbb{S}}$ 
on $\mathbb{S} = (\Lambda F^{*} \otimes | \mathrm{det}\, F |^{1/2} )\hat{\otimes} \mathcal S_{\mathcal G}\otimes L$
is given by  
\begin{equation}\label{mmss}
\nabla^{\mathbb{S}} := \nabla^{F,\,\mathrm{spin}}\otimes \nabla^{{\mathcal G},\,\mathrm{spin}}\otimes (\nabla^{E})^{L} 
\end{equation}
where we preserve  the same symbols 
$\nabla^{F,\,\mathrm{spin}}$ and $\nabla^{\mathcal G ,\,\mathrm{spin}}$ for the $E$-connections induced 
by $\nabla^{F,\,\mathrm{spin}}$ and $\nabla^{\mathcal G ,\,\mathrm{spin}}$
on 
$\Lambda F^{*}\otimes |\mathrm{det}\,  F|^{1/2}$ and $\mathcal S_{\mathcal G}$ respectively. 
In order to compute $\nabla^{\mathbb{S}}$ we shall compute  $\nabla^{F,\,\mathrm{spin}}\otimes (\nabla^{E})^{L}$
and $\nabla^{{\mathcal G},\,\mathrm{spin}}$ separately.

We begin with $\nabla^{F,\,\mathrm{spin}}\otimes (\nabla^{E})^{L}$. This is 
 an $E$-connection on $\Lambda F^{*} \otimes  | \mathrm{det}\, F|^{1/2} \otimes L.$
Relation (\ref{rel-spin-c})
holds also on  $\Lambda F^{*}\otimes | \mathrm{det}\,  F|^{1/2}$, as the endomorphism 
$\omega \rightarrow \frac{1}{3} (\iota_{X}{\mathcal  H})\wedge \omega$
of  $\Lambda F^{*}$ is  trace-free. Let $\omega \in \Gamma (\Lambda F^{*})$, $l\in \Gamma ( |\mathrm{det}\, F|^{1/2})$ and 
$\mu \in \Gamma ( L)$.
Then
\begin{align}
\nonumber  & (\nabla^{F,\,\mathrm{spin}} \otimes (\nabla^{E})^{L})_{u} ( \omega \otimes l \otimes \mu ) = 
\nabla^{F,spin}_{X}(\omega\otimes l) \otimes \mu + \omega\otimes l \otimes (\nabla^{E})^{L}_{X} \mu   \\
\nonumber& = (\nabla^{F}_{X} (\omega \otimes l)  +\frac{1}{3}( i_{X}{\mathcal H}\wedge \omega )\otimes l ) \otimes \mu + 
\omega \otimes l \otimes ( {\mathcal L}_{X} \mu  -\frac{1}{2} \mathrm{trace} (\nabla^{F}X) \mu ) \\
\label{sase-0} & = (\nabla^{F}_{X} \omega ) \otimes l \otimes \mu + \omega \otimes {\mathcal L}_{X} (l\otimes \mu ) 
+\frac{1}{3} ( \iota_{X}{\mathcal H} \wedge \omega )\otimes l \otimes \mu 
\end{align} 
where in the last equality we used 
\begin{align}
\nonumber&  \nabla^{F}_{X} (\omega \otimes l)= (\nabla^{F}_{X}\omega ) \otimes l + \omega \otimes \nabla^{F}_{X}l\\
\label{ajutatoare} &\nabla^{F}_{X} l = {\mathcal L}_{X} l +\frac{1}{2} \mathrm{trace} (\nabla^{F} X) l.
\end{align}
(see (\ref{torsion-free-formula}) for the second relation (\ref{ajutatoare})).\

Next, we compute $\nabla^{\mathcal G,\,\mathrm{spin}}.$  Relation (\ref{sase})
holds also on $\mathcal S_{\mathcal G}$ (with $\nabla^{0, S_{\mathcal G}}$ replaced 
by $\nabla^{0, \mathcal S_{\mathcal G}}$, the $E$-connection on 
$\mathcal S_{\mathcal G}$ induced by $\nabla^{0, S_{\mathcal G}}$).   
This follows from the fact  that the  Clifford action by  $\mathrm{ad}_{r}\in \Gamma ( \Lambda^{2} \mathcal G )$ on $S_{\mathcal G}$ is 
trace-free.  The latter is a consequence of Remark \ref{traceRemark}.
So we have proven:
\begin{equation}\label{saseprime}
\nabla^{\mathcal G ,\,\mathrm{spin}}_{u} = \nabla^{0, \mathcal S_{\mathcal G}}_{X} -\frac{1}{3} \mathrm{ad}_{r}  \quad\mbox{on}\quad \mathcal S_{\mathcal G}.
\end{equation}
Combining (\ref{mmss}), (\ref{sase-0}) and (\ref{saseprime}), we obtain (\ref{expr-c}). 
\end{proof}

\begin{notation} {\rm For any $\tau\in \Gamma ( | \mathrm{det}\,  ( \mathrm{Ann}\, F)|^{1/2})$, 
${\mathcal L}_{X} \tau $ is $C^{\infty}(M)$-linear in $X$, when $X\in \Gamma (F)$
(as ${\mathcal L}_{fX}\omega = f {\mathcal L}_{X}\omega$, for any $\omega \in \Gamma (\mathrm{Ann}\, F)$).
The map $\Gamma (F)\ni X \rightarrow {\mathcal L}_{X} \tau $ is a $1$-form on $F$ with values in 
$| \mathrm{det}\, ( \mathrm{Ann}\, F)|^{1/2}$, which was denoted by ${\mathcal L} (\tau )$.}
\end{notation}

We arrive now at  what we call  the {\cmssl standard form} for the canonical Dirac generating operator in terms of the 
data encoding the regular Courant algebroid.
\begin{thm}\label{thm-alt} 
Let $E$ be a regular Courant algebroid  with anchor $\pi : E \rightarrow TM$ and   scalar product of neutral signature.
In terms of a dissection  $F^{*}\oplus \mathcal G\oplus F$  
of $E$, 
the canonical Dirac generating operator   is given by
\begin{align}
\nonumber \slashed{d} (\omega \otimes \tau  \otimes s) &= (d^{F}\omega )\otimes \tau  \otimes s + {\mathcal L} (\tau  ) \wedge \omega \otimes s
 + \nabla^{0, {\mathcal S}_{\mathcal G}} (s) \wedge \omega \otimes \tau \\ 
\label{alt-dirac}& - ({\mathcal H}\wedge \omega )\otimes \tau \otimes s + \frac{1}{4} (-1)^{|\omega |+1} \omega \otimes \tau  \otimes   Cs \nonumber\\ &  
+ (-1)^{|\omega| +1} \bar{R} ( \omega \otimes \tau \otimes s),
\end{align}
where $\omega \in\Gamma ( \Lambda F^{*})$, $\tau \in\Gamma (  | \mathrm{det}\,  ( \mathrm{Ann} F) |^{1/2})$ 
and $s\in\Gamma  (\mathcal{S}_{\mathcal G}).$ 
Above $d^{F} : \Gamma ( \Lambda^\bullet  F^{*} ) \rightarrow \Gamma  (  \Lambda^{\bullet+1}F^{*})$ is the 
exterior derivative along the integrable distribution  $F= \mathrm{im}\, \pi$,  
$C\in \Gamma (\Lambda^{3} \mathcal G^{*})$ is the Cartan form
$C(u, v, w) = ([u, v]^{\mathcal G}, w)^{\mathcal G}$ of $\mathcal G$
viewed as a  section of $\mathrm{Cl} ( \mathcal G )$,  $Cs$ denotes its Clifford action on $s$
and 
\begin{equation}
\bar{R}  (\omega\otimes \tau \otimes s) = \frac{1}{2} \sum_{i,j,k}( R(X_{i}, X_{j}), r_{k})^{\mathcal G} (\alpha_{i}\wedge \alpha_{j}\wedge \omega) \otimes \tau 
\otimes \tilde{r}_{k}s,
\end{equation}
where $ (X_{i})$ is a basis of $F$, $(\alpha_{i})$ is the dual basis, 
i.e.\ $\alpha_{i}(X_{j}) = \delta_{ij}$, $( r_{i} )$ is a basis of $\mathcal G$ and 
$(\tilde{r}_{i})$ is the dual basis with respect to $(\cdot , \cdot )^{\mathcal G}.$ 
\end{thm}

\begin{proof}  The bases of $E$ 
\begin{equation}\label{bases-dual}
 (e_{i} ) : = ( X_{i}, r_{j}, \alpha_{k}),\     (\tilde{e}_{i}) : = ( 2\alpha_{i}, \tilde{r}_{j},  2 X_{k} ) 
\end{equation}
are  dual with respect to the scalar product  (\ref{metric-diss}) of $E$.  Using Lemma \ref{connection-as} and the definition of the Clifford action, we obtain 
\begin{align*}
\sum_{i}\tilde{e}_{i}  \nabla^{\mathbb{S}}_{e_{i}} (\omega \otimes\tau \otimes s) & = \sum_{i} (
2 (\alpha_{i} \wedge \nabla^{F}_{X_{i}} \omega ) \otimes\tau \otimes s  + 
2(\alpha_{i} \wedge \omega )\otimes {\mathcal L}_{X_{i}} \tau \otimes s\\
& + \frac{2}{3} (\alpha_{i}\wedge \iota_{X_{i}} {\mathcal H}\wedge \omega ) \otimes\tau \otimes s + 
2(\alpha_{i}\wedge \omega )\otimes \tau \otimes \nabla_{X_{i}}^{0, {\mathcal S}_{\mathcal G}}s \\
&  + (-1)^{|\omega |+1} \frac{1}{3} \omega \otimes\tau \otimes  (\tilde{r}_{i} 
 \mathrm{ad}_{r_{i}}) (s) ). 
\end{align*}
But 
\begin{align}
\nonumber& \sum_{i}\alpha_{i} \wedge \nabla^{F}_{X_{i}} \omega = d^{F}\omega ,\  
\sum_{i}\alpha_{i}\wedge \iota_{X_{i}} {\mathcal H} = 3{\mathcal H},\  \sum_{i}\alpha_{i}
\otimes  {\mathcal L}_{X_{i}} \tau    = {\mathcal L} (\tau  ),\\
\label{elem}& \sum_{i}\tilde{r}_{i} \mathrm{ad}_{r_{i}}  =\frac{1}{2} \sum_{i,j,k}
(\mathrm{ad}_{r_{i}} (r_{j}), r_{k})^{\mathcal G} \tilde{r}_{i}\tilde{r}_{j} \tilde{r}_{k}
= 3C, 
\end{align}
where in the first equality (\ref{elem}) we used that $\nabla^{F}$ is torsion-free
and the last equality holds in the Clifford algebra $\mathrm{Cl} (\mathcal G).$  
We obtain that
\begin{align}
\nonumber&\frac{1}{2} \sum_{i} \tilde{e}_{i} \nabla^{\mathbb{S}}_{e_{i}} (\omega 
\otimes\tau \otimes s)  = 
(d^{F}\omega )\otimes \tau \otimes s  +  {\mathcal L} (\tau ) \wedge \omega \otimes s\\
\label{sum-i}& + ( \nabla^{0, {\mathcal S}_{\mathcal G}}s) \wedge \omega \otimes \tau  
 + ( {\mathcal H}\wedge \omega ) \otimes\tau \otimes s   +
\frac{1}{2}  (-1)^{|\omega | +1}  \omega \otimes\tau \otimes {C} s. 
\end{align}
From (\ref{sum-i})  and  the definition of $\slashed{d}$ we obtain 
\begin{align}
\nonumber& \slashed{d} ( \omega \otimes\tau \otimes s)  =
( d^{F}\omega ) \otimes \tau  \otimes s + {\mathcal L} (\tau ) \wedge \omega \otimes s
+(\nabla^{0, {\mathcal S}_{\mathcal G}}s )\wedge\omega \otimes \tau \\
\label{severa-rel}&  + ({\mathcal H}\wedge \omega ) \otimes\tau \otimes s  
+ \frac{1}{2} (-1)^{|\omega | +1} \omega \otimes\tau \otimes C  s
+\frac{1}{4}  T^{\nabla^{E}} (\omega \otimes\tau \otimes s). 
\end{align}
In order to conclude our proof we need to express $T^{\nabla^{E}}$ as a section of $\Lambda^{3}E$ 
and to compute $T^{\nabla^{E}}(\omega \otimes\tau \otimes s)$,
where $T^{\nabla^{E}}\in \Gamma (\Lambda^{3} E)\subset \Gamma (\mathrm{Cl}(E))$ acts 
by Clifford multiplication on $\omega \otimes\tau \otimes s$.  
Using the bases (\ref{bases-dual}), we  write 
\begin{eqnarray*}
 T^{\nabla^{E} } &=&  \frac{1}{6} \sum_{i,j,k}T^{\nabla^{E}} (e_{i}, e_{j},e_{k}) \tilde{e}_{i}\wedge \tilde{e}_{j}\wedge
\tilde{e}_{k}\\  
&=& \frac{4}{3} \sum_{i,j,k}T^{\nabla^{E}} 
 (X_{i}, X_{j}, X_{k}) \alpha_{i}\wedge \alpha_{j}\wedge \alpha_{k}\\ 
  &&+ 2 \sum_{i,j,k}T^{\nabla^{E}} (X_{i}, X_{j}, r_{k}) \alpha_{i}\wedge \alpha_{j}\wedge \tilde{r}_{k} 
+ C,
\end{eqnarray*}
where we used  that $T^{\nabla^{E}}(\alpha_{i}, \cdot , \cdot  )=0$ and $T^{\nabla^{E}}(r_i,r_j,X_k)=0$ from  relation (\ref{torsion-ne}).  
Again from relation (\ref{torsion-ne}), 
$$
 \sum_{i,j,k}T^{\nabla^{E}} (X_{i}, X_{j}, X_{k}) \alpha_{i}\wedge \alpha_{j}\wedge \alpha_{k} = -  \sum_{i,j,k}{\mathcal H}
(X_{i}, X_{j}, X_{k}) \alpha_{i}\wedge \alpha_{j}\wedge \alpha_{k} = - 6{\mathcal H}
$$
and
$$
T^{\nabla^{E}} (X_{i} , X_{j}, r_{k}) \alpha_{i}\wedge \alpha_{j}\wedge \tilde{r}_{k} = - 
(R(X_{i}, X_{j}), r_{k})^{\mathcal G} \alpha_{i}\wedge \alpha_{j}\wedge \tilde{r}_{k}.
$$
We have proven that $T^{\nabla^{E}}$, as a section of $\Lambda^{3}E$, is given by 
$$
T^{\nabla^{E}} = - 8 {\mathcal H} - 2 \sum_{i,j,k}( R(X_{i}, X_{j}), r_{k})^{\mathcal G} \alpha_{i}\wedge \alpha_{j}
\wedge \tilde{r}_{k} + C.
$$
This implies 
\begin{align}
\nonumber  &T^{\nabla^{E}} (  \omega \otimes \tau \otimes s)  = - 8 ( {\mathcal H}
\wedge \omega )\otimes\tau \otimes s \
 + (-1)^{|\omega |} \omega \otimes \tau  \otimes C  s\\
\label{torsion-clifford} & + 2 (-1)^{| \omega | +1}  \sum_{i,j,k}(R(X_{i}, X_{j}), r_{k})^{\mathcal G}
(\alpha_{i}\wedge\alpha_{j} \wedge \omega  )\otimes\tau \otimes {\tilde{r}_{k}}  s  .
\end{align}
We conclude by combining  (\ref{severa-rel}) 
with (\ref{torsion-clifford}). 
\end{proof}

\begin{exa}{\rm Consider a regular Courant algebroid $E$ 
with surjective anchor $\pi : E \rightarrow TM$ such that $\mathrm{ker}\, \pi  = (\mathrm{ker}\, \pi )^{\perp}.$  
A  dissection of $E$ defines an isomorphism between $E$ and the
Courant algebroid $TM\oplus TM^{*}$ from Example \ref{exa-fundam} (see Lemma
2.1 of \cite{AX}). 
From Theorem \ref{thm-alt}, 
the canonical Dirac generating operator acts on $\Gamma (\Lambda T^{*}M)$ 
 by  $\di  (\omega ) = d\omega -{\mathcal H} \wedge \omega$.    
We recover the  expression  of  the Dirac generating operator for  exact Courant algebroids, see e.g.\  
\cite{gualtieri-annals}.}
\end{exa}

\begin{center}
\part[Third Part: Applications]{}\label{lastpart}
\end{center}

\section{Generalized almost Hermitian structures: integrability and spinors}\label{hermitian-spinors}

In Theorem 6.4 of \cite{AX} an integrability criterion for a  generalized almost complex structure $\mathcal J$  on a regular Courant algebroid $E$ with scalar product of neutral signature $(n, n)$,  using the canonical Dirac generating operator $\di :\Gamma (\mathbb{S}) \rightarrow \Gamma (\mathbb{S})$ of $E$
and the pure spinor associated to $\mathcal J$,  was developed. For completeness of our exposition we recall it in the appendix.
As an application of the theory from the previous sections, we now 
characterize the integrability of  a generalized almost
Hermitian structure $(G, \mathcal J )$ on  $E$  in terms of suitably chosen Dirac operators
and the pure spinors associated to $\mathcal J\vert_{E_{\pm}}$, where 
$E = E_{+} \oplus  E_{-}$ is the decomposition  of $E$  determined  by $G$.  
Remark that  $\mathrm{rank}\, E_{+}$ and  $\mathrm{rank}\, E_{-}$ are even  (as $\mathcal J$ preserves $E_{+}$  and $E_{-}$). 
We  consider 
$E_{\pm}$ endowed with the (non-degenerate)  scalar products  $\langle\cdot , \cdot \rangle\vert_{E_{\pm}}$
and we 
denote by $\mathrm{Cl}(E_{\pm})$ the bundle of Clifford algebras over $(E_{\pm}, \langle \cdot , \cdot\rangle\vert_{E_{\pm}}).$ 
We assume 
that $E_{\pm}$ are oriented and 
that $\omega_{\pm}^{2} =1$,  where
$\omega_{\pm}$ are volume forms determined by positive oriented bases of $E_{\pm}$, viewed as elements of
$\mathrm{Cl}(E_{\pm}).$     
We assume that there are given irreducible  
$\mathrm{Cl}(E_{\pm})$-bundles  $S_{\pm}$ with Schur algebra $\mathbb{R}.$ 
The latter condition means that any vector  bundle
morphism $f_{\pm} : S_{\pm} \rightarrow S_{\pm}$ which commutes with the $\mathrm{Cl}(E_{\pm})$-action is a multiple
of the identity.   A quick inspection of Table 1 from \cite{LM}  (see page 29)  implies that either 
$\langle\cdot , \cdot \rangle\vert_{E_{\pm}}$ have  both neutral signature, or 
$\mathrm{dim}\, E_{+} = \mathrm{dim}\, E_{-}$ is a multiple of eight,  and 
one of $\langle \cdot , \cdot \rangle\vert_{E_{\pm}}$ is  
positive definite (while the other is negative definite).
Moreover,  $S_{\pm}$ are $\mathbb{Z}_{2}$-graded with gradation defined  by 
$S_{\pm}^{0}:= \frac{1}{2}\gamma_{ (1+ \omega_{\pm})}S_{\pm}$ and $S_{\pm}^{1} := \frac{1}{2} 
\gamma_{( 1-\omega_{\pm})}S_{\pm}$.

Since $S_{\pm}$ are irreducible  $\mathbb{Z}_{2}$-graded $\mathrm{Cl} (E_{\pm})$-bundles,  $S_{+}\hat{\otimes} S_{-}$ is an irreducible $\mathbb{Z}_{2}$-graded 
$\mathrm{Cl}(E)$-bundle, with Clifford action given by
\begin{equation}\label{clif-gr}
\gamma_{v} (s_{+} {\otimes} s_{-}) = \gamma_{v_{+}}(s_{+}) {\otimes} s_{-} +
(-1)^{|s_{+}|} s_{+}{\otimes}\gamma_{v_{-}}( s_{-}),
\end{equation}
for any $v = v_{+} + v_{-} \in E$ (see appendix for more details). 
Remark that  $\mathcal S = \mathcal S_{+} \hat{\otimes} \mathcal S_{-}$ is the canonical spinor bundle of $S_{+}\hat{\otimes} S_{-}$,
where  
$\mathcal S_{\pm}:= S_{\pm}\otimes  | \mathrm{det}\,  S_{\pm}^{*}|^{\frac{1}{r_{S_{\pm}}}}$
is the canonical spinor bundle of $S_{\pm}$ 
(and $r_{S_{\pm }}$  is the rank of  $S_{\pm }$).

\begin{lem} 
If $\eta_{\pm}\in \Gamma ((\mathcal S_{\pm})_{\mathbb{C}})$ are pure spinors associated to $\mathcal J\vert_{E_{\pm}}$,
then $\eta_{+}{\otimes} \eta_{-} \in \Gamma (\mathcal S_{\mathbb{C}} )$ is a pure spinor associated to $\mathcal J .$ 
\end{lem}

\begin{proof} Let $L$ be the $(1,0)$-bundle of 
$\mathcal J$ and   
\[  (E_{\pm})_{\mathbb{C}} \cap L=  L_{\eta_\pm} := \{ v\in (E_\pm)_\mathbb{C} \mid \gamma_v \eta_\pm =0\}.\]
Both $L_{\eta_{+}} \oplus   L_{\eta_{-}}$ and  $ L_{\eta_{+}  {\otimes}\eta_{-}}$ are subbundles of
$(E_{+})_{\mathbb{C}} \oplus  (E_{-})_{\mathbb{C}} = E_{\mathbb{C}}$.
From (\ref{clif-gr}),  $L_{\eta_{+}} \oplus   L_{\eta_{-}} \subset  L_{\eta_{+}  {\otimes}\eta_{-}}$. 
Since $L_{\eta_{+} {\otimes}\eta_{-}}$  
is an isotropic subbundle of $E_{\mathbb{C}}$, its rank is at most $n$. By comparing ranks  we obtain 
$L_{\eta_{+}} \oplus   L_{\eta_{-}} =   L_{\eta_{+}  \otimes \eta_{-}}$. 
As $L= L_{\eta_{+}} \oplus  L_{\eta_{-}}$ we obtain $L = L_{\eta_{+} {\otimes}\eta_{-}}$ as needed.
\end{proof}

\begin{defn}
The pure spinors $\eta_{\pm}$ from the above lemma are called {\cmssl pure spinors associated to $(G, \mathcal J)$}.  
\end{defn}

Let $D$ be  a generalized Levi-Civita connection of $G$ and   $D^{\pm}$  the  $E$-connections on $E_{\pm}$ induced by $D$. Choose $E$-connections 
$D^{S_{\pm}}$ on $S_{\pm}$ compatible with $D^{\pm}$. 
Like in the proof of Theorem \ref{ex:thm} one can show that  $D^{S_{\pm}}$ exist and   
preserve the grading of $S_{\pm}$, i.e.\ 
$D^{S_{\pm}}_{e} \Gamma (S_{\pm}^{0})\subset \Gamma (S_{\pm}^{0})$ and 
$D^{S_{\pm}}_{e} \Gamma (S_{\pm}^{1})\subset \Gamma (S_{\pm}^{1})$, 
for any $e\in \Gamma (E).$    Since $S_{\pm}$ have Schur algebra $\mathbb{R}$,
any other $E$-connection compatible with $D^{\pm}$ 
is related to $D^{S_{\pm}}$ by  
$\tilde{D}^{S_{\pm}} = {D}^{S_{\pm}} + \lambda^{\pm} \otimes\mathrm{Id}_{S_{\pm}}$, for  
$\lambda^{\pm}\in \Gamma (E^{*})$.
We denote by $D^{\mathcal S_{\pm}}$ the  $E$-connections induced by $D^{S_{\pm}}$ on $\mathcal S_{\pm}$.
Let $D^{\mathcal S} : = D^{\mathcal S_{+}} \otimes D^{\mathcal S_{-}}$ be their tensor product
(independent of gradations),  which is an $E$-connection on 
$\mathcal S = \mathcal S_{+}\hat{\otimes}\mathcal S_{-}$ induced by the $E$-connection $D^{S_{+}}\otimes D^{S_{-}}$
on $S_{+}\hat{\otimes} S_{-}$.
Recall that we use the convention 
$v_{\pm}s_{\pm}$ for the Clifford action $\gamma_{v_{\pm}} s_{\pm}$. Similarly, to simplify notation,  we write 
$v (s_{+} {\otimes} s_{-})$ instead of 
$\gamma_{v} (s_{+} {\otimes} s_{-})$, for any $v\in E$.

\begin{lem} The  $E$-connection $D^{\mathcal S}$  
is compatible with the Clifford action of $\mathrm{Cl}(E)$ on $\mathcal S.$
\end{lem}

\begin{proof}  
Using (\ref{clif-gr}) and that  $D^{\mathcal S_{+}}$ preserves the grading of $\mathcal S_{+}$, 
we obtain
$$
D_{e}^{\mathcal S} (v (s_{+} {\otimes} s_{-})) = 
D_{e}(v) (s_{+}\otimes  s_{-}) + v D^{\mathcal S} _{e} (s_{+} {\otimes}s_{-}),
$$
for any $e, v\in\Gamma ( E)$ and $s_{\pm}\in \Gamma (\mathcal S_{\pm})$, as needed.  
\end{proof}

In the above setting, there are three Dirac operators  to be considered:  two Dirac operators
$\slashed{D}^{\mathcal S_{\pm}}:\Gamma (\mathcal S_{\pm}) \rightarrow \Gamma (\mathcal S_{\pm})$ computed using the $E$-connections $D^{\mathcal S_{\pm}}$ of the $\mathrm{Cl}(E_{\pm})$-bundles $\mathcal S_{\pm}$,  defined by
$$
\slashed{D}^{\mathcal S_{\pm}} (s_{\pm})  := \frac{1}{2} \sum_{i}\tilde{e}_{i}^{\pm} D^{\mathcal S_{\pm}}_{e_{i}^{\pm}}s_{\pm},
$$
where $( e_{i}^{\pm})$ is a basis of $E_{\pm}$ and  
 $( \tilde{e}_{i}^{\pm})$ is the metric dual basis, i.e.\
$\tilde{e}_{i}^{\pm}$ belong to $E_{\pm}$ and 
$\langle e_{i}^{\pm}, \tilde{e}_{j}^{\pm}\rangle =\delta_{ij}.$ 
The third Dirac operator $\slashed{D}^{\mathcal S}$ is the one from Section \ref{sect-d}, computed using the  $E$-connection
$D^{\mathcal S}$ on the $\mathrm{Cl}(E)$-bundle $\mathcal S .$ 
Using that  $ ( e_{i}^{+}, e_{j}^{-})$  and 
$( \tilde{e}_{i}^{+}, \tilde{e}_{j}^{-})$
are  bases of $E$ dual with respect to $\langle\cdot , \cdot \rangle$,  we obtain that 
$$
\slashed{D}^{\mathcal S} (s_{+} {\otimes} s_{-})=  \frac{1}{2} \sum_{i}\tilde{e}_{i}^{+} D^{\mathcal S}_{e_{i}^{+}}(
s_{+} {\otimes} s_{-}) +  \frac{1}{2}  \sum_{i}\tilde{e}_{i}^{-} D^{\mathcal S}_{e_{i}^{-}}(s_{+} {\otimes} s_{-}),
$$ 
for any $s_{\pm }\in \Gamma (\mathcal S_{\pm})$.
The next lemma can be checked  from definitions.

\begin{lem}\label{rel-dirac} The operators $\slashed{D}^{\mathcal S}$, $\slashed{D}^{\mathcal S_{+}} $
and  $\slashed{D}^{\mathcal S_{-}} $ are related by
\begin{align}
\nonumber& \slashed{D}^{\mathcal S}  (s_{+}{\otimes} s_{-})  = ( \slashed{D}^{\mathcal S_{+}} s_{+}) {\otimes} s_{-} + (-1)^{|s_{+}|} s_{+} { \otimes} (\slashed{D}^{\mathcal S_{-}} s_{-})\\
 \label{comp-dirac} & +\frac{1}{2}  \sum_{i}\left( \tilde{e}_{i}^{+} s_{+}  {\otimes} (D^{\mathcal S_{-}}_{e_{i}^{+}} s_{-})
+  (-1)^{|s_{+}|} (D^{\mathcal S_{+}}_{e_{i}^{-}} s_{+}) {\otimes} \tilde{e}_{i}^{-}s_{-}\right) ,
 \end{align}
where $s_{\pm}\in \Gamma (\mathcal S_{\pm}).$ 
\end{lem}

In the next theorem 
we use the notation $D^{\mathcal S_{\pm}}_{e_{\mp}}(\eta_{\pm}) \equiv \eta_{\pm}$ 
(for $e_{\mp}\in E_{\mp}$)  if 
$D^{\mathcal S_{\pm}}_{e_{\mp}}(\eta_{\pm})  = f \eta_{\pm}$
for a function $f = f(e_{\mp})$
which depends on $e_{\mp}$.   By $\slashed{D}^{\mathcal S_{\pm}} \eta_{\pm} \equiv  \eta_{\pm}$ we mean 
$\slashed{D}^{\mathcal S_{\pm}} \eta_{\pm} \in \gamma_{(E_{\pm})_{\mathbb{C}}} (\eta_{\pm})$.

\begin{thm}\label{spinors-gk} In the above setting,  the generalized almost Hermitian structure $(G, \mathcal J)$ on the regular Courant algebroid $E$ with scalar product of neutral signature is generalized 
K\"{a}hler if and only if there is a Levi-Civita connection $D$ of $G$ such that 
\begin{equation}\label{conditions}
\slashed{D}^{\mathcal S_{\pm}} \eta_{\pm}\equiv \eta_{\pm},\ D^{\mathcal S_{\pm}}_{e_{\mp}}\eta_{\pm} \equiv \eta_{\pm}
\end{equation} 
for any  $e_{\mp}\in \Gamma (E_{\mp})$.  Here $E = E_{+} \oplus  E_{-}$ is the decomposition determined by $G$ and  
$\eta_{\pm} \in \Gamma (\mathcal S_{\pm})$ are  pure spinors associated to $(G, \mathcal J).$ 
\end{thm}

\begin{proof} 
Let $D$ be a Levi-Civita connection of $G$.  From  (\ref{inv:eq}), 
$\di = \slashed{D}^{\mathbb{S}} +\frac{1}{4} \gamma_{T} = \slashed{D}^{\mathbb{S}}$ since $D$ is torsion-free
and, using   $D^{\mathbb{S}} = D^{\mathcal S}\otimes D^{L}$,  
$$\slashed{D}^{\mathbb{S}} (s\otimes l) = \slashed{D}^{\mathcal S} (s) \otimes l +\frac{1}{2}\sum_{i}\tilde{e}_{i} s \otimes D^{L}_{e_{i}} l,\
\forall s\in \Gamma (\mathcal S ),\ \forall  l\in \Gamma (L),
$$
where $(e_{i})$ is a basis of $E$ and $(\tilde{e}_{i})$ the dual
basis with respect to $\langle\cdot , \cdot \rangle .$   
We obtain that  a pure spinor $\eta \otimes l$
from $\mathbb{S}=\mathcal S \otimes L$ is projectively closed if and only if
$\slashed{D}^{\mathcal S}\eta \in \gamma_{E_{\mathbb{C}}} \eta$.\

Assume now that relations (\ref{conditions}) hold, 
with $\slashed{D}^{\mathcal S_{\pm}}$ and $D^{\mathcal S_{\pm}}$ computed starting with  $D$. From (\ref{comp-dirac}),
we deduce that  the pure spinor 
$\eta = \eta_{+}{\otimes}\eta_{-}$ associated to $\mathcal J$ satisfies  
$\slashed{D}^{\mathcal S} \eta \in \gamma_{E_{\mathbb{C}}}\eta $, i.e.\ 
$\mathcal J$ is integrable (see Corollary \ref{proj-closed-thm}). 
In a similar way, we show  that  
$\mathcal J_{2} =G^{\mathrm{end}}\mathcal J$   is integrable. For this,
we  use the fact   $\eta_{+}{\otimes}\bar{\eta}_{-}$ is a pure spinor associated to $\mathcal J_{2}$ 
and $\slashed{D}^{\mathcal S_{-}} \bar{\eta}_{-}  \in 
\gamma_{(E_{-})_{\mathbb{C}}}\bar{\eta}_{-}$
(because $\slashed{D}^{\mathcal S_{-}}: \Gamma ((\mathcal {S}_{-})_{\mathbb{C}}) \rightarrow \Gamma ((\mathcal{S}_{-})_{\mathbb{C}})$ 
 is the complex linear extension of its restriction to $\mathcal S_{-}$ and hence commutes with the natural conjugation of $(\mathcal S_{-})_{\mathbb{C}}$).
We obtain that $(G, \mathcal J)$ is generalized K\"{a}hler. 

Conversely, assume  now that $\mathcal J$ is integrable
and let $D$ be a Levi-Civita connection of $G$,  with  $D \mathcal J =0$ (which exists from
Theorem \ref{integr-kahler}). 
The relation 
\begin{equation}\label{a-r}
D^{\mathcal S_{+}}_{e} (v \eta_{+}) = (D^{+}_{e} v)\eta_{+} + v D^{\mathcal S_{+}}_{e} \eta_{+},\ e\in \Gamma (E), 
\end{equation}
together with the fact that $D$ preserves $L_{\eta_{+}} = L\cap (E_{+})_{\mathbb{C}}$ imply 
that 
$v D^{\mathcal S_{+}}_{e} \eta_{+} =0$, for any $v\in \Gamma (L_{\eta_{+}})$, i.e. 
\begin{equation}\label{any}
D^{\mathcal S_{+} }_{e} \eta_{+} = \lambda  (e)\eta_{+},\ \forall e\in \Gamma (E),  
\end{equation}
for $\lambda \in \Gamma (E^{*}).$ 
Relation (\ref{any}) with $e:= e_{-}\in \Gamma (E_{-})$  implies that   $D^{\mathcal S_{+} }_{e_{-}} \eta_{+} \equiv\eta_{+}$. 
On the other hand,  letting  $e:= e_{i}^{+}$ in (\ref{any}), applying the Clifford action of 
$\tilde{e}^{+}_{i}$  and summing over $i$ we obtain 
that  $\slashed{D}^{\mathcal S_{+}} \eta_{+} 
= \lambda_{+}\eta_{+}$ where $\lambda_{+}:=\lambda\vert_{E_{+}}$ is a section of $E_{+}^{*}\cong E_{+}$
(identified using $\langle\cdot , \cdot \rangle\vert_{E_{+}}$)  
and $\lambda_{+} \eta_{+}=\gamma_{\lambda_{+}}\eta_{+} $ is the Clifford action of $\lambda_{+}\in E_{+}$ on $\eta_{+}$. 
We proved  $\slashed{D}^{\mathcal S_{+}} \eta_{+}\equiv \eta_{+}$ as needed. 
The same argument with $\mathcal S_{+}$ and $\mathcal S_{-}$ interchanged 
shows that all relations (\ref{conditions}) are satisfied. 
\end{proof}

\begin{lem}\label{indep-cond} 
Relations (\ref{conditions}) are independent of the choice of Levi-Civita connection.
\end{lem}

\begin{proof} Let $D$ be a Levi-Civita connection. Since $D$ is torsion-free and preserves $E_{\pm}$, 
for any $e_{-}\in\Gamma (E_{-})$ and $v_{+},w_{+}\in \Gamma (E_{+})$,  
\begin{align*}
\nonumber& 0 = T^{D}(e_{-},v_{+},w_{+}) = \langle D_{e_{-}}v_{+} - D_{v_{+}}e_{-} - [e_{-}, v_{+}], w_{+}
\rangle + \langle v_{+}, D_{w_{+}} e_{-}\rangle\\
\nonumber& =  \langle  D_{e_{-}} v_{+} - [e_{-}, v_{+}], w_{+}\rangle,
\end{align*}
which implies  $D_{e_{-}} v_{+}= [e_{-}, v_{+}]_{+}$ (see also Lemma 3.2 of \cite{garcia}).  In particular, $D^{+}_{e_{-}} = D_{e_-}|_{\Gamma (E_+)}$ is independent of the choice of
Levi-Civita connection $D$, for any $e_{-} \in \Gamma (E_{-}).$ 
We obtain that any two $E$-connections 
$D^{\mathcal S_{+}}$ and $\tilde{D}^{{\mathcal S}_{+}}$  
on $\mathcal S_{+}$, compatible with any two  Levi-Civita connections of $G$, 
satisfy   $\tilde{D}^{\mathcal S_{+}}_{e_{-}} = D^{\mathcal S_{+}}_{e_{-}} + \lambda_{-}(e_{-}) \mathrm{Id}_{\mathcal S_{+}}$, for $\lambda_{-}\in E_{-}^{*}.$    
This  implies that the condition 
$D^{\mathcal S_{+}}_{e_{-}} \eta_{+}\equiv \eta_{+}$  is independent of the choice of $D$.
In a similar way we prove the statement for $D^{\mathcal S_{-}}_{e_{+}} \eta_{-}\equiv \eta_{-}$.\

Next, consider two Levi-Civita connections $D$ and $\tilde{D} = D+A$ of $G$.
The arguments from Propositions  \ref{trafo:prop} and   \ref{indep0:Prop} 
 show that
$D^{\mathcal S_{+}}$ (hence, also $\slashed{D}^{\mathcal S_{+}}$) depends only
on $D^{+}$ and 
\begin{equation}\label{role-divergence}
\tilde{\slashed{D}}^{\mathcal S_{+}} = \slashed{D}^{\mathcal S_{+}}-\frac{1}{4} \gamma_{\alpha^{+}} -\frac{1}{4} \gamma_{v_{A^{+}}}
\end{equation}
where $\alpha^{+}\in \Gamma (\Lambda^{3} E_{+}^{*})$ is given by  
$\alpha^{+}(u, v, w) := \sum_{(u,v,w)\;  \mathrm{cyclic}} \langle A_{u}v, w\rangle ,\ u, v, w\in E_{+}$
and $v_{A^{+}} := \sum_{i=1}^{n} A_{e^{+}_{i}} \tilde{e}^{+}_{i} \in \Gamma (E_{+})$,
 where $( e_{i}^{+})$ and $ (\tilde{e}_{i}^{+})$ are $\langle \cdot , \cdot \rangle$-dual bases of
$E_{+}.$  As $D$ and $\tilde{D}$ are torsion-free, $\alpha^{+}=0$ and we obtain that 
$$
\tilde{\slashed{D}}^{\mathcal S_{+}} \eta_{+}= \slashed{D}^{\mathcal S_{+}}\eta_{+}-\frac{1}{4} v_{A^{+}}\eta_{+}.
$$ 
This  implies that  the condition $\slashed{D}^{\mathcal S_{+}}\eta_{+}\equiv \eta_{+}$ is independent of the choice of $D$.
In a similar way we prove the statement for $\slashed{D}^{\mathcal S_{-}}\eta_{-}\equiv \eta_{-}$.
\end{proof}

\begin{rem}{\rm  The Dirac  operators 
$\slashed{D}^{\mathcal S_{\pm}}$ are in fact  independent of the Levi-Civita connection of $G$, as long as we fix
the divergence of $D$. The  statement for $\slashed{D}^{\mathcal S_{+}}$
follows from relation  (\ref{role-divergence}), by noticing that if $\mathrm{div}_{D} = \mathrm{div}_{\tilde{D}}$ then
$v_{A^{+}} =0$ (a similar argument holds for 
 $\slashed{D}^{\mathcal S_{-}}$).}
\end{rem}

From Theorem \ref{spinors-gk} combined with Lemma 
\ref{indep-cond} we  obtain the following 
characterization for the integrability of generalized almost hyper-Hermitan structures.

\begin{cor} \label{hyperKCor}   Let $E$ be a regular Courant algebroid with scalar product $\langle \cdot , \cdot \rangle $ of neutral signature.
Let   $E = E_{+}\oplus E_{-}$ be the decomposition of $E$ determined
by a generalized metric $G$. Assume that   $\langle \cdot , \cdot \rangle\vert_{E_{\pm}}$ are either both of neutral signature,
or  $\mathrm{rank}\, E_{+} = \mathrm{rank}\, E_{-}$ is a multiple of eight
and one of $\langle \cdot , \cdot \rangle\vert_{E_{\pm}}$ is positive definite (and the other negative
definite). A generalized almost hyper-Hermitian structure $(G, \mathcal J_{1}, \mathcal J_{2}, \mathcal J_{3})$ 
is generalized hyper-K\"{a}hler if and only if conditions
(\ref{conditions}) from Theorem \ref{spinors-gk} hold 
for a Levi-Civita connection  $D$ of $G$ and 
each of the pure spinors $\eta^{i}_{\pm}$ associated to $(G, \mathcal J_{i})$, $i=1,2,3$. The conditions are independent of the choice of $D$. 
\end{cor}

\section{Appendix }\label{complex-spinors}

\subsection{$\mathbb{Z}_{2}$-graded algebras and Clifford algebras}

Recall that if  $A= A^{0}\oplus  A^{1}$ and $B= B^{0}\oplus  B^{1}$
are $\mathbb{Z}_{2}$-graded vector spaces, then the tensor product $A\otimes B$ inherits a $\mathbb{Z}_{2}$-gradation
$$
(A{\otimes} B)^{0}:= A^{0}{\otimes} B^{0} + A^{1}{\otimes}
B^{1},\  (A{\otimes} B)^{1}:= A^{0}{\otimes} B^{1} + A^{1}{\otimes} B^{0}.
$$ 
We denote by $A\hat{\otimes} B$ the vector space $A\otimes B$ together with this gradation.
If, moreover, $A$ and $B$ are $\mathbb{Z}_{2}$-graded algebras, then $A\hat{\otimes} B$ inherits the structure of a
$\mathbb{Z}_{2}$-graded algebra with multiplication on homogeneous elements defined by 
$$
(a\otimes b) (\tilde{a}\otimes \tilde{b}) := (-1)^{|b| |\tilde{a}|} a\tilde{a}\otimes b\tilde{b},
$$
where $|a|, |\tilde{b}|\in\{ 0,1\}$ are the degrees  of $a$ and $\tilde{b}.$

We say that a $\mathbb{Z}_{2}$-graded vector space $S= S^{0} \oplus S^{1}$ is a $\mathbb{Z}_{2}$-graded $A$-module if it is a representation space for $A$ and the action of $A$ on $S$ is compatible with
gradations,  i.e.  $A^{i} \cdot S^{j} \subset S^{i+j}$ for any $i, j\in \mathbb{Z}_{2}.$

Finally, if $S$ and $S'$ are $\mathbb{Z}_{2}$-graded $A$- and $B$-modules respectively, then their graded tensor product
 $S\hat{\otimes} S'$ is a $\mathbb{Z}_{2}$-graded
$A\hat{\otimes} B$-module with action 
given by
$$
\gamma_{a {\otimes}b} (s{\otimes} s' ) := (-1)^{|b||s|} a(s) {\otimes} b(s'),
$$
where   $s\in S$, $s'\in S'$, $a\in A$, $b\in B$,   $|b|:= \mathrm{deg}(b)$,  
$|s|:= \mathrm{deg}(s)$  are  the degrees of the homogeneous elements  $b$ and $s$.\

We   apply these facts to Clifford algebras and their representations. 
Assume  that $(V_{+}, q_{+})$ and $(V_{-}, q_{-})$ are two vector spaces with scalar products and let
$(V:= V_{+} \oplus V_{-}, q:= q_{+} + q_{-})$ be  their direct sum.  As $\mathrm{Cl}(V_{\pm})$ are $\mathbb{Z}_{2}$-graded algebras 
we can consider 
$\mathrm{Cl}(V_{+})\hat{\otimes}\mathrm{Cl}(V_{-})$  which is   a $\mathbb{Z}_{2}$-graded algebra, and as such is  isomorphic 
to $\mathrm{Cl}(V)$ (for the latter statement see e.g.\ Chapter I of  \cite{LM}).
The isomorphism  between $\mathrm{Cl}(V)$ and $\mathrm{Cl}(V_{+})\hat{\otimes}\mathrm{Cl}(V_{-})$ 
is obtained by 
extending the map
$V\rightarrow \mathrm{Cl}(V_{+}) \hat{\otimes}\mathrm{Cl}( V_{-})$
which  assigns to  any $v= v_{+}+ v_{-}\in V_{+} \oplus V_{-}$ the vector   $v_{+}{\otimes}1  +1{\otimes}v_{2}$.\

Let $S_{\pm}$ be   $\mathbb{Z}_{2}$-graded $\mathrm{Cl}(V_{\pm})$-modules.
From above,  the graded tensor product 
$S_{+}\hat{\otimes} S_{-}$ is a  $\mathbb{Z}_{2}$-graded 
$\mathrm{Cl}(V_{+})\hat{\otimes} \mathrm{Cl}(V_{-})$-module, hence also a 
$\mathbb{Z}_{2}$-graded   $\mathrm{Cl}(V)$-bundle. 
Any $v= v_{+}+ v_{-}\in V_{+} \oplus  V_{-} \subset \mathrm{Cl}(V)$  acts on $S_{+} \hat{\otimes} S_{-}$ as
\begin{equation}\label{gen-graded-spinor}
\gamma_{v_{+} + v_{-}} (s_{+} {\otimes} s_{-}) = \gamma_{v_{+}}(s_{+}) {\otimes} s_{-} +
(-1)^{|s_{+}|} s_{+}{\otimes}\gamma_{v_{-}}( s_{-}).
\end{equation}

\subsection{Integrability of generalized almost complex structures and  spinors}

Let $E$ be a regular Courant algebroid  with scalar product of signature $(n, n)$,
anchor $\pi : E \rightarrow TM$ and canonical Dirac generating operator $\di : \Gamma (\mathbb{S}) \rightarrow
 \Gamma (\mathbb{S}).$ 
An  {\cmssl almost  Dirac structure} of $E_{\mathbb{C}}$ is an  isotropic 
complex subbundle of 
$E_{\mathbb{C}}$ of rank $n$. It  is  {\cmssl integrable} (or a {\cmssl Dirac structure})
if is is closed under the (complex linear extension of)  the Dorfman bracket of $E.$ For a  non-vanishing section  $\eta \in \Gamma (\mathbb{S}_{\mathbb{C}})$ we define
$$
L_{\eta} :=  \{ v\in E_{\mathbb{C}} \mid  \gamma_{v} \eta  =0\} .
$$
The spinor $\eta$ is called {\cmssl pure} if 
$L_{\eta}$ is a vector bundle of rank $n$.    
Assume that $\eta$ is a pure spinor.
A simple computation shows that $L_{\eta}$ is  isotropic, and,  being of rank $n$,  $L_{\eta}$  is  an almost Dirac structure. 
It is called the {\cmssl null bundle} of $\eta .$   The assignment $L_{\eta}\rightarrow [\eta ]$ is a one-to-one correspondence 
between almost Dirac structures of $E_{\mathbb{C}}$ and classes of projectively equivalent  pure spinors of $\mathbb{S}$
(two pure spinors $\eta_{1}$ and $\eta_{2}$ defined on an open set $U\subset M$ are projectively equivalent if $\eta_{2} = f\eta_{1}$ for a non-vanishing
function $f$ on $U$).  A  pure spinor
$\eta \in \Gamma (\mathbb{S}_{\mathbb{C}})$  is called {\cmssl projectively closed} if   
there is  $v\in \Gamma (E_{\mathbb{C}})$ such that 
$\di  (\eta) = \gamma_{v} \eta $.   (In order to simplify notation, we use the same symbols $\di$ and $\gamma$ for their complex  linear extensions).  
Note  that any pure spinor which is projectively equivalent to a projectively closed spinor is also projectively closed. 
This follows from $\di (f \eta ) =\gamma_{\pi^{*}(df)}(\eta ) + f\di (\eta )$, for any $\eta \in \Gamma (S_{\mathbb{C}})$ and $f\in C^{\infty}(M, \mathbb{C})$.

\begin{thm}(\cite{AX}) \label{proj-closed-main} 
An almost Dirac structure $L$  of $E_{\mathbb{C}}$  is a Dirac structure if and only if, locally, any  pure spinor $\eta$
associated to $L$ is projectively closed.  
\end{thm}

\begin{proof}   
Assume that $\eta$ is projectively closed and let $e\in \Gamma (E_{\mathbb{C}})$ such that
$\di  (\eta ) = \gamma_{e} \eta .$ Let $v, w\in \Gamma (L) $.    Using condition ii) from 
Definition  \ref{Dgo:def} and $\gamma_{v}\eta  =\gamma_{w} \eta  =0$, we obtain 
\begin{equation}\label{e1}
\gamma_{[v,w]} \eta = [ [ \di, \gamma_{v}],\gamma_{w} ]]\eta 
= -\gamma_{w}\gamma_{v} \di  (\eta )  
= - \gamma_{w}\gamma_{v} \gamma_{e}\eta .
\end{equation}
On the other hand,
\begin{equation}\label{e2}
\gamma_{w}\gamma_{v} \gamma_{e} = -
\gamma_{w}\gamma_{e}\gamma_{v} + 2\langle v,e\rangle \gamma_{w}.
\end{equation}
Combining (\ref{e1}) with (\ref{e2})  and using  $\gamma_{v} \eta =\gamma_{w} \eta  =0$, we obtain
$\gamma_{[v,w]} \eta = 0$. This proves that $L$ is a Dirac structure.

Conversely, assume that $L$ is a Dirac structure. Then, for any $v, w\in \Gamma (L)$, $[v, w]\in \Gamma (L)$ 
and $\gamma_{[v,w]} \eta =0$. This implies, using  condition ii) from 
Definition \ref{Dgo:def},
$[[ \di, \gamma_{v}], \gamma_{w} ]\eta =0$, or
$\gamma_{w}[\di, \gamma_{v} ]\eta = 0,\ \forall w\in \Gamma (L).$ 
We  obtain that  $[\di, \gamma_{v} ]\eta$, which is equal to  $\gamma_{v}
\di  (\eta ) $,  is  a multiple of  
$\eta$, i.e. $\gamma_{v} \di   (\eta ) = \lambda (v)\eta$
for  $\lambda (v)\in C^{\infty}(M, \mathbb{C})$. 
Remark that $\lambda \in \Gamma (L^{*}).$ 
Extend $\lambda$ to a (complex linear) $1$-form on $E_{\mathbb{C}}$ and let $v_{0}\in \Gamma (E_{\mathbb{C}})$, such
that $2v_{0}$ is dual to this $1$-form with respect to the complex linear extension of $\langle\cdot , \cdot \rangle .$  Then
$$
\lambda (v)\eta = 2\langle v, v_{0}\rangle \eta = \gamma_{v}\gamma_{v_{0}} \eta + \gamma_{v_{0}}\gamma_{v}\eta
= \gamma_{v} \gamma_{v_{0}}\eta .
$$
The above computations show that $\gamma_{v}( \di (\eta ) - \gamma_{v_{0}}\eta ) =0$, for any $v\in \Gamma (L)$, which 
implies $ \di (\eta ) - \gamma_{v_{0}} \eta  = g\eta$ for $g\in C^{\infty}(M, \mathbb{C}).$ 
As $\di$ and $\gamma_{v_{0}}$ are odd operators and pure spinors are chiral, i.e.\ either even or odd, we conclude $g=0$.  This shows that 
$\di (\eta )  =  \gamma_{v_{0}} \eta $, i.e.\ $\eta$ is projectively closed. 
\end{proof}

Let $\mathcal J$ be a generalized almost complex structure on $E$. 
The $(1,0)$-bundle  $L\subset E_{\mathbb{C}}$ of  $\mathcal J$  is isotropic with respect to $\langle\cdot , \cdot \rangle$
and satisfies $L \oplus   \bar{L} = E_{\mathbb{C}}.$  In particular, $\mathrm{rank}\, L  = n$ 
and $L$ is an almost Dirac structure.  A pure spinor $\eta\in\Gamma (\mathbb{S})$ is called {\cmssl associated to $\mathcal J$}  if $L = L_{\eta}.$ From Theorem \ref{proj-closed-main} we obtain:

\begin{cor}\label{proj-closed-thm} A generalized almost complex structure $\mathcal J$ on a regular Courant algebroid
is integrable if and only if,  locally, one (equivalently, any)  pure spinor 
 associated to $\mathcal J$  is projectively closed. 
\end{cor}

{\bf Acknowledgements.}  We are grateful to 
Mario Garc\'ia-Fern\'andez,  for explaining to us his characterization of generalized K\"{a}hler
structures  on  exact Courant algebroids \cite{mario-priv}, which we have generalized to regular Courant algebroids in Theorem \ref{spinors-gk}. 
We are also grateful to Thomas Mohaupt for  discussions about Born geometry and for pointing out the reference \cite{FRS}.
We thank Paul Gauduchon for sending us a copy of his paper  \cite{gauduchon},  Carlos Shahbazi for drawing our attention to reference \cite{b-h}, 
Roberto Rubio for drawing our attention to references \cite{rubio,u} and Thomas Leistner for useful comments. 
Research of V.C.\  was partially funded by the Deutsche Forschungsgemeinschaft (DFG, German Research Foundation) under Germany's Excellence Strategy -- EXC 2121 Quantum Universe  
-- 390833306.
L.D.\ was supported by a grant of the Ministry of Research and Innovation, project no PN-III-ID-P4-PCE-2016-0019 within PNCDI. \\

V.\  Cort\'es: vicente.cortes@math.uni-hamburg.de\

Department of Mathematics and Center for Mathematical Physics, University of Hamburg,  Bundesstrasse 55, D-20146, Hamburg, Germany.\\

L.\  David: liana.david@imar.ro\

Institute of Mathematics  Simion Stoilow of the Romanian Academy,   Calea Grivitei no. 21,  Sector 1, 010702, Bucharest, Romania.

\end{document}